\documentclass{SCIYA2017enOL}

\newcommand{\EX}{\mathbb{E}}

\newcommand{\pp}[2]{\frac{\partial #1}{\partial #2}}
\newcommand{\dd}[2]{\frac{\d #1}{\d #2}}
\renewcommand{\d}{\ensuremath{\,\mathrm{d}}}
\newcommand{\laplace}{\Delta}
\renewcommand{\vec}[1]{\ensuremath{\boldsymbol{#1}}}
\renewcommand{\abs}[1]{\left| #1 \right|}
\newcommand{\norm}[1]{\left\|{#1}\right\|}
\newcommand{\inp}[2]{\left\langle #1 , #2\right\rangle}
\newcommand{\ssum}[2]{\sum\limits_{#1}^{#2}}
\newcommand{\T}{\mathrm{T}}
\newcommand{\hS}[1]{\hat{S}_{#1}}

\newcommand{\FW}{{\rm FW}}
\newcommand{\OM}{{\rm OM}}
\newcommand{\me}{{\rm e}}

\renewcommand{\phi}{\varphi}
\renewcommand{\epsilon}{\varepsilon}

\DeclareMathOperator*{\argmax}{argmax}

\online
\begin{document}
\ensubject{fdsfd}

\ArticleType{ARTICLES}
\Year{2017}
\Month{January}%
\Vol{60}
\No{1}
\BeginPage{1} %
\DOI{10.1007/s11425-000-0000-0}
\ReceiveDate{January 1, 2017}
\AcceptDate{January 1, 2017}

\title[]{The Graph Limit of The Minimizer of The Onsager-Machlup Functional and Its Computation}
{The Graph Limit of The Minimizer of The Onsager-Machlup Functional and Its Computation}

\author[1]{Qiang Du}{qd2125@columbia.edu}
\author[2,$\ast$]{Tiejun Li}{tieli@pku.edu.cn}
\author[3]{Xiaoguang Li}{lixiaoguang@hunnu.edu.cn}
\author[4]{Weiqing Ren}{matrw@nus.edu.sg}



\address[1]{Department of Applied Physics and Applied Mathematics, Columbia University, New York, NY {\rm 10027}, USA}
\address[2]{LMAM and School of Mathematical Sciences, Peking University, Beijing {\rm 100871}, P.R.China}
\address[3]{Beijing Computational Science Research Center, Beijing, {\rm 100193}, China, \\MOE-LCSM, School of Mathematics and Statistics, Hunan Normal University, Changsha, Hunan  {\rm 410081} , P. R. China}
\address[4]{Department of Mathematics, National University of Singapore, Singapore {\rm 119076}, Singapore}

\abstract{The Onsager-Machlup (OM) functional is well-known  for characterizing
the most probable transition path of a diffusion process with non-vanishing noise.
However, it suffers from a notorious issue that the functional is unbounded below
when the specified transition time $T$ goes to infinity. This hinders
the interpretation of the results obtained by minimizing the OM functional.
We provide a new perspective on this issue. Under mild conditions,
we  show that although the infimum of the OM functional becomes
unbounded when $T$ goes to infinity, the sequence of minimizers
does contain convergent subsequences on the space of curves. The graph limit of this
 minimizing subsequence is an extremal
of the abbreviated action functional, which is related to the OM functional
via the Maupertuis principle with  an optimal energy. We further propose an
 energy-climbing geometric minimization algorithm (EGMA) which identifies the optimal energy
and the graph limit of the transition path simultaneously.
This algorithm is successfully applied to several typical examples in rare event studies.
Some interesting comparisons with the Freidlin-Wentzell action functional are also made.}



\maketitle

\section{Introduction}

Consider a stochastic dynamics modeled by the stochastic differential equation (SDE)
\begin{equation}\label{eq:diff}
  \d \vec{X}_t = \vec{b}(\vec{X}_t)\d t + \sqrt{2\epsilon}\d \vec{W}_t,
\end{equation}
where $\vec{X}_t, \vec{b}\in\mathbb{R}^d$,  $\boldsymbol{W}_t=(W^1_t,W^2_t,\ldots,W^d_t)$ is the standard $d$-dimensional Wiener process with $\EX W^i_t=0$ and $\EX (W^i_tW^j_s)=\delta_{ij}\cdot\min(t,s)$ for $i,j=1,\ldots,d$ and $t,s\in \mathbb{R}^+$.
Due to the presence of the noise, the system makes transitions from one metastable state
to another; when $\epsilon$ is small, however, these transitions happen on a time scale
which is much longer than the relaxation time scale of the system. These rare but important transition events are very common in different fields
of science, and their study has attracted much attention in recent years
\cite{Bovier1, Bovier2, CameronJNS, EVReview,RMP,SchuttBook}.
One major object in the study of such rare events is to understand the transition mechanism, which
can be characterized by the most probable  path (MPP), i.e. the transition path with a
dominant probability,
connecting an initial state $\vec{x}_s$  and a terminal state $\vec{x}_f$
in the configuration space.
How to characterize and compute these transition paths is a fundamental problem in the study of rare events \cite{durr1978onsager, weinan2002string, weinan2006towards, heymann2008geometric, pinski2010transition, zhou2008adaptive, Xiaoliang2011,zhang2016recent} and also the focus of this paper.
In particular, we study the transition path at finite noise provided
by the Onsager-Machlup functional and its graph limit  as the prescribed transition time goes to infinity.

In the zero noise limit, i.e. $\epsilon\to 0$, it is well-known from
the large deviation theory that the MPP from $\vec{x}_s$ to $\vec{x}_f$
is given by the solution to the double minimization
problem \cite{Weinan2004MAM, heymann2008geometric}
\begin{equation}\label{eq:doubleinfFW}
  S^{\FW}(\vec{x}_s,\vec{x}_f ) = \inf\limits_{T>0}\inf\limits_{\psi(0)=\vec{x}_s, \psi(T)=\vec{x}_f} S_T^{\FW}[\psi],
\end{equation}
where $\psi(t)$ is an absolutely continuous function on $[0,T]$ and
$ S_T^{\FW}[\psi]$ is the Freidlin-Wentzell (FW) action functional
\begin{equation}\label{eq:fw}
  S_T^{\FW}[\psi] = \int_{0}^{T} \frac{1}{2}|\dot{\psi}-\vec{b}(\psi)|^2 \d t.
\end{equation}
Here $\dot{\psi}$ denotes the time derivative of $\psi$.
It is known  that the infimum is achieved when $T=\infty$
if  the MPP connecting $\vec{x}_s $ and $\vec{x}_f$ passes through a stationary point of the deterministic dynamics $\dot{\vec{x}} = \vec{b}(\vec{x})$.
The quasi-potential $S^{\FW}(\vec{x}_s ,\vec{x}_f)=\inf_T\inf_{\psi} S_T^{\FW}[\psi]$ characterizes,
in the zero-noise limit, the transition rate, the invariant distribution of the
stochastic dynamics and so forth. To avoid the difficulty caused by the infinite transition time,
an alternative approach, which uses the arclength parameterization, has been proposed
for the computation of the MPP on the space of curves \cite{heymann2008geometric}.

For finite $\epsilon>0$, the Onsager-Machlup (OM) functional
has been proposed in the literature to find the MPP
\cite{dunlop2015map, fujisaki2010onsager, fujisaki2013multiscale, lee2017finding, pinski2010transition, wang2010kinetic}.
The OM functional is given by
\begin{equation}\label{eq:OM}
  S_T^{\OM}[\psi]=\int_{0}^{T} L(\psi,\dot{\psi})\d t
\end{equation}
if $\psi$ is absolutely continuous on $[0,T]$, and takes the infinite value otherwise.
Here the OM Lagrangian is given by
\begin{equation}\label{eq:lagrangian}
  L(\vec{x},\vec{y})=\frac{1}{2}\abs{\vec{y}-\vec{b}(\vec{x})}^2 + \epsilon\nabla\cdot\vec{b}(\vec{x}) = \frac{1}{2}\abs{\vec{y}}^2 - \vec{b}(\vec{x})\cdot\vec{y} - U(\vec{x}),
\end{equation}
where $U(x)$, the so-called path potential \cite{pinski2010transition}, is given by
\begin{equation}\label{eq:U}
  U(\vec{x}) = -\epsilon\nabla\cdot\vec{b}(\vec{x}) - \frac{1}{2}\abs{\vec{b}(\vec{x})}^2.
\end{equation}
The OM functional was first introduced by Onsager and Machlup for
SDEs with linear drift $\vec{b}(\vec{x})=-\gamma\vec{x}$ and constant diffusion
by means of path integrals \cite{onsager1953fluctuations}.
Indeed, their original formulation does not
contain the term $\epsilon \nabla\cdot \vec{b}(\vec{x})$ in \eqref{eq:lagrangian}
as this is only a constant in the case of linear drift.
It was later generalized to cases with nonlinear drift and
non-constant diffusion terms based on either physical arguments
\cite{Graham1977} or more rigorous mathematical derivations
\cite{Fujita1982, IkedaWatanabe1980, takahashi1981probability, Zeitouni1989onsager}. It was also argued in \cite{durr1978onsager} that,
in the scalar case with a constant diffusion,  the MPP
 is given by the minimizer of the OM functional.
Indeed, as shown in
\cite{Fujita1982, IkedaWatanabe1980, takahashi1981probability},
the OM functional arises from
the limiting probability of a $\delta$-ball problem, i.e.
\begin{equation} \label{eq:OM-DeltaBall1}
\mathbb{P}\left(\sup_{0\le t\le T}|\vec{X}_t-\psi(t)|\le \delta\right)\sim C\exp\left(-\frac{2\epsilon \lambda T}{\delta^2}\right)\exp\left(-\frac{1}{2\epsilon}S^{\OM}_T[\psi]\right)
\end{equation}
as $\delta\rightarrow 0+$. Here $\epsilon>0$ is fixed,  $C$ is a positive constant and $\lambda$ is the leading eigenvalue of the operator $-\frac12\Delta$ with zero boundary condition on the domain $|\vec{x}|\le 1$. The first exponential term on the right hand side of \eqref{eq:OM-DeltaBall1} can be also identified as the probability $\mathbb{P}(\sup_{0\le t\le T}|\sqrt{2\epsilon}\vec{W}_t|\le \delta)$. Equation~\eqref{eq:OM-DeltaBall1} shows that the minimizer
of the OM functional characterizes the MPP under a suitable rescaling as $\delta\rightarrow 0$  while $\epsilon>0$ is fixed.
This methodology has been used to find the MPP in many practical problems,
such as in the study of transition pathways of phage-$\lambda$
switching \cite{wang2010kinetic}, conformation changes of polymer systems \cite{fujisaki2010onsager, fujisaki2013multiscale}, protein folding pathways \cite{lee2017finding},
and the maximum posteriori estimator \cite{dunlop2015map}, etc.

Although both the FW and OM functionals  are widely used in applications,
the connection between them is not fully understood. In practice, making the choice between the FW and OM functionals is often still a dilemma.  The minimization of the two functionals gives different results and the minimizer of the OM functional is even dependent on the choice of the transition time $T$. Understanding the relations between the FW and OM functionals is thus an interesting mathematical problem. One formal connection between these two is that the FW functional can be obtained
from the OM functional by simply setting $\varepsilon=0$, therefore,
they are likely equivalent in the zero-noise limit under certain conditions.
Another formal connection can be made via the path integral approach. With the Girsanov theorem, we have
\begin{equation}\label{eq:Girsanov}
\EX F[\vec{X}_t]=\EX \Big(F[\vec{W}_t]
\exp\big(-\frac12\int_0^T |\vec{b}(\vec{W}_t)|^2\d t +\int_0^T\vec{b}(\vec{W}_t)\d \vec{W}_t\big)\Big)
\end{equation}
for Brownian functional $F[\vec{X}]$, where we take $\varepsilon=1/2$ for simplicity and  $\vec{X}_t$
is the solution of the SDE \eqref{eq:diff}.  In the path integral, we formally represent the Wiener measure using the density $p(\vec{x})=Z^{-1} \exp\big(-\frac12 \int_0^T |\dot{\vec{x}}|^2\d t\big)$.
Then the path weight for $\{\vec{X}_t\}_{t\in [0,T]}$ will be given by
the exponential of the FW or OM functional respectively, depending on whether
 Ito or Stratonovich version of the stochastic integral  $\int_0^T\vec{b}(\vec{W}_t)\d \vec{W}_t$ being
interpreted  as $\int_0^T\vec{b}(\vec{x})\cdot \dot{\vec{x}} \d t$ in path integrals. There have been some  investigations on the relationship between the FW and OM functionals.
For example, in \cite{pinski2012gamma} it was shown
that the OM functional $\Gamma$-converges to a functional
completely characterized by the FW functional when $T=\varepsilon^{-1}$ and
$\varepsilon\rightarrow 0$; in \cite{pinski2010transition}, the numerical studies
showed that in general the OM functional does not have a lower bound
as $T\rightarrow \infty$, and the minimizer of the OM functional with $\varepsilon>0$ exhibits
quite different behavior compared to that of the FW functional.

Accepting the assertion that the minimizer of OM functional characterizes
the MPP for the SDEs \eqref{eq:diff} when $\varepsilon$ is finite,
in this paper, we analyze the behaviour of the minimizer as the prescribed transition
time $T$ goes to infinity. This is meaningful since
the transition time between metastable states is usually exponentially large
in $O(1/\epsilon)$ as suggested by the Arrhenius law.
In particular, we will focus on the graph limit of the OM minimizers
on the space of curves.  As we will show, the minimization problem with $T=\infty$ is not a well-posed problem. However, we can study the graph limit of the OM minimizers
as $T$ goes to infinity, and this graph limit indeed gives a simpler
description about the MPP of the OM functional with finite but sufficiently
large $T$.  This fact is clearly demonstrated in the example shown in ~\ref{fig:2D},
in which the path with a sharp corner has a simpler structure
but still well characterizes
the transition path one would obtain with a sufficiently large $T$.
In this sense, the graph limit can be viewed as a good description
of the MPP of the OM functional when $T$ is finite but sufficiently large.
This situation draws an analogy with the shock solution
of hyperbolic conservation laws  where the discontinuous shock
solution with vanishing viscosity provides a simpler
description for the solution to the corresponding parabolic system
with small viscosity \cite{dafermos2005hyperbolic}.

A natural procedure to investigate this graph limit is to first
find  the minimizer with a fixed $T$,
then let $T$ go to infinity. However, this
procedure based on the original OM functional is neither effective
nor transparent for
the characterization of the graph limit due to the time parametrization of the path.
To avoid this difficulty, we directly study the limit of the minimizers on the space of curves
in which the path is parametrized by the normalized arclength.
This gives more direct physical intuition and indeed the time parameterization
can be recovered from this geometric path afterwards.
Specifically, using a similar idea employed in \cite{heymann2008geometric} we
reformulate the double minimization problem in a geometric fashion and
look for the extremal of the action functional
\begin{equation}\label{eq:ROM}
\hat{S}_E[\phi]=\int_{0}^{1}\sqrt{2E-2U(\phi)}\abs{\phi'} - \vec{b}(\phi)\cdot\phi'\d\alpha
\end{equation}
with a proper energy $E$, where $\phi\in C[0,1]$ is  the geometric path with an
arclength type of parameterization,
and $\phi'$ is the derivative of $\phi$ with respect to $\alpha\in [0,1]$
(see details in Theorems \ref{prop:conv} and \ref{thm:extremal}).
The functional $\hat{S}_E[\phi]$ is referred to as the geometric OM functional.
Numerically, this approach was also pursued in \cite{faccioli2006dominant, wang2010kinetic},
but  no theory was developed there  on the choice of the energy $E$ beforehand.
This point will be made clear in this work through rigorous analysis.

We now summarize the main contributions of this paper. First, we demonstrate that the cause of the singularity arising
from minimizing $S^{\OM}_T[\psi_T]$ when $T\rightarrow\infty$
can be explicitly identified by separating the OM functional into two parts:
a regular part containing the geometric OM functional \eqref{eq:ROM},
and a singular part given by $-E(T)T$. It is this singular part that drives
the OM functional to $-\infty$ as $T\rightarrow \infty$. The graph limit of the
minimizer of the OM functional can be identified from the regular part, i.e.
the geometric OM functional \eqref{eq:ROM}.
 This observation is crucial for the subsequent
theoretical studies.  Secondly, we prove that, up to a subsequence, the graph of the minimizer
of the OM functional uniformly converges to $\phi^\star$
as $T_{k}\rightarrow\infty$, and the corresponding transition energy $E_{k}$
converges to $E^\star= \max_{\vec{x}}U(\vec{x})$ for gradient dynamics with
$\vec{b}(\vec{x})=-\nabla V(\vec{x})$. Furthermore, we prove
that the limit $\phi^\star$
is an extremal of the geometric OM functional \eqref{eq:ROM}
with the energy $E=E^\star$. Thirdly, we propose an iterative numerical method
to identify the graph limit $\phi^\star$ and at the same time, to compute
the critical energy $E^\star$.
We also analyze the convergence of the semi-discretized numerical scheme
and apply the numerical method to several typical model problems in rare event studies. On the technical side, proving the compactness of the graph minimizers
is highly nontrivial and represents the main challenge in the analysis.
To the best of our knowledge,
both the theoretical results and the numerical method are new and should be beneficial
to future studies in understanding FW-OM connections and calculus of variations with similar issues.

The rest of the paper is organized as follows. In Section \ref{sec:Assump},
we introduce the notations and assumptions that are used in the later analysis.
In Section \ref{sec:MinOM},  we theoretically study the graph limit of
the minimizer of the OM functional  by first transforming the minimization of the OM functional with respect to $\psi$ and $T$
into a geometric minimization problem on the space of curves through
the  Maupertuis principle. We then show the subsequence convergence
of the graph minimizers and the convergence of the corresponding transition energy.
In Section \ref{sec:Compare}, we compare the minimizers of the FW and OM functionals
for some simple but enlightening examples. In Section \ref{sec:Method},
we propose an iterative energy-climbing geometric minimization algorithm (EGMA) and
discuss its convergence property.  In Section \ref{sec:Results},
we apply the numerical method to some typical model problems
and compare them with the FW minimizers. Some concluding remarks are made in Section \ref{sec:con}.
The technical details about the BV compactness of the derivative of $\{\phi_k\}$,
are provided in the Appendix.

\section{Assumptions and preliminary setup}\label{sec:Assump}

Before proceeding to the analysis of the OM functional, we first introduce
some notations and assumptions. Our analysis will focus on gradient systems
with the drift function $\vec{b}(\vec{x})=-\nabla V(\vec{x})$  for some potential energy
$V(\vec{x})\in C^7(\mathbb{R}^d)$,
although some results can be readily extended to the non-gradient case.
 In the special gradient flow case,
the path potential $U(\vec{x})$, which plays an important role in the analysis
of the OM functional, is given by
\begin{equation}\label{eq:GU}
  U(\vec{x}) = \epsilon\laplace V(\vec{x})-\frac{1}{2}\abs{\nabla V(\vec{x})}^2.
\end{equation}
We have $U(\vec{x})\in C^5$ by the smoothness assumption on $V$.
We further assume that the potentials $V$ and $U$ satisfy the following properties:

\begin{assumption}\label{asm:potV}
There exists a local minimizer $\vec{x}_m$ of $V(\vec{x})$, such that $\nabla V(\vec{x}_m)=0$, $\nabla^2 V(\vec{x}_m)$ is strictly positive definite.
\end{assumption}

\begin{assumption}\label{asm:decay}
 The maximum points of $U(\vec{x})$ are contained in a bounded domain.
\end{assumption}

\begin{assumption}\label{asm:levelset}
  For any $E\in \mathbb{R}$, the level set $\mathcal{L}_E=\{\vec{x}\in\mathbb{R}^d|U(\vec{x})=E\}$ can be decomposed into a finite number of closed and connected subsets, i.e.
  \begin{equation}\label{eq:decomposelevelset}
    \mathcal{L}_E= \bigcup\limits_{k=1}^N B_k,
  \end{equation}
where the subsets $B_k$ are closed and connected, and  $B_j\cap B_k=\emptyset$ if $j\neq k$.
\end{assumption}

When $\epsilon=0$, the path potential  reduces to
$U(\vec{x})=-\frac{1}{2}\abs{\nabla V}^2\leqslant 0$,
which attains its maximum iff $\nabla V(\vec{x})=0$, i.e.
the stationary states of the deterministic dynamics $\dot{\vec{x}}=-\nabla V(\vec{x})$. Assumption \ref{asm:decay} requires that all points satisfying $\nabla V(\vec{x})=0$
lie in a bounded region. This is a usual assumption when dealing with the FW functional.
When $\epsilon$ is finite but not very large, this requirement is not restrictive either.
Indeed this is true in most previous studies on the OM functional.

For the FW functional, the point $\vec{x}^\star$ is called a critical point
if $\nabla V(\vec{x}^\star)=0$. In the following,
we generalize this concept to the OM functional.
\begin{definition}\label{def:cp}
$\vec{x}^\star\in\mathbb{R}^d$ is called a critical point of the OM functional
if
\begin{equation}\label{eq:critical}
  U(\vec{x}^\star)= \max_{\vec{y}\in\mathbb{R}^d}U(\vec{y}).
\end{equation}
\end{definition}

Denoting  the set of critical points by $\Lambda$,
we make the following  mild assumptions on them.

\begin{assumption}\label{asm:discret}
 $\Lambda$ is discrete and has no accumulation points.
\end{assumption}

\begin{assumption}\label{asm:nondegenerate}
  $\nabla^2 U(\vec{x})$ has no zero eigenvalue for every $\vec{x}\in\Lambda$.
\end{assumption}

Let $C[0,T]$ denote the set of continuous functions on $[0,T]$
equipped with the norm $\norm{f}=\sup_{t\in [0,T]}\abs{f(t)}$.
Define
$$C_{\vec{x}_s}^{\vec{x}_f}[0,T]=\big\{\psi\in C[0,T]|\ \psi(0)=\vec{x}_s,\psi(T)=\vec{x}_f\big\}.$$
Moreover, let $\bar{C}[0,T]$, $\bar{C}_{\vec{x}_s}^{\vec{x}_f}[0,T]$ be the set of
corresponding absolutely continuous functions. We further define
\begin{equation}\label{eq:acm}
   \bar{C}_{M}^{\vec{x}_s,\vec{x}_f}[0,T]=\{\psi\in
\bar{C}_{\vec{x}_s}^{\vec{x}_f}[0,T]|\ |\dot\psi|\leqslant M\ a.e.\}
\end{equation}
and
\begin{equation}\label{eq:acme}
    \bar{C}_{M,E}^{\vec{x}_s,\vec{x}_f}[0,T] = \{\psi\in \bar{C}_{M}^{\vec{x}_s,\vec{x}_f}[0,T]|
  U(\psi(t))\leqslant E,\  t\in [0,T]\}
  \end{equation}
for $M>0$.

\begin{lemma}[Compactness of $\bar{C}_M$ and $\bar{C}_{M,E}$]\label{lem:compact}
\begin{enumerate}
  \item For any $M>0$, $\vec{x}_s, \vec{x}_f\in\mathbb{R}^d$, the set
$\bar{C}_{M}^{\vec{x}_s,\vec{x}_f}[0,T]$
  is compact in $C[0,T]$.
  \item For any $M>0, E\in\mathbb{R}$, $\vec{x}_s, \vec{x}_f\in\mathbb{R}^d$, the set
$ \bar{C}_{M,E}^{\vec{x}_s,\vec{x}_f}[0,T]$
  is compact in $C[0,T]$.
\end{enumerate}
\end{lemma}

\begin{proof}
The first statement is the same as Lemma 2.5 in \cite{heymann2008geometric}, which is a standard application of Arzel\`a-Ascoli theorem. The second statement is true since  $U$ is continuous and $\bar{C}_{M,E}^{\vec{x}_s,\vec{x}_f}$ is a closed subset of $\bar{C}_M^{\vec{x}_s,\vec{x}_f}$.
\end{proof}

Hereafter, we will neglect the super- or sub-scripts $\vec{x}_s$ and $\vec{x}_f$
to simplify the notations when we do not emphasize
the dependence on the initial and terminal states. It is also evident that
the function spaces defined above and the compactness lemma are also applicable  to
functions on the interval $[0,1]$, which is the case when we
study the geometric minimizer on the space of curves.
With a slight abuse of notation, we will use $\bar{C}_{M}$ ($\bar{C}_{M,E}$)
for $\bar{C}_{M}[0,T]$ ($\bar{C}_{M,E}[0,T]$) or
$\bar{C}_{M}[0,1]$ ($\bar{C}_{M,E}[0,1]$)
in later text when the domain is not specifically emphasized.

 For every function $f\in \bar{C}[0,T]$, we denote its graph by
\[\gamma(f) = \{f(t)|t\in [0,T]\}.\]
For any two functions $f_1\in \bar{C}[0,T_1]$ and $f_2\in \bar{C}[0,T_2]$
with possibly different parameterizations, we use the Fr\'echet distance to measure the distance
between their graphs $\gamma(f_1)$ and $\gamma(f_2)$:
\begin{equation}\label{eq:freche}
  \rho(f_1,f_2) = \inf\limits_{t_i:[0,1]\to [0,T_i], i=1,2}\
\max_{\alpha\in[0,1]}\abs{f_1\circ t_1(\alpha)-f_2\circ t_2(\alpha)},
\end{equation}
where the infimum is taken over all monotonically increasing, continuous and
surjective reparametrizations $t_1$ and $t_2$.
It is easy to check that $\rho(f_1,f_2)=0$ if the two functions have the same graph.

Given the initial and final states $\vec{x}_s, \vec{x}_f$,
we aim at solving the double minimization problem
\begin{equation}\label{eq:doublemin}
\inf\limits_{T>0}\inf\limits_{\psi\in \bar{C}_{\vec{x}_s}^{\vec{x}_f}[0,T]}
\left\{S_T^{\OM}[\psi] = \int_{0}^{T}\left(
\frac{1}{2}|\dot{\psi}+\nabla V(\psi)|^2-\epsilon\laplace V(\psi)\right)\d t\right\}.
\end{equation}
Note that
\begin{eqnarray*}
S_T^{\OM}[\psi] &=& \int_{0}^{T}\left(
\frac{1}{2}|\dot{\psi}|^2 + \nabla V\cdot \dot{\psi} -U(\psi)\right)\d t\\
&=& \int_{0}^{T}\left(\frac{1}{2}|\dot{\psi}|^2 - U(\psi)\right)\d t
+ V(\vec{x}_f) - V(\vec{x}_s).
\end{eqnarray*}
Since $V(\vec{x}_f) - V(\vec{x}_s)$ is a constant when $\vec{x}_s$ and $\vec{x}_f$
are given,  we will ignore this constant and
use the resulting simplified functional in place of $S_T^{\OM}[\psi]$.
With a slight abuse of notation, we still
use $S_T^{\OM}[\psi]$ to denote the simplified functional and
call it the OM functional:
\begin{equation}
S_T^{\OM}[\psi] = \int_{0}^{T} L(\psi, \dot\psi) \d t,
\end{equation}
where the Lagrangian $L(\vec{x},\vec{y})$ is given by
\begin{equation}
L(\vec{x},\vec{y}) = \frac{1}{2}|\vec{y}|^2 - U(\vec{x}).
\end{equation}

In the following, we assume that for any fixed $T>0$,
the minimizer of $S_T^{\OM}[\psi]$ is unique and
has  a uniformly  bounded finite length.  We state it precisely as follows.
\begin{assumption}\label{asm:unique}
  For any given $\vec{x}_s, \vec{x}_f\in\mathbb{R}^d$ and $T>0$, $S_T^{\OM}[\psi]$ has a unique minimizer $\psi_T$.
\end{assumption}

\begin{assumption}\label{asm:uniformlybounded}
For any given $\vec{x}_s, \vec{x}_f\in\mathbb{R}^d$ and $T>0$, there exists a constant $M=M(\vec{x}_s,\vec{x}_f)>0$ such that  the minimizer $\psi_T$ of $S_T^{\OM}[\psi]$ satisfies $$\int_{0}^{T}|\dot{\psi}_T|\d t\leqslant M.$$
\end{assumption}

\begin{corollary}\label{col:bound}
Under the Assumption \ref{asm:uniformlybounded}, $\psi_T$ is uniformly bounded in $C[0,T]$ for any $T>0$.
\end{corollary}
\begin{proof}We have
\begin{equation}\label{eq:col1}
  \abs{\psi_T(t)} = \Big|\psi_T(0) + \int_{0}^{t}\dot{\psi}_T(s)\d s\Big|\leqslant \abs{\vec{x}_s} + \int_{0}^{t}|\dot{\psi}_T|\d s\leqslant \abs{\vec{x}_s} + M.
\end{equation}
\end{proof}

The boundedness of $\psi_T$ also implies the boundedness of $f(\psi_T)$ for any continuous
function $f$.

\section{Graph limit of the OM minimizer}\label{sec:MinOM}

In this section, we study the minimizers of the OM functional and their graph limit.
First, we consider the problem of minimizing the OM functional with the transition time
$T$ fixed. Then we transform the minimization problem to a geometric one
in which the path is parameterized by the normalized arc-length using the
Maupertuis principle. This is followed by the analysis of the minimization problem
with respect to the transition time $T$. The last part of this section is
concerned with the graph limit of the minimizers and its governing equation.

\subsection{Minimizing $S^{\OM}_T[\psi]$ for a fixed time $T$}

We first consider the minimization of $S^{\OM}_T[\psi]$ over the set of absolutely continuous functions
connecting given initial and final states.
For a fixed $T>0$, this is a standard variational problem.
The following proposition gives the regularity result of the minimizer
and its governing equation.

\begin{proposition}{\label{prop:finitT}}
 Assume that $V\in C^7(\mathbb{R}^d)$. For any $T>0$, the functional $S^{\OM}_T[\psi]$
is lower-semicontinuous in $\bar{C}_{\vec{x}_s}^{\vec{x}_f}[0,T]$,
thus attains its minimum in $\bar{C}_M^{x_s,x_f}[0,T]$. Moreover,
under the Assumption \ref{asm:uniformlybounded}, the minimizer $\psi_T\in C^6[0,T]$
and satisfies the Euler-Lagrange equation
  \begin{equation}\label{eq:el0}
    \mathcal{D}L = \pp{L}{\vec{x}}(\psi_T,\dot{\psi}_T)-\frac{\d}{\d t}\pp{L}{\vec{y}}(\psi_T,\dot{\psi}_T)=0,
  \end{equation}
  i.e.
  \begin{equation}\label{eq:EL}
  \begin{cases}
    \ddot{\psi}_T + \nabla U(\psi_T) =0,\\
    \psi_T(0)=\vec{x}_s,\ \psi_T(T)=\vec{x}_f.
  \end{cases}
  \end{equation}
\end{proposition}

This proposition ensures the existence and smoothness of the minimizer in a proper function space. Its proof can be found in \cite{giaquinta2013calculus} (Proposition 4 in page 42).

\begin{remark} For general systems with a
drift $\vec{b}(\vec{x})$, the minimizer $\psi_T$ satisfies the Euler-Lagrange equation
  \begin{equation}\label{eq:ELNG}
    \ddot{\psi}_T + (\nabla\vec{b}^\T - \nabla \vec{b})\cdot\dot{\psi}_T + \nabla U(\psi_T) =0,
  \end{equation}
  where $(\nabla\vec{b})_{ij}=\partial b_i/\partial x_j$.
\end{remark}

It is a classical result that the energy of the system is conserved along
the path $\psi_T$, i.e.
\begin{equation}\label{eq:conservation}
  \frac{1}{2}|\dot{\psi}_T|^2 + U(\psi_T)\equiv E,\quad \forall t\in [0,T].
\end{equation}
With the Assumption \ref{asm:unique}, the constant $E$
is uniquely determined by the initial and terminal states $\vec{x}_s, \vec{x}_f$
and the transition time $T$.
For fixed $\vec{x}_s$ and $\vec{x}_f$, $E$ is a function of $T$ only.
In this case, $E$ and $T$ are connected by the equation:
\begin{equation}\label{eq:TandE0}
  T = \int_{\gamma(\psi_T)}\frac{\abs{\d\psi_T}}{\sqrt{2K_{E(T)}(\psi_T)}},
\end{equation}
where $K_E(\psi)=E-U(\psi)$ is the kinetic energy,  $\gamma(\psi_T)$ is the graph of $\psi_T$.

\subsection{Maupertuis principle}

In many cases, we are mainly interested in the graph of the transition path
in the configuration space rather than how it is parameterized by time $t$.
It is well known that the Maupertuis variational principle, which is equivalent to the
Hamilton's variational principle, is more convenient for
such representations \cite{Rota1990Mathematical, landau1972mechanics}.
\begin{theorem}[Maupertuis Principle]
Up to a reparameterization of the path, the minimizer of variational problem
  \begin{equation}\label{eq:vp1}
    \inf\limits_{\psi\in \bar{C}_{\vec{x}_s}^{\vec{x}_f}[0,T]}S_T^{\OM}[\psi]
  \end{equation}
  is an extremal of the functional
\begin{equation}\label{eq:vp2}
  S_0[\psi] = \int_{\gamma(\psi)}\vec{p}(\psi,\dot{\psi})\cdot\d\psi
\end{equation}
subject to the constraint  $H(\psi,\vec{p})\equiv E(T)$, where $H$ is the corresponding Hamiltonian
  \begin{equation}\label{eq:Hamiltonian}
    H(\vec{x},\vec{p}) = \frac{1}{2}\abs{\vec{p}}^2 + U(\vec{x}),
  \end{equation}
 and $\vec{p}(\psi,\dot\psi) = \pp{L}{\vec{y}}(\psi,\dot\psi) = \dot\psi$  is the momentum.
\end{theorem}

The functional $S_0[\psi]$ is called the abbreviated action functional
or effective functional \cite{landau1972mechanics, lee2017finding,wang2010kinetic}.
Note that the value of $S_0$ does not depend on how
$\gamma(\psi)$ is parameterized.  A convenient way is to parameterize it using the normalized arclength.
The curve with this parameterization is denoted by $\phi(\alpha)$,
which is an element of $\bar{C}[0,1]$,
such that $\phi(0)=\vec{x}_s$, $\phi(1)=\vec{x}_f$
and $\abs{\phi'(\alpha)}\equiv \text{const}$ for $\alpha\in [0,1]$.
For the Hamiltonian \eqref{eq:Hamiltonian} under the constraint
$H(\phi,\vec{p})\equiv E$  and using the normalized arclength parameterization
for the curve, the abbreviated action functional becomes
\begin{equation}\label{hse}
\hS{E}[\phi] = \begin{cases}
\int_0^1\sqrt{2K_E(\phi)}\abs{\phi'}\d\alpha, & \mbox{if } \phi\in \bar{C}[0,1], \,U(\phi)\leqslant E \text{ for }\alpha\in [0,1],\\
+\infty, & \mbox{otherwise}.
\end{cases}
\end{equation}

More specifically, suppose the minimizer of $S^{\OM}_T[\psi]$ is $\psi_T$,
we define the new parameter $\alpha=\ell(t)$ as
\begin{equation}\label{eq:repara}
  \ell(t) = \frac{1}{L}\int_{0}^{t}|\dot{\psi}_T(s)|\d s, \quad \text{where }L = \int_{0}^{T}|\dot{\psi}_T|\d t.
\end{equation}
Let $\ell^{-1}(\alpha) = \inf\{t\in[0,T]|\ell(t)\geqslant\alpha\}\in [0,T]$, then by the
Maupertuis principle,
\[\phi_{E(T)}(\alpha) = \psi_T(\ell^{-1}(\alpha))\]
 is an extremal of $\hS{E(T)}[\phi]$. Note that $\dot{\ell}(t) = L^{-1}|\dot{\psi}_T|$ and $\psi_T\in C^2[0,T]$. We have that $\ell^{-1}(\alpha)$ is also twice differentiable when $|\dot{\psi}_T|\neq 0$ and
\[\abs{\dd{\phi_E}{\alpha}} = |\dot{\psi}_T|\dd{t}{\alpha} = L.\]
If $\psi_T$ satisfies the Assumption \ref{asm:uniformlybounded}, $\phi_E\in \bar{C}_M[0,1]$ with $M\geqslant L$. By the equation\eqref{eq:conservation}, $|\dot{\psi}_T|=0$ iff $E(T) = U(\psi_T(t))$, so $\phi_E(\alpha)$ is twice differentiable where $U(\phi_E(\alpha))<E$.

We emphasize that even when $\psi_T$ is the minimizer of $S^{\OM}_T[\psi]$,
$\phi_{E(T)}$ need not to be a minimizer of $\hS{E(T)}[\phi]$ but only an extremal.
An illustrative example will be discussed in Section \ref{sec:Exam}.

Denote the Lagrangian of $\hS{E}[\phi]$ by
\begin{equation}\label{eq:L0}
  L_E(\vec{x},\vec{y}) = \sqrt{2K_E(\vec{x})}\abs{\vec{y}}.
\end{equation}
 When $\vec{y}\neq 0$, $\partial L_E/\partial \vec{y}$ is well-defined. The extremal of $\hS{E}[\phi]$, denoted by $\phi_E$, satisfies the Euler-Lagrange equation in a weak form
 \begin{equation}\label{eq:wel}
   \int_{0}^{1}\pp{L_E}{\vec{y}}\cdot\Phi'\d\alpha + \int_{0}^{1}\pp{L_E}{\vec{x}}\cdot\Phi\d\alpha=0,\quad \forall \Phi\in C_0^{\infty}[0,1].
 \end{equation}
Using the specific form of $L_E$, we get
 \begin{equation}\label{eq:weakEL}
   \int_{0}^{1}\sqrt{2K_E(\phi)}\frac{\phi'}{\abs{\phi'}}\cdot\Phi'\d\alpha  - \int_{0}^{1}\frac{\abs{\phi'}\nabla U(\phi)\cdot\Phi}{\sqrt{2K_E(\phi)}}\d\alpha=0.
 \end{equation}
 When $U(\phi_E)<E$, $\phi_E\in C^2$, we have the strong form of the Euler-Lagrange equation
\begin{align}\label{eq:EL0}
& 2K_E(\phi) \phi''  +\abs{\phi'}^2 (I-\hat\tau\otimes\hat\tau)\cdot\nabla U(\phi)=0,\quad \alpha\in(0,1), \\
&\phi(0)=\vec{x}_s ,\  \phi(1)=\vec{x}_f \quad \text{and}\quad (\abs{\phi'})'=0,
\end{align}
where $\hat\tau=\phi'/\abs{\phi'}$ and $U(\phi)< E$.

Moreover, Eq. \eqref{eq:TandE0} can be written as
\begin{equation}\label{eq:TandE}
  T=\int_{0}^{1}\frac{|\phi_{E(T)}'|}{\sqrt{2K_{E(T)}(\phi_{E(T)})}}\d\alpha,
\end{equation}
and the infimum of the OM functional $S^{\OM}_T[\psi]$ is given by
\begin{align}\label{eq:ST=S0-ET}
  S^{\OM}_T[\psi_T] = & \int_{0}^{T}
\left(\frac{1}{2}|\dot{\psi}_T|^2 - U(\psi_T)\right) \d t \nonumber\\
              = & \int_{0}^{T}|\dot{\psi}_T|^2\d t -E(T)T \nonumber\\
              = & \int_{0}^{T}\sqrt{2K_{E(T)}(\psi_T)}|\dot{\psi}_T|\d t -E(T)T \nonumber\\
              = & \hS{E(T)}[\phi_{E(T)}] -E(T)T.
\end{align}
The connection between  $S^{\OM}_T[\psi_T]$ and $\hS{E(T)}[\phi_{E(T)}]$ in \eqref{eq:ST=S0-ET}
is essential for later analysis.

\begin{remark}
 The above results can be generalized to non-gradient systems.
For a general system with a drift $\vec{b}$, we have the Lagrangian $L(\vec{x},\vec{y}) = \frac{1}{2}\abs{\vec{y}}^2 - \vec{b}(\vec{x})\cdot\vec{y} - U(\vec{x})$ and the corresponding Hamiltonian $H(\vec{x},\vec{p})=\frac{1}{2}\abs{\vec{p}+\vec{b}(\vec{x})}^2+U(\vec{x})$. The Maupertuis principle also holds with the geometric functional
  \begin{equation}\label{NGhse}
  \hS{E}[\phi] = \begin{cases}
                   \int_0^1 (\sqrt{2K_E(\phi)}\abs{\phi'}
- \vec{b}\cdot\phi')\d\alpha, & \mbox{if } \phi\in \bar{C}[0,1] \text{ and }U(\phi)\leqslant E,\\
                   +\infty, & \mbox{otherwise}.
                 \end{cases}
\end{equation}
With the Lagrangian $L_E(\vec{x},\vec{y}) = \sqrt{2K_E(\vec{x})}\abs{\vec{y}}-\vec{b}(\vec{x})\cdot\vec{y}$, Eqs.~\eqref{eq:wel}, \eqref{eq:TandE} and \eqref{eq:ST=S0-ET} still hold. The strong form of the Euler-Lagrange equation is
\begin{equation}\label{eq:NGEL0}
  2K_E(\phi)\phi'' +\sqrt{2K_E(\phi)}(\nabla\vec{b}^\T-\nabla\vec{b})\phi' + \abs{\phi'}^2(I-\hat\tau\otimes\hat\tau)\cdot\nabla U(\phi) =0
\end{equation}
with the constraint $\abs{\phi'}=\text{const}$ and $U(\phi)< E$.
\end{remark}

\subsection{Minimization when $T$ goes to infinity}\label{sec:minT}

Given $\psi_T$, the minimizer of $S^{\OM}_T[\psi]$, the value of $S^{\OM}_T[\psi_T]$
can be viewed as a function of $T$.
 As we will see below,  taking the limit of $S^{\OM}_T[\psi_T]$ as
 $T\rightarrow\infty$ is equivalent to minimizing the OM functional
with respect to $\psi$ followed by the minimization with respect to $T>0$.
We will follow the double minimization approach,
which provides insights to the numerical results obtained
in \cite{pinski2010transition} and also facilitates comparisons with FW functional.
The double minimization problem \eqref{eq:doublemin}
reduces to a minimization with respect to $T$:
$S(\vec{x}_s,\vec{x}_f)=\inf_T S^{\OM}_T[\psi_T]$.
We have:
\begin{proposition}\label{prop:pspt}
  \[\pp{\hS{E}[\phi_E]}{E} = T,\quad \pp{S^{\OM}_T[\psi_T]}{T}=-E(T).\]
\end{proposition}
This is a classical result in the Hamilton-Jacobi theory.
One may refer to \cite{giaquinta2013calculus} for a rigorous proof.
A formal derivation is as follows.
\begin{align}\label{eq:pspt}
  \pp{\hS{E}[\phi_E]}{E} &= \int_{0}^{1}\frac{\abs{\phi_E'}\d\alpha}{\sqrt{2K_E(\phi_E)}} + \int_{0}^{1}\mathcal{D}L_E(\phi_E,\phi_E')\frac{\partial \phi_E}{\partial E}\d\alpha.
\end{align}
Since $\phi_E$ is an extremal of $\hS{E(T)}[\phi]$, we have
$\mathcal{D}L_E(\phi_E,\phi_E')=0$, thus $\pp{\hS{E}[\phi_E]}{E}=T$.
Furthermore, by Eq. \eqref{eq:ST=S0-ET},
\[\pp{S^{\OM}_T[\psi_T]}{T}=\pp{\hS{E}[\phi_E]}{E}\pp{E}{T}-E(T)-T\pp{E}{T}=-E(T).\]

Next we show that $E(T)>0$ when $T$ is sufficiently large.
Then it follows from Proposition \ref{prop:pspt} that $S^{\OM}_T[\psi_T]$
is a decreasing function of $T$ when $T$ is sufficiently large.

\begin{lemma}\label{lem:liminfE}
  \[\liminf_{T\to +\infty}E(T)\geqslant \max_{\vec{x}\in\mathbb{R}^d}U(\vec{x})>0.\]
\end{lemma}
\begin{proof}By Assumption \ref{asm:potV}, we have $\max_{\vec{x}\in\mathbb{R}^d}U(\vec{x})\geqslant U(\vec{x}_m)=\epsilon\laplace V(\vec{x}_m)>0$. Denote $E_m=\liminf_{T\to +\infty}E(T)$, $E_{\rm max} = \max_{\vec{x}\in\mathbb{R}^d}U(\vec{x})$, then
there exists a sequence $T_k\to +\infty$ such that $E_k=E(T_k)\to E_m$.

We will prove the result by contradiction. If $\Delta E = E_{\rm max}-E_m>0$,  we have $E_{\rm max}-E_{k}>\frac{3}{4}\Delta E>0$ when $k$ is large enough. Denote
the minimizer of $S^{\OM}_{T_k}[\psi]$ by $\psi_k$,
and the corresponding extremal of $\hS{E_k}[\phi]$ that has the same
graph as $\psi_k$ by $\phi_k$. We have
\[U(\phi_{k})\leqslant E_{k}< E_{\rm max}-\frac{3}{4}\Delta E\]
for any sufficiently large $k$. This implies that
$\phi_{k}(\alpha)\notin S = \{\vec{x}|E_{\rm max}-\Delta E/2<U(\vec{x})\leqslant E_{\rm max}\}$
for any $\alpha\in[0,1]$. Let us choose a point $\vec{x}_c$ from the set $S$.
For $T_{k}$ sufficiently large,
let $\tilde{\psi}_1(t)\in \bar{C}_M^{\vec{x}_s,\vec{x}_c}[0,T_{k}/2]$ be
the minimizer of $S^{\OM}_{T_{k}/2}[\psi]$ with $\psi(0)=\vec{x}_s$, $\psi(T_{k}/2)=\vec{x}_c$.
The energy conservation \eqref{eq:conservation} yields
\[\big|\dot{\tilde{\psi}}_1\big|^2 +2U(\tilde{\psi_1})\equiv 2\tilde{E}_1\]
for some constant $\tilde{E}_1$. Clearly,
$\tilde{E}_1\geqslant U(\vec{x}_c)>E_{\rm max}-\Delta E/2$.
By the Maupertuis principle, there exists an extremal
$\tilde{\phi}_1\in \bar{C}_M[0,1]$ of $\hS{\tilde{E}_1}[\phi]$
which has the same graph as $\tilde{\psi}_1$. Eq.~\eqref{eq:ST=S0-ET} gives
\[S^{\OM}_{T_{k}/2}[\tilde{\psi}_1]
= \hS{\tilde{E}_1}[\tilde{\phi}_1]-\frac{1}{2}\tilde{E}_1T_{k}.\]

Similarly, let $\tilde{\psi}_2\in \bar{C}_M^{\vec{x}_c,\vec{x}_f}[0,T_{k}/2]$ be
the minimizer of $S^{\OM}_{T_{k}/2}[\psi]$
with $\psi(0)=\vec{x}_c$ and $\psi(T_{k}/2)=\vec{x}_f$.
The energy along $\tilde{\psi}_2$ satisfies $\tilde{E}_2>E_{\rm max}-\Delta E/2$.
Let $\tilde{\phi}_2$ be the corresponding extremal of $\hS{\tilde{E}_2}[\phi]$
which has the same graph as $\tilde{\psi}_2$. We have
\[S^{\OM}_{T_{k/2}}[\tilde{\psi}_2]
= \hS{\tilde{E}_2}[\tilde{\phi}_2]-\frac{1}{2}\tilde{E}_2T_{k}.\]

Define $\hat{\psi}\in \bar{C}_M[0,T_{k}]$ as
\begin{equation}\label{eq:hatpsi}
  \hat{\psi}(t)=\begin{cases}
                  \tilde{\psi}_1(t), & \mbox{if } 0\leqslant t\leqslant \frac{1}{2}T_{k}, \\
                  \tilde{\psi}_2(t-\frac{1}{2}T_{k}), & \mbox{if } \frac{1}{2}T_{k}\leqslant t\leqslant T_{k}.
                \end{cases}
\end{equation}
We have
\begin{equation}\label{eq:Shat-Sk}
\begin{aligned}
  S^{\OM}_{T_{k}}[\hat{\psi}] - S^{\OM}_{T_{k}}[\psi_k]
=& \hS{\tilde{E}_1}[\tilde{\phi}_1] + \hS{\tilde{E}_2}[\tilde{\phi}_2] - \hS{E_{k}}[\phi_{k}] - \frac{1}{2}(\tilde{E}_1 + \tilde{E}_2 - 2E_{k})T_{k}\\
 \leqslant &\int_{0}^{1}\sqrt{2\tilde{E}_1-2U(\tilde{\phi}_1)}\abs{\tilde{\phi}_1'}\d\alpha + \int_{0}^{1}\sqrt{2\tilde{E}_2-2U(\tilde{\phi}_2)}\abs{\tilde{\phi}_2'}\d\alpha\\
   & -\int_{0}^{1}\sqrt{2E_{k}-2U(\phi_k)}\abs{\phi_k'}\d\alpha - \frac{1}{4}\Delta E T_{k}.
\end{aligned}
\end{equation}
Since $\tilde{\phi}_1, \tilde{\phi}_2, \phi_k\in \bar{C}_M[0,1]$, $\hS{\tilde{E}_1}[\tilde{\phi}_1]$, $\hS{\tilde{E}_2}[\tilde{\phi}_2]$ and $\hS{E_{k}}[\phi_{k}]$ are all uniformly bounded. Thus
\[S_{T_{k}}[\hat{\psi}] - S_{T_{k}}[\psi_k]\leqslant C - \frac{1}{4}\Delta ET_{k}<0\]
when $T_{k}$ is sufficiently large, where $C$ is a generic constant. This contradicts with the
assumption that $\psi_{k}$ is the minimizer of $S_{T_{k}}[\psi]$.
\end{proof}

Proposition \ref{prop:pspt} and Lemma \ref{lem:liminfE} show
that $S^{\OM}_T[\psi_T]$ is a decreasing function of $T$
with a positive decreasing rate when $T$ is sufficiently large.
Thus  $S^{\OM}_T[\psi_T]\rightarrow - \infty$ as $T\rightarrow +\infty$.
Next we show that the infimum of $S^{\OM}_T[\psi_T]$ occurs only when $T\to +\infty$.

\begin{proposition}\label{prop:D}
$S^{\OM}_T[\psi_T]$ assumes its infimum only when $T\to +\infty$. Moreover, we have
\[\inf\limits_{T>0}\inf\limits_{\psi\in \bar{C}_{\vec{x}_s}^{\vec{x}_f}[0,T]}S^{\OM}_T[\psi] = \lim\limits_{T\to +\infty}\inf\limits_{\psi\in \bar{C}_{\vec{x}_s}^{\vec{x}_f}[0,T]}S^{\OM}_T[\psi]=-\infty.\]
\end{proposition}

\begin{proof}
By the mean value theorem and Assumption \ref{asm:uniformlybounded},
we derive from Eq. \eqref{eq:TandE} that
  \begin{equation}\label{eq:propD1}
 \begin{aligned}
    \frac{\abs{\vec{x}_s-\vec{x}_f}}{\sqrt{2K_{E}(\psi_T(t_c))}}
\leqslant T & =
\frac{\int_{0}^{T}|\dot{\psi_T}|\d t}{\sqrt{2K_{E}(\psi_T(t_c))}}
\leqslant \frac{M}{\sqrt{2K_{E}(\psi_T(t_c))}}
    \end{aligned}
  \end{equation}
where $K_{E}(\psi_T(t_c)) =E(T)-U(\psi_T(t_c))$ for some $t_c\in [0,T]$. Thus
  \begin{align}\label{eq:propD2}
      U(\psi_T(t_c))   + \frac{\abs{\vec{x}_s -\vec{x}_f}^2}{2T^2}\leqslant E(T)
    \leqslant U(\psi_T(t_c)) + \frac{M^2}{2T^2}.
  \end{align}
Since $\{\psi_T\}$ is uniformly bounded, $\{U(\psi_T)\}$ is also uniformly bounded by a constant $M_U$. Inequality \eqref{eq:propD2} implies $E(T)\to +\infty$ when $T\to 0+$.
By Proposition \ref{prop:pspt}, $S^{\OM}_T[\psi_T]$
is decreasing in a neighborhood of $T=0$. So there exists $T_m>0$,
such that $S^{\OM}_T[\psi_T]$ cannot attain its infimum in $T\in[0,T_m]$.

By Lemma \ref{lem:liminfE}, there exists $T_M>0$
such that when $T>T_M$, $E(T)>\frac{1}{2}E_{\rm max}>0$.
For $T\in [T_m,T_M]$, \eqref{eq:propD2} implies the uniform boundedness of $E(T)$. Denote the upper bound  by $M_E$, lower bound $m_E$. For any path $\phi\in \bar{C}_{M,E}[0,1]$, $\phi$ is uniformly bounded. We have
  \[\hS{E(T)}[\phi_E]\leqslant M\sqrt{2M_E+2M_U},\]
which yields
   \[-M\sqrt{2M_E+M_U}-M_ET_m \leqslant S^{\OM}_T[\psi_T]
\leqslant M\sqrt{2M_E+M_U} - m_ET_M.\]
So  $S^{\OM}_T[\psi_T]$ is bounded in $[T_m,T_M]$.

When $T>T_M$, $S^{\OM}_T[\psi_T]$ is a decreasing function of $T$. So
\[\inf_{T>T_M}\inf_{\psi}S^{\OM}_T[\psi]=\lim_{T\to +\infty}\inf_{\psi}S^{\OM}_T[\psi]\leqslant\lim_{T\to +\infty} M\sqrt{2M_E+2M_U} - \frac{1}{2}E_{\rm max}T= -\infty.\]
Combining with previous facts we obtain
   \[\inf_{T>0}\inf_{\psi}S^{\OM}_T[\psi]=\lim_{T\to +\infty}\inf_{\psi}S^{\OM}_T[\psi]=-\infty.\]
\end{proof}

Proposition \ref{prop:D} shows that the infimum of the
OM functional is always negative infinity. This can be viewed as a special case of the so-called Man\'e potential \cite{contreras1997lagrangian, mane1997lagrangian}.
The Man\'e potential  has a positive Man\'e critical value.
This critical value is given by $E_{\max}=\max_{\vec{x}\in\mathbb{R}^d}U(\vec{x})$
for the OM functional. It can be shown that
for any $C$ below the critical value (e.g. $C=0$ as in the current case),
one has $\inf_{T} \inf_{\psi} S^{\OM}_T[\psi] + CT = -\infty$.
In this sense, we can also call $\hat{S}_E[\phi]$
the \textit{renormalized OM functional} since it removes the divergence term $-E(T)T$ as $T\rightarrow\infty$.

\begin{remark}
Propositions \ref{prop:pspt} and \ref{prop:D} can be generalized to non-gradient case
with slight modifications.
\end{remark}

\subsection{The Graph limit of minimizers of OM functionals}

 Proposition \ref{prop:D} shows that it is inappropriate to define the minimizer $(T^\star, \psi_{T^\star})$ of the original functional $S_T^{\OM}[\psi]$ with time parameterization. However, this limit process could be meaningful if we inspect the transition path sequence in the configuration space instead of using the time coordinate. Next we will study the graph limit of OM minimizers
 as $T\rightarrow +\infty$.

We have shown using the Maupertuis principle that for any $T>0$,
the minimizer $\psi_T$ can be mapped to $\phi\in C[0,1]$ by a reparameterization
such that $\gamma(\psi_T)=\gamma(\phi)$. Moreover, $\phi$ is an extremal
of $\hS{E}[\phi]$ with $E=E(T)$.
In the following, we show that there exist a proper energy parameter
$E^\star$, a path $\phi^\star$ which is an extremal of $\hS{E^\star}[\phi]$
and a subsequence of the minimizer $\left\{\psi_{T_k}\right\}$,
such that $\rho(\psi_{T_k},\phi^\star)\to 0$ when $k\to \infty$.

\begin{theorem}\label{prop:conv}
  Let $\{T_k\}_{k=1}^{\infty}$ be a positive sequence such that $T_k\to +\infty$
when $k\to\infty$. Let $\psi_k$ be the minimizer of the OM functional
$S^{\OM}_{T_k}[\psi]$, $E_k = E(T_k)$, and $\phi_k$ be the extremal of
$\hS{E_k}[\phi]$ which has the same graph as $\psi_k$. We have
  \begin{enumerate}
    \item There exists a subsequence $\{\phi_{k_l}\}$ which uniformly converges to
  $\phi^\star\in \bar{C}_M^{\vec{x}_s,\vec{x}_f}[0,1]$.
    \item The subsequence $\{E_{k_l}\}$ converges to $E^\star=\max_{\vec{x}} U(\vec{x})$
and $\phi^\star\in \bar{C}_{M,E^\star}[0,1]$.
    \item $\phi^\star$ passes through a critical point at some $\alpha_c\in [0,1]$.
    \item $\rho(\psi_{k_l},\phi^\star)\to 0$ when $l\to+\infty$.
    \item For any $\delta>0$, the integral
    \[\int_{\alpha_c-\delta}^{\alpha_c+\delta}\frac{\abs{\phi{^\star}{'}}\d\alpha}
{\sqrt{2E^\star-2U(\phi^\star)}}\]
     diverges. In particular,
     \[T^\star=\int_{0}^{1}\frac{\abs{\phi{^\star}{'}}\d\alpha}
{\sqrt{2E^\star-2U(\phi^\star)}}=+\infty.\]
  \end{enumerate}
\end{theorem}

\begin{proof}
  (1)-(2)-(3).
  Denote $U_k = \max_\alpha U(\phi_k(\alpha))$. Eq.~\eqref{eq:TandE} gives
  \begin{equation}\label{eq:propA1}
    T_k=\int_{0}^{1}\frac{\abs{\phi'_k}\d\alpha}{\sqrt{2E_k - 2U(\phi_k)}}\leqslant \frac{M}{\sqrt{2E_k - 2U_k}}.
  \end{equation}
  So we have
  \begin{equation}\label{eq:propA2}
    U_k\leqslant E_k\leqslant U_k + \frac{M^2}{2T_k^2}.
  \end{equation}

By Assumption \ref{asm:uniformlybounded}, $\{E_k\}$ is uniformly bounded by a constant $M_E$ and $\phi_k\in \bar{C}_{M,M_E}$. From Lemma \ref{lem:compact},
there exists a subsequence $\left\{k_l\right\}$ such that
$E_{k_l}\rightarrow E^\star$ and $\phi_{k_l}\rightarrow\phi^\star\in \bar{C}_{M,M_E}$.  To show $\phi^\star\in \bar{C}_{M,E^\star}$, we note that
\[0\leqslant E_{k_l}-U(\phi_{k_l})\to E^\star-U(\phi^\star)\quad \text{when}\quad l\to\infty,\]
which means $\max_{\alpha}U(\phi^\star(\alpha))\leqslant E^\star$.

Next  we  show $E^\star=\max_{\alpha}U(\phi^\star(\alpha))$. Otherwise $E^\star - \max_{\alpha\in [0,1]}U(\phi^\star(\alpha))=m>0$. Since $E_{k_l}\to E^\star$ and $\phi_{k_l}\to \phi^\star$ uniformly, we have $\abs{E_{k_l}-E^\star}<\delta_0$, $\abs{U(\phi^\star)-U(\phi_{k_l})}<\delta_0$  for any positive $\delta_0$ when $l$ is sufficiently large.
Choose $\delta_0<m/4$, we obtain
  \begin{equation}\label{eq:Ek-Uk}
    E_{k_l}-U(\phi_{k_l}) = (E_{k_l}-E^\star) + (E^\star-U(\phi^\star)) + (U(\phi^\star)-U(\phi_{k_l}))>m-2\delta_0>0.
  \end{equation}
Since $E_{k_l}-U(\phi_{k_l})\to E^\star-U(\phi^\star)$, it follows from
the dominate convergence theorem that
  \begin{equation}\label{Estar=U}
  \begin{aligned}
    +\infty &=\lim\limits_{l\to\infty}T_{k_l} = \lim\limits_{l\to\infty}\int_{0}^{1}\frac{\abs{\phi_{k_l}'}\d\alpha}{\sqrt{2E_{k_l}-2U(\phi_k)}}\\
   & \leqslant\lim\limits_{l\to\infty}M\int_{0}^{1}\frac{\d\alpha}{\sqrt{2E_{k_l}-2U(\phi_k)}}=M\int_{0}^{1}\lim\limits_{l\to\infty}\frac{\d\alpha}{\sqrt{2E_{k_l}-2U(\phi_k)}}\\
    &\leqslant M\int_{0}^{1}\frac{\d\alpha}{\sqrt{2m}}=\frac{M}{\sqrt{2m}},
  \end{aligned}
  \end{equation}
which is a contradiction. Thus $E^\star = \max_{\alpha\in [0,1]}U(\phi^\star(\alpha))$.

 By Lemma \ref{lem:liminfE}, we have
\[\max_{\vec{x}\in\mathbb{R}^d}U(\vec{x})\leqslant\liminf_{T\to +\infty}E(T)\leqslant E^\star= \max_{\alpha\in [0,1]}U(\phi^\star(\alpha))\leqslant\max_{\vec{x}\in\mathbb{R}^d}U(\vec{x}).\]
 So we have $E^\star=\max_{\vec{x}\in\mathbb{R}^d}U(\vec{x})=\max_{\alpha\in[0,1]}U(\phi^\star(\alpha))$. It follows that there must exist an
$\alpha_c\in [0,1]$,
such that $U(\phi^\star(\alpha_c))
= \max_{\vec{x}\in\mathbb{R}^d}U(\vec{x})$.
Thus $\vec{x}_c = \phi^\star(\alpha_c)$ is a critical point.

 (4).   Since $\psi_{k_l}$ has the same graph as $\phi_{k_l}$, we have
  \[\rho(\psi_{k_l},\phi^\star) = \rho(\phi_{k_l},\phi^\star)\leqslant \max\limits_{\alpha\in [0,1]}\abs{\phi_{k_l}(\alpha)-\phi^\star(\alpha)}\to 0,\,\quad l\to +\infty.\]

(5). By Taylor expansion of $U(\phi^\star)$ near $\vec{x}_c$, we get
  \begin{align}\label{eq:taylorstar}
    E^\star-U(\phi^\star(\alpha)) = & -\frac{1}{2}[\phi^\star(\alpha)-\phi^\star(\alpha_c)]^\T\nabla^2 U(\vec{x}_c)[\phi^\star(\alpha)-\phi^\star(\alpha_c)]\\
    & + o(\abs{\phi^\star(\alpha)-\phi^\star(\alpha_c)}^2).\nonumber
  \end{align}
By Assumption \ref{asm:nondegenerate},  we obtain
\[E^\star-U(\phi^\star(\alpha))\leqslant\frac{\mu_M}{2}\abs{\phi^\star(\alpha)-\phi^\star(\alpha_c)}^2 + o(\abs{\alpha-\alpha_c})^2\leqslant C\abs{\alpha-\alpha_c}^2\]
for some constant $C>0$ in a neighborhood of $\alpha_c$, where $\mu_M$ is the largest eigenvalue of $-\nabla^2 U(\vec{x}_c)$. The integral
\[\int_{\alpha_c-\delta}^{\alpha_c+\delta}\frac{\abs{\phi{^\star}{'}}\d\alpha}{\sqrt{2E^\star-2U(\phi^\star)}}\geqslant\int_{\alpha_c-\delta}^{\alpha_c+\delta}\frac{\abs{\vec{x}_s-\vec{x}_f}\d\alpha}{\sqrt{2C}\abs{\alpha-\alpha_c}}=+\infty.\]
This ends the proof of the last statement.
\end{proof}

Theorem \ref{prop:conv} shows that, up to a subsequence, the limiting
transition energy $E^\star$ and the graph limit of the OM minimizer $\phi^\star$
are well-defined.
Next we show that the graph limit $\phi^\star$ is an extremal of $\hS{E^\star}[\phi]$.
\begin{theorem}\label{thm:extremal}
  $\phi^\star$ is an extremal of $\hS{E^\star}[\phi]$ passing through a critical point with $E^\star=\max_{\vec{x}\in\mathbb{R}^d}U(\vec{x})$.
\end{theorem}

The proof of Theorem \ref{thm:extremal} relies on the following result,
whose proof is rather technical thus will be carried out in the Appendix.

\begin{proposition}\label{prop:condev}
There exists a subsequence $\left\{\phi_{k_l}\right\}$
such that $\phi'_{k_l}\to \phi^{\star}{'}$ almost everywhere.
\end{proposition}

\textit{Proof of Theorem \ref{thm:extremal}.}
 By Theorem \ref{prop:conv} and Proposition \ref{prop:condev}, there exists
a subsequence $\{\phi_{k_l}\}$ which satisfies (1) $\phi_{k_l}\to \phi^\star$ uniformly; (2) $\phi'_{k_l}\to \phi{^\star}{'} a.e.$; (3) $E_{k_l}\to E^\star$.

Since $\phi_{k_l}$ is an extremal of $\hS{E_{k_l}}[\phi]$,  it satisfies the Euler-Lagrange equation \eqref{eq:weakEL}. We have
\begin{equation}\label{eq:Ellim}
  \lim\limits_{l\to\infty}\int_{0}^{1}\left(
\sqrt{2K_{E_{k_l}}(\phi_{k_l})}\frac{\phi_{k_l}'\cdot\Phi'}{\abs{\phi_{k_l}'}} - \frac{\abs{\phi_{k_l}'}\nabla U(\phi_{k_l})\cdot\Phi}{\sqrt{2K_{E_{k_l}}(\phi_{k_l})}}\right)
\d\alpha=0,\quad \forall \Phi\in C_0^{\infty}[0,1].
\end{equation}
From $\phi_{k_l}\in \bar{C}_M[0,1]$, we have $\sqrt{2K_{E_{k_l}}(\phi_{k_l})}\leqslant M_1$ for some constant $M_1$, thus
\[\sqrt{2K_{E_{k_l}}(\phi_{k_l})}\frac{\abs{\phi_{k_l}'\cdot\Phi'}}{\abs{\phi_{k_l}'}}\leqslant M_1\abs{\Phi'}.\]
It follows from the dominated convergence theorem that
\begin{equation}\label{eq:boundlhs}
  \lim\limits_{l\to\infty}\int_{0}^{1}\sqrt{2K_{E_{k_l}}(\phi_{k_l})}\frac{\phi_{k_l}'\cdot\Phi'}{\abs{\phi_{k_l}'}}\d\alpha=
  \int_{0}^{1}\sqrt{2K_{E^\star}(\phi^\star)}\frac{\phi{^\star}{'}\cdot\Phi'}{\abs{\phi{^\star}{'}}}\d\alpha.
\end{equation}

By Theorem \ref{prop:conv}, $\phi^\star$ passes through a critical point $\vec{x}_c$ at $\alpha=\alpha_c$. The Taylor expansion \eqref{eq:taylorstar} shows
  \[E^\star-U(\phi^\star(\alpha)) \geqslant \frac{\mu_m}{2}\abs{\phi^\star(\alpha)-\phi^\star(\alpha_c)}^2 + o(\abs{\phi^\star(\alpha)-\phi^\star(\alpha_c)}^2), \]
  where $\mu_m$ is the smallest eigenvalue of $-\nabla^2 U(\vec{x}_c)$.
By Assumption \ref{asm:nondegenerate},  there is a constant $C_1>0$
such that $E^\star-U(\phi^\star)\geqslant C_1\abs{\phi^\star(\alpha)-\phi^\star(\alpha_c)}^2$ in a neighborhood of $\alpha_c$.
Furthermore, we have
  \begin{equation*}
  \begin{aligned}
    \abs{\nabla U(\phi^\star)} &= \abs{\nabla^2 U(\vec{x}_c)(\phi^\star(\alpha)-\phi^\star(\alpha_c))} + o(\abs{\phi^\star(\alpha)-\phi^\star(\alpha_c)})\\
    & \leqslant \mu_M\abs{\phi^\star(\alpha)-\phi^\star(\alpha_c)} + o(\abs{\phi^\star(\alpha)-\phi^\star(\alpha_c)}).
  \end{aligned}
  \end{equation*}
There exists a constant $C_2>0$ such that
$\abs{\nabla U(\phi^\star)\cdot\Phi}\leqslant C_2
\abs{\Phi}\abs{\phi^\star(\alpha)-\phi^\star(\alpha_c)}$.
Combining the above results, we obtain
  \[\frac{\nabla U(\phi^\star)\cdot\Phi}{\sqrt{2E^\star-2U(\phi^\star)}}\leqslant C\abs{\Phi},\]
where the right-hand side is a bounded function.

  For any $l>0$ and given $\Phi(\alpha)$, define
  \[G_{k_l}(\alpha) = \frac{\nabla U(\phi_{k_l})\cdot\Phi}{\sqrt{2E_{k_l}-2U(\phi_{k_l})}}\quad \mbox{and}\quad G^\star(\alpha) = \frac{\nabla U(\phi^\star)\cdot\Phi}{\sqrt{2E^\star-2U(\phi^\star)}}.\]
  We know that $G_{k_l}\to G^\star$ and $G^\star$ is bounded by a constant $C_0>0$.  For any $n>C_0$, let $G_{k_l,n}=\min\{n,G_{k_l}\}$. We have
  \[\lim\limits_{n\to +\infty}G_{k_l,n}=G_{k_l},\quad \lim\limits_{l\to +\infty}G_{k_l,n}=G^\star \quad \mbox{uniformly for } n>C_0.\]
  For a fixed $n$, $G_{k_l,n}$ is bounded. By the dominated convergence theorem,
  \[\lim\limits_{l\to +\infty}\int_{0}^{1}G_{k_l,n}(\alpha)\d\alpha = \int_{0}^{1}G^\star(\alpha)\d\alpha.\]
  Since when $l\to\infty$, $G_{k_l,n}\to G^\star$ uniformly, we have
  \begin{equation}\label{eq:conv}
  \begin{aligned}
    \lim\limits_{l\to +\infty}\int_{0}^{1}\frac{\nabla U(\phi_{k_l})\cdot\Phi}
{\sqrt{2E_k-2U(\phi_k)}} d\alpha& = \lim\limits_{l\to +\infty}\lim\limits_{n\to +\infty}\int_{0}^{1}G_{k_l,n}\d\alpha\\
    & = \lim\limits_{n\to +\infty}\lim\limits_{l\to +\infty}\int_{0}^{1}G_{k_l,n}\d\alpha \\
    & = \int_{0}^{1}G^\star(\alpha)\d\alpha=\int_{0}^{1}\frac{\nabla U(\phi^\star)\cdot\Phi}{\sqrt{2E^\star-2U(\phi^\star)}}\d\alpha.
  \end{aligned}
  \end{equation}

  Because $\abs{\phi'_k}$ is constant and $\abs{\phi'_{k_l}}\to \abs{\phi^\star{'}}$, together with \eqref{eq:Ellim} and \eqref{eq:boundlhs},
  we see  that $\phi^\star$ satisfies the Euler-Lagrange equation \eqref{eq:weakEL} with energy parameter $E=E^\star$. Thus $\phi^\star$ is an extremal of $\hS{E^\star}[\phi]$.
\endproof

\begin{remark}
The weak convergence argument for $\varphi'_k$ is not enough to prove
the theorem due to the existence of the term $\abs{\phi'_k}$.
Although we have the uniform $L^\infty$ bound for both, which guarantees
$\varphi'_k\rightarrow f$ and $|\varphi'_k|\rightarrow g$,
we do not have $g=\abs{f}$ in general.
\end{remark}

Theorem \ref{thm:extremal} shows that the graph limit of minimizers of
$S^{\OM}_T[\psi]$ can be found by solving the Euler-Lagrange equation
of $\hS{E}[\phi]$ with the energy $E^\star=\max_{\vec{x}}U(\vec{x})$.
The extremal of $\hS{E^\star}[\phi]$ always exists, but may not be unique.

Theorem \ref{prop:conv} shows that it takes infinite time
for MPP to pass through a critical point. In general, the time that the MPP
spends in any interval is characterized by the function
\begin{equation}
\lambda^\star=\frac{\sqrt{2E^\star-2U(\phi^\star)}}{\abs{\phi{^\star}'}}.
\end{equation}

\begin{proposition}\label{prop:deltaT}
Let $[\alpha_1, \alpha_2]\subset [0,1]$ be an interval such that  $\lambda^\star(\alpha)>0$ for any $\alpha\in [\alpha_1,\alpha_2]$. Then the time that the MPP $\phi^\star$
takes to go from $\vec{x}_1=\phi^\star(\alpha_1)$ to $\vec{x}_2=\phi^\star(\alpha_2)$ is $\Delta T = \int_{\alpha_1}^{\alpha_2}\lambda{^\star}^{-1}\d\alpha$.
\end{proposition}
\begin{proof}
By Theorem \ref{prop:conv} and Proposition \ref{prop:condev}, we may assume $\phi_k\to\phi^\star$, $\phi_k'\to \phi{^\star}{'}$ without loss of generality. For any $T_k$,  the relation between $\phi_k$ and the minimizer $\psi_k$ of $S_{T_k}[\psi]$ is
\begin{equation}\label{eq:repa}
  \phi_k(\alpha) = \psi_k(\ell_k^{-1}(\alpha)),\quad \text{ where }\ell_k(t) = L_k^{-1}\int_{0}^{t}|\dot{\psi}_k|\d t,\ L_k = \int_{0}^{T_k}|\dot{\psi}_k|\d t.
\end{equation}
We have $\psi_k(t) = \phi_k(\ell_k(t))$.
Eq.~\eqref{eq:conservation} shows that the time spent by $\phi_k$  from $\vec{x}_1^k=\phi_k(\alpha_1)$ to $\vec{x}_2^k=\phi_k(\alpha_2)$ is
  \begin{equation}\label{eq:deltaTk}
    \Delta T_k = \int_{t_1}^{t_2}\frac{|\dot{\psi}_k|\d t}{\sqrt{2E_k-2U(\psi_k)}} = \int_{\alpha_1}^{\alpha_2}\frac{\abs{\phi'_k}\d \alpha}{\sqrt{2E_k-2U(\phi_k)}},
  \end{equation}
  where $t_1 = \ell_k^{-1}(\alpha_1)$, $t_2=\ell_k^{-1}(\alpha_2)$. Let $k\to \infty$. We obtain
  \begin{equation}\label{eq:DeltaT}
    \Delta T = \lim\limits_{k\to\infty}\Delta T_k = \lim\limits_{k\to\infty}\int_{\alpha_1}^{\alpha_2}\frac{\abs{\phi'_k}\d \alpha}{\sqrt{2E_k-2U(\phi_k)}} = \int_{\alpha_1}^{\alpha_2}\frac{\d\alpha}{\lambda^\star}.
  \end{equation}
\end{proof}

Proposition \ref{prop:deltaT} and Theorem \ref{prop:conv} show that $\lambda^\star(\alpha)$
characterizes the time that the MPP spends at any point $\alpha$.
It takes longer time to pass through points with smaller values of $\lambda^\star$.
Since $|\phi^\star{'}|$ is a constant, the local minima of $\lambda^\star$
correspond to the local maxima of $U(\phi^\star)$ along the MPP.
These states will be referred to as
the \textit{$\lambda$-critical state}.
Note that it does not need to take infinite time to pass through such states.

\begin{remark}
Theorem \ref{prop:conv} can be readily generalized to non-gradient systems.
However, the proof of Proposition \ref{prop:condev} highly relies
on the gradient nature of the system.
\end{remark}

\section{Comparison with the Freidlin-Wentzell functional}\label{sec:Compare}

In this section, we make comparisons between the double minimization problems
of the OM and FW functionals. We first summarize the differences
between these two problems and then give an illustrative example.

\subsection{Differences between the FW and OM functionals}

 When $\epsilon=0$, the OM functional reduces to the FW functional.
The {double minimization problem of the FW functional, which is stated as}
\[S^{\FW}(\vec{x}_s,\vec{x}_f) =
\inf\limits_{T>0}\inf\limits_{\psi\in \bar{C}_{\vec{x}_s}^{\vec{x}_f}[0,T]}S_T^{\FW}[\psi],\]
was studied in \cite{heymann2008geometric, Lv2014, zhou2016construction}.
The main results are summarized as follows:
\begin{enumerate}\label{summaryOM}
  \item The infimum occurs at some $T=T^\star$.
The MPP may not pass through any critical point.
$T^\star= +\infty$ if and only if the MPP passes through a critical point.
  \item $S^{\FW}(\vec{x}_s,\vec{x}_f)$, which is called
the local quasi-potential starting from $\vec{x}_s$ if $\vec{x}_s$ is a steady state,
is positive and finite in general.
  \item The corresponding energy parameter $E(T^\star)=0$
regardless of the value of $T^\star$ (finite or infinite).
The value of $S^{\FW}(\vec{x}_s,\vec{x}_f)$ can be obtained by minimizing the geometric functional $\hS{0}[\phi]$, i.e.
      \[S^{\FW}(\vec{x}_s,\vec{x}_f) = \inf\limits_{\phi\in \bar{C}_{\vec{x}_s}^{\vec{x}_f}}\hS{0}[\phi].\]
We note that for a prescribed value of $T$, the corresponding
energy $E$ may assume non-zero value.

  \item Under Assumption \ref{asm:uniformlybounded},
the minimizer of $S_T^{\FW}[\psi]$
has a subsequence that converges to the minimizer of $\hS{0}[\phi]$
as $T\rightarrow T^\star$ in the sense that $\rho(\psi_T,\phi_0)\to 0$, where $\phi_0$ is the minimizer of $\hS{0}[\phi]$.
\end{enumerate}

Our results in previous sections generalize the above conclusions for the FW functional
to  the OM functional. In Proposition \ref{prop:D}, we showed that
the double minimization of the OM functional always occurs at $T^\star=+\infty$.
Even when neither $\vec{x}_s$ nor $\vec{x}_f$ is critical, the minimizer
must pass through a critical point. The corresponding value of the OM functional
$S^{\OM}(\vec{x}_s,\vec{x}_f)$ is always $-\infty$ when $\epsilon>0$.
This implies that we cannot define an analog of the quasi-potential
as in the Freidlin-Wentzell theory via the minimization of the OM functional.
In spite of that, we can still study the graph limit of the minimizers of $S^{\OM}_T[\psi]$
in the configuration space with the help of the Maupertuis principle.
Theorem \ref{prop:conv} showed that the corresponding energy parameter $E(T)$
converges to $E^\star=\max_{\vec{x}\in\mathbb{R}^d}U(\vec{x})$.
The minimizer of $S^{\OM}_T[\psi]$ has a subsequence that converges
to an extremal of $\hS{E^\star}[\phi]$. However, the extremal may
not be a minimizer of $\hS{E^\star}[\phi]$.
Note that when $\epsilon\to 0$, $\max_{\vec{x}\in \mathbb{R}^d} U(\vec{x})\to 0$,
then formally $E(T)\to 0$ as $T\to +\infty$, which is consistent
with the Freidlin-Wentzell theory. The convergence from the OM functional to the FW
functional was rigorously studied using the $\Gamma$-convergence technique
in \cite{pinski2012gamma}.

In summary, the main differences between the double minimization of the FW and
OM functionals are: (1) $S^{\FW}(\vec{x}_s,\vec{x}_f)$ is finite while $S^{\OM}(\vec{x}_s,\vec{x}_f)$ is $-\infty$; (2) The minimizer of the FW functional has the energy
$E^\star=0$ while the minimizer of the OM functional has a positive energy;
(3) The graph limit of the minimizer of the OM functional is an extremal
rather than a minimizer of $\hS{E^\star}[\phi]$.
The second point implies that in the numerical methods, one has to identify the
MPP together with the unknown energy $E^\star$ at the same time. This is pursued in Section \ref{sec:Method}.

\subsection{An illustrative example}\label{sec:Exam}
In this section, we present a simple but informative example to
illustrate the main results derived earlier for the OM and FW functionals.
We consider a 2-D diffusion process with drift $\vec{b}(\vec{x}) = -\nabla V(\vec{x})$,
where $V(\vec{x}) = \frac{1}{2}|\vec{x}|^2$ is a quadratic potential.
The path potential $U(\vec{x}) = 2\epsilon -\frac{1}{2}|\vec{x}|^2$
has only one critical point $\vec{x}_c=0$, which is also the steady state
of the dynamics $\dot{\vec{x}}=-\nabla V(\vec{x})$.

 For any pair of states $\vec{x}_s$, $\vec{x}_f$ and fixed $T>0$,
the Euler-Lagrange equation of the OM functional is given by
\begin{equation}\label{eq:EL2d}
  \ddot{\psi}_T(t) - \psi_T(t) = 0,\quad t\in (0,T)
\end{equation}
with the boundary conditions $\psi_T(0)=\vec{x}_s $ and $\psi_T(T) = \vec{x}_f$.
The solution of Eq.~\eqref{eq:EL2d} is given by
\begin{equation}\label{eq:psiT}
  \psi_T(t) = \vec{A}\me^t + \vec{B}\me^{-t},
\end{equation}
where the vectors $\vec{A}$ and $\vec{B}$ are given by
\begin{equation}\label{eq:AB}
\vec{A} = (\me^T-\me^{-T})^{-1}(\vec{x}_f - \vec{x}_s  \me^{-T}),
\quad \vec{B} = (\me^T-\me^{-T})^{-1}(\vec{x}_s  \me^T - \vec{x}_f).
\end{equation}
They can be viewed as functions of $T$. Because $\nabla^2_{\vec{y}}L=I$ is positive definite, $\psi_T$ satisfies strict Legendre-Hadamard condition. So $\psi_T$ is the unique minimizer of $S^{\OM}_T[\psi]$. Since $\psi_T$ is independent of $\epsilon$, it is also the minimizer of the FW functional $S_T^{\FW}[\psi]$.

Along the path $\psi_T$, the conservation of energy \eqref{eq:conservation} gives
\[\frac{1}{2}|\dot{\psi}_T|^2 + U(\psi_T) =2\epsilon -2\vec{A}\cdot\vec{B},\]
thus the energy $E=E(T) = 2\epsilon -2\vec{A}\cdot\vec{B}$.

The value of $S^{\OM}_T[\psi_T]$ can be also calculated as
\begin{align}\label{eq:ST}
  S^{\OM}_T[\psi_T] &= \int_{0}^{T}\Big(\frac{1}{2}|\dot{\psi}_T|^2 + \psi_T\cdot\dot{\psi}_T + \frac12\abs{\psi_T}^2\Big)\d t -2\epsilon T\\
              &= \frac{1}{2}|\vec{x}_f|^2 - \frac{1}{2}|\vec{x}_s  |^2 - 2\epsilon T + \frac{1}{2}\abs{\vec{A}}^2(\me^{2T}-1) + \frac{1}{2}\abs{\vec{B}}^2(1-\me^{-2T}).\nonumber
\end{align}
It is straightforward to check that
\begin{equation}\label{eq:dST}
\pp{S^{\OM}_T[\psi_T]}{T} = 2\vec{A}\cdot\vec{B} - 2\epsilon=-E(T).
\end{equation}

In what follows we consider two special choices of $\vec{x}_s$ and $\vec{x}_f$.
Let $\vec{x}_s  =(R_1\cos\theta_1,$ $R_1\sin\theta_1)$ and
$\vec{x}_f=(R_2\cos\theta_2, R_2 \sin\theta_2)$.
We first consider the case where $\theta_1=\theta_2=0$ and $R_2>R_1>0$, i.e.
the two states $\vec{x}_s$ and $\vec{x}_f$ are on the positive $x$-axis on the $x$-$y$ plane,
and the initial state $\vec{x}_s$ is closer to the origin.
In the second case, we set $R_1=R_2>0$ and $0<\theta_1<\theta_2<\pi/2$ where
the two states $\vec{x}_s$ and $\vec{x}_f$ are on the same countor line of the
energy $V(\vec{x})$.
In both cases, we examine the minimizers of the OM and FW functionals
in detail.

\subsection*{Case 1: $\theta_1=\theta_2=0$ and $R_2>R_1>0$}

Let $\vec{x}_s   = (x_s,0)$ and $\vec{x}_f = (x_f,0)$,
where $0<x_s<x_f$. In this case, the vectors $\vec{A}=(A,0)$, $\vec{B}=(B,0)$
in \eqref{eq:psiT}, thus this is essentially a one-dimensional transition on the $x$-axis.
For simplicity we will use $\psi_T$ to denote the $x$ coordinate of the transition
path only.

\cref{fig:1d} (left panel) shows the graph of $\partial S^{\OM}_T[\psi_T]/\partial T = 2AB-2\epsilon$
(or $-E(T)$) versus $T$. Denote the two zeros of  $\partial S^{\OM}_T[\psi_T]/\partial T$
by $T_{\min}$ and $T_{\max}$, respectively. Then it is easily seen that
$S^{\OM}_T[\psi_T]$ attains a local minimum at $T=T_{\min}$ and
a local maximum at $T=T_{\max}$.
As $T\rightarrow +\infty$, it follows from \eqref{eq:AB} that
$\partial S_T^{\OM}[\psi_T]/\partial T \rightarrow -2\epsilon$.
Thus for sufficiently large $T$, we have
$$\pp{S_T^{\OM}[\psi_T]}{T} < -\epsilon<0, $$
and consequently, $\lim_{T\to\infty}S_T^{\OM}[\psi_T]= -\infty$,
which agrees with the results in Proposition \ref{prop:D}.

Next we examine the transition path $\psi_T$ and its graph in the configuration space
in detail.
Note that we have $A>0$ for any $T>0$ and $\dot{\psi}_T = A\me^t-B\me^{-t}$.
To see whether there exists $t\in [0,T]$ such that $\dot\psi(t)<0$, we define
$$f(T)=\frac{B}{A}=\frac{x_s \me^T-x_f}{x_f-x_s\me^{-T}}.$$
Since $f(T)$ is an increasing function for $T>0$,
there exist unique $T_b> T_a>0$ such that $f(T_a)=0$ and $f(T_b)=1$.
A straightforward calculation shows that $\partial S_T^{\OM}[\psi_T]/\partial T$ attains
the maximum at $T=T_b$ as shown in \cref{fig:1d}.
We consider the following two cases:

\vspace{0.3cm}
\noindent 1) $T \le T_b$. In this case, $B\le A$. Thus $\dot{\psi}_T\ge 0$
for all $t\in [0, T]$, which implies the system moves  from the initial state
$\vec{x}_s$  towards the final state $\vec{x}_f$ along the path $\psi_T(t)$.
The graph of the path is simply the line segment connecting $\vec{x}_s$ and $\vec{x}_f$.
Using the normalized arc-length as the parameterization, the path can be expressed
as
 \begin{equation}\label{eq:quadmin1}
 \phi_E(\alpha) = x_s  + L\alpha,
 \end{equation}
where $0\le \alpha \le 1$ and $L = x_f -x_s$. Note that although different
choices of $T$ correspond to different minimizers $\psi_T$,
they have the same graph in the configuration space.

\vspace{0.3cm}
\noindent 2) $T > T_b$. In this case, we have $B> A$ and
there exists a unique $t^\star = \ln \sqrt{B/A}\in [0,T]$
such that $\dot{\psi}_T(t^\star)=0$.
Furthermore we have $\dot{\psi}_T<0$ when $t\in [0, t^{\star})$ and
$\dot{\psi}_T>0$ when $t\in (t^{\star}, T]$.
It follows that, along the transition path, the system moves towards the origin
until $t=t^\star$ then it turns back and moves towards the final state $\vec{x}_f$.
The turning point is given by
$x_c = \psi_T(t^\star) = 2\sqrt{AB}<x_s $.

The mapping from $\psi_T$ on $[0,T]$ to $\phi_E$ on $[0,1]$ is done as follows.
We define the normalized arc-length parameterization as
\begin{equation}\label{eq:alphat}
  \alpha=\ell(t) = \begin{cases}
                L^{-1}(x_s  - \psi_T(t)), &  0\leqslant t\leqslant t^\star \\
                L^{-1}(x_s  - 2x_c + \psi_T(t)), & t^\star\leqslant t\leqslant T,
              \end{cases}
\end{equation}
where $L$ is the arc-length of the path
\[L = \int_{0}^{T}|\dot{\psi}_T|\d t = -\int_{0}^{t^\star}\dot{\psi}_T\d t + \int_{t^\star}^{T}\dot{\psi}_T\d t = x_f+x_s  -2x_c.\]
Then $\psi_T(t)$ is mapped to
\begin{equation}\label{eq:phi}
  \phi_E(\alpha) = \begin{cases}
                x_s  - L\alpha, &  0\leqslant \alpha\leqslant \alpha^\star \\
                2x_c - x_s  + L\alpha, & \alpha^\star\leqslant \alpha\leqslant 1,
              \end{cases}
\end{equation}
where $\alpha^\star = \ell(t^\star) = L^{-1}(x_s -x_c)$.

The graph $\phi_E$ in \eqref{eq:phi} depends on $T$. In particular,
the turning point $x_c$ is a function of $T$,
$$x_c = 2\sqrt{AB}=(2x_s x_f\cosh T -x_s ^2-x_f^2)^{1/2}/\sinh T .$$
It is easilly seen that as $T\to +\infty$, $\vec{x}_c\to 0$,
 i.e. the turning point converges to the critical point.
This shows that the graph limit of the MPP will pass through
a critical point even when $\vec{x}_s$ and $\vec{x}_f$ are not
at the steady state of $\dot{\vec{x}}=-\nabla V(\vec{x})$.
Furthermore, the energy $E(T)\rightarrow E^\star=2\epsilon$ as $T\rightarrow +\infty$.
In this limit, $\phi_E$ converges to
\begin{equation}\label{eq:phistar}
\phi_{E^\star}(\alpha)  = \begin{cases}
                x_s  - L\alpha, &  0\leqslant \alpha\leqslant (x_f+x_s )^{-1}x_s, \\
                L\alpha - x_s , & (x_f+x_s )^{-1}x_s \leqslant \alpha\leqslant 1.
              \end{cases}
\end{equation}

The geometric functional corresponding to $E^\star=2\epsilon$ is given by
$$\hS{E^\star}[\phi] = \int_{0}^{1}|\phi|\abs{\phi'}\d\alpha.$$
The corresponding strong form of the Euler-Lagrange equation is
\begin{equation}\label{eq:ELE-Exam}
|\phi|^2\phi''+(\phi\cdot \phi')\phi'-|\phi'|^2\phi=0.
\end{equation}
It is easy to see that the graph limit $\phi_{E^\star}$
with the turning point $\vec{x}_c=0$ is a weak solution
of \eqref{eq:ELE-Exam}, thus an extremal of the functional $\hS{E^\star}[\phi]$.
This verifies the result obtained in Theorem \ref{thm:extremal}.
To further check whether $\phi_{E^\star}$ is the global minimizer
of $\hS{E^\star}[\phi]$, we compare it with the following function
$$\tilde\phi(\alpha)=x_s+\tilde{L}\alpha,\quad \alpha\in [0,1], $$
where $\tilde{L}=x_f-x_s$. This is the graph of the path obtained when $T\le T_b$.
A direct calculation yields
$$\hS{E^\star}[\tilde\phi]=\frac12(x_f^2-x_s^2).$$
In comparison,
$$\hS{E^\star}[\phi_{E^\star}]=\frac12(x_f^2+x_s^2) > \hS{E^\star}[\tilde\phi].$$
Thus $\phi_{E^\star}$ is only an extremal of $\hS{E^\star}[\phi]$,
but not its global minimizer.

Next we compare the results obtained above for the OM functional
with the FW functional.
For the FW functional, since $\epsilon=0$, we have
$\partial S^{\FW}_T[\psi_T]/\partial T=2AB$. The right panel of \cref{fig:1d}
shows the graph of  $\partial S^{\FW}_T[\psi_T]/\partial T$ versus $T$.
The only zero of  $\partial S^{\FW}_T[\psi_T]/\partial T = 0$ occurs at $T=T_a$,
at which $S^{\FW}_T[\psi_T]$ attains the global minimum. The graph of the
corresponding transition path is the one in \eqref{eq:quadmin1}, which
has the action
$$S^{\FW}_{T_a}[\psi_{T_a}]=x_f^2-x_s ^2 = 2V(\vec{x}_f)-2V(\vec{x}_s  ).$$
In comparison, when $T\to\infty$, the graph of the transition path converges
to the one given in \eqref{eq:phistar} with the action
\[\lim\limits_{T\to +\infty}S_T^{\FW}[\psi_T] = x_f^2=2V(\vec{x}_f)-2V(\vec{0}),\]
which is obviously larger than $S^{\FW}_{T_a}[\psi_{T_a}]$.
\begin{figure}
  \centering
  \scalebox{0.45}{\includegraphics[angle=90]{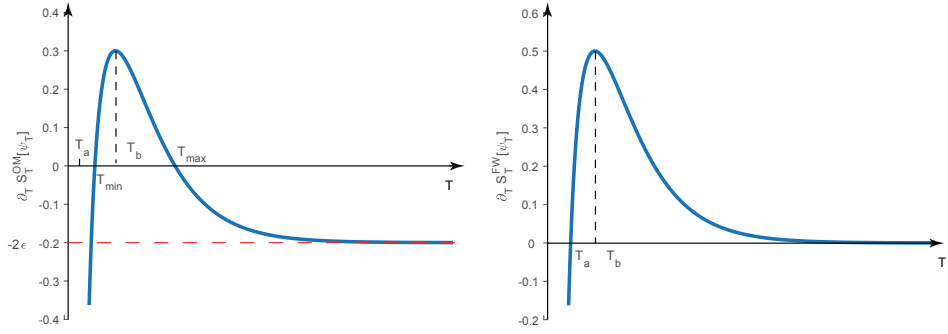}}
\caption{
The plot of  $\partial S_T[\psi_T]/\partial T$ (or $-E(T)$)
versus $T$ for the Onsager-Machlup  functional (left panel)
and the Freidlin-Wentzell (right panel), where $x_s =1, x_f=2, \epsilon = 0.1$.
$T_a<T_b$ satisfy $f(T_a)=0, f(T_b)=1$. $T_{\min}$ and $T_{\max}$ are two zeros of $\partial S_T^{\OM}[\psi_T]/\partial T$.
}\label{fig:1d}
\end{figure}

\subsection*{Case 2: $R_1=R_2=R>0$ and $0<\theta_1<\theta_2<\pi/2$}
In this case, we illustrate the different transition behaviors modeled by
the OM and FW functionals between the two states
$\vec{x}_s  =(R\cos\theta_1, R\sin\theta_1)$, $\vec{x}_f=(R\cos\theta_2, R\sin\theta_2)$.
The two states lie on the same contour line of $V(\vec{x})$ with
  $V(\vec{x}_s  )=V(\vec{x}_f)=R^2/2$.

Let us consider the OM functional first.
As in case 1, the fact that
$$\pp{S^{\OM}_T[\psi_T]}{T}=-E\to -2\epsilon<0$$
as $T \to +\infty$ implies that the infimum of $S^{\OM}_T[\psi_T]$ occurs when $T\to +\infty$.
To examine if the transition path passes through the origin, we consider
the distance of the minimizer $\psi_T$ to the origin. For any given $T$, we have
\[\dd{\,}{t} |\psi_T|^2= 2\abs{\vec{A}}^2\me^{2t} - 2\abs{\vec{B}}^2\me^{-2t}.\]
The solution of $\d |\psi_T|^2/\d t=0$ is $t^\star=\ln(|\vec{B}|/|\vec{A}|)/2>0$.
In addition, we have $\d^2|\psi_T|^2/\d t^2>0$ for all $t$, therefore,
$|\psi_T|$ attains its minimum at $t=t^\star$, i.e.
$\psi_T(t^\star)$ is the state along the path which is the closest to the origin.
This shortest distance is given by
\begin{equation}\label{eq:dist}
|\psi_T|_{\min}^2=|\psi_T|^2(t^\star)=2|\vec{A}||\vec{B}|+2\vec{A}\cdot\vec{B},
\end{equation}
which converges to 0 as $T\rightarrow \infty$.
Therefore, the graph limit of the transition path
passes through the critical point $\vec{x}=0$.
A typical path is shown in  \cref{fig:2D}.

For the FW functional, the minimum of $S^{\FW}_T[\psi_T]$ is attained at
$T=T_c$, the solution to  $\partial S^{\FW}_T[\psi_T]/\partial T=0$, or explicitly,
\[
T_c =\cosh^{-1}\left(\frac{1}{\cos(\theta_2-\theta_1)}\right).
\]
The corresponding action is given by
$S^{\FW}_{T_c}[\psi_{T_c}]=4R^2\sin(\theta_2-\theta_1)$.
The shortest distance from the path $\psi_{T_c}$ to the origin can be computed
using \eqref{eq:dist}, which yields
\[\abs{\psi_{T_c}}^2_{\min} =  R^2\cos(\theta_2-\theta_1)\leqslant R^2.\]
This shows that the minimizer of the FW functional, $\psi_{T_c}$,
has a positive distance from the origin, and the distance is smaller than $R$.
A typical graph of the transition path is shown in  \cref{fig:2D}.
Note that it is not along the contour of $V(\vec{x})$.

\begin{figure}
  \centering
  \includegraphics[width=7cm]{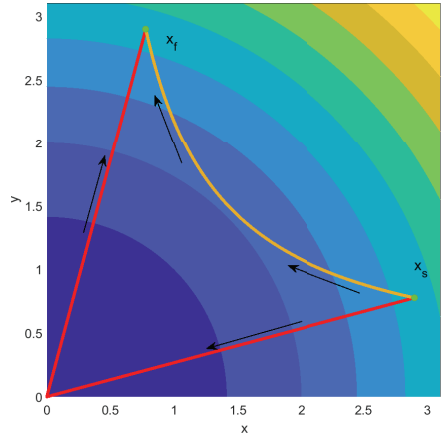}
  \caption{(Color Online). The minimizer of the FW and OM functionals
from $\vec{x}_s$ to $\vec{x}_f$, where $R=3$, $\theta_1=\pi/12$, $\theta_2=5\pi/12$, $\epsilon =0.1$. Yellow curve: the minimizer of the FW functional.
Red line: the minimizer of the OM functional. The background colors show the contour of $V(\vec{x})$.}\label{fig:2D}
\end{figure}

\section{Numerical method}\label{sec:Method}

Based on the insights gained from the theoretical analysis in previous sections,
we design an  energy-climbing geometric minimization method
to compute the graph limit of the minimizer of the OM
functional when $T$ goes to infinity.

The graph limit of the minimizer satisfies the Euler-Lagrange
equation \eqref{eq:EL0}, with the kinetic energy $K_E = E-U(\phi)$ and
$E =\max_\alpha U(\phi(\alpha))$. This is a highly nonlinear equation.
We propose to use the following steepest decent like dynamics
to find the solution:
\begin{align}\label{eq:relax}
\pp{\phi}{s} =2K(s, \alpha)\phi'' +\abs{\phi'}^2 (I-\hat\tau\otimes\hat\tau)\cdot\nabla U(\phi)+\lambda \hat\tau,\quad \alpha\in(0,1),\ s>0\nonumber\\
    \phi(s,0) = \vec{x}_s  ,\  \phi(s,1) = \vec{x}_f, \quad \text{and}\quad (\abs{\phi'})'=0
\end{align}
with initial condition $\phi(0,\alpha) = \phi^0(\alpha)$,
where $\lambda$ is a Lagrange multiplier to ensure the
constraint $(\abs{\phi'})'=0$, and
\begin{equation}
K(s, \alpha):=E(s)-U(\phi(s, \alpha))
\quad \text{and}\quad E(s) := \max_\alpha U(\phi(s,\alpha)).
\end{equation}
Here $s$ denotes the artificial relaxation time.

The well-posedness and long time behavior of Eq.~\eqref{eq:relax}
is not well-understood
yet due to the nonlinearity and degeneracy of $K(s, \alpha)$ at
$\alpha_c= \argmax_{\alpha} U(\phi)$. We will leave this problem to
our future studies.
Nevertheless, provided the dynamical system \eqref{eq:relax} reaches
a steady state as $s\to \infty$, the steady-state solution solves
 \eqref{eq:EL0}, thus gives the graph limit
of the minimizer of the OM functional.

We use \eqref{eq:relax} to construct numerical schemes.
We first semi-discretize Eq.~\eqref{eq:relax} with respect to the
relaxation time $s$. Using an explicit time-stepping with stepsize $\delta s$,
we get
\begin{align}\label{eq:semischeme}
  \phi^{n+1} = \phi^n + \delta s \Big(2K_n(\alpha) (\phi^n)''
+\abs{(\phi^n)'}^2 (I-\hat\tau^n\otimes\hat\tau^n)\cdot\nabla U(\phi^n)\Big)
+\lambda^{n+1} \hat\tau^{n+1},
\end{align}
where $\phi^n(\alpha)$ is the numerical approximation of $\phi(n\delta s,\alpha)$ and
\begin{equation}\label{eq:EnKn}
 K_n(\alpha):=E_n-U(\phi^n), \quad E_n := \max\limits_{\alpha}U(\phi^n(\alpha)).
\end{equation}

The scheme \eqref{eq:semischeme} is equivalent to first obtain  $\tilde\phi^{n+1}$ by
\begin{align}\label{eq:SchemeWithoutRep}
 \tilde\phi^{n+1} = \phi^n + \delta s \Big(2K_n(\alpha) (\phi^n)''
+\abs{(\phi^n)'}^2 (I-\hat\tau^n\otimes\hat\tau^n)\cdot\nabla U(\phi^n)\Big),
\end{align}
then get $\phi^{n+1}$ by reparameterizing  $\tilde\phi^{n+1}$ with equi-arclength condition.

For a given $\phi^0$, the scheme \eqref{eq:semischeme} generates a sequence
of paths $\{\phi^n\}$ and the corresponding energy  $\{E_n\}$.
We have the following theorem concerning the properties
of the above numerical scheme.

\begin{proposition}\label{prop:C}
  Suppose $\phi^0(\alpha)\in C^{\infty}[0,1]$, $U(\vec{x})\in C^{\infty}(\mathbb{R}^d)$,
 then $\phi^n(\alpha)\in C^{\infty}[0,1]$.
When $\delta s$ is sufficiently small, the sequence $\{E_n\}$
generated using \eqref{eq:semischeme} is nondecreasing.
Moreover, if $\phi^n$ converges to a steady state  $\phi^\star$, then $E_n\to E_c$, and $\phi^\star$ solves the Euler-Lagrange
equation \eqref{eq:EL0} with $E=E_c$. Here, $E_c=U(\phi^\star(\alpha^\star))$ for some $\alpha^\star\in [0,1]$. $\vec{x}^\star=\phi^\star(\alpha^\star)$ satisfies the condition $\nabla U(\vec{x}^\star)=0$.
\end{proposition}
\begin{proof}
Because $\phi^0$ and $U$ are smooth, we have $\phi^1\in C^{\infty}$
from \eqref{eq:semischeme}. The fact that $\phi^n\in C^{\infty}$ can be
deduced by induction.

Let $\alpha_n=\inf\{\alpha\in [0,1]|U(\phi^n(\alpha)) = \max_\alpha U(\phi^n(\alpha))\}$. We have $E_n = U(\phi^n(\alpha_n))$. We first assume $\alpha_n\in(0,1)$.
If $\nabla U(\phi^n(\alpha_n))=0$,  then there exists $\beta_n\in (0,1)$ such that
 $\phi^{n+1}(\beta_{n})=\tilde\phi^{n+1}(\alpha_{n})=\phi^n(\alpha_n)$ by \eqref{eq:semischeme} and \eqref{eq:SchemeWithoutRep}.

So $E_{n+1}\geqslant U(\phi^{n+1}(\beta_{n}))=E_n$. If $\nabla U(\phi^n(\alpha_n))\neq 0$, we have
  \begin{align}\label{eq:propC1}
    E_{n+1} & - E_n  \geqslant U(\tilde\phi^{n+1}(\alpha_{n})) - U(\phi^{n}(\alpha_{n}))\notag\\
          & = \nabla U(\phi^n(\alpha_n))\cdot(\tilde\phi^{n+1}(\alpha_n)-\phi^{n}(\alpha_n)) + O(\abs{\tilde\phi^{n+1}(\alpha_n)-\phi^n(\alpha_n)}^2).
  \end{align}
  When $\alpha = \alpha_n$, $\nabla U\cdot(\phi^n)'=0$. According to scheme \eqref{eq:SchemeWithoutRep}, $\tilde\phi^{n+1}(\alpha_n)-\phi^{n}(\alpha_n)=\delta s\abs{(\phi^n)'}^2\nabla U(\phi^n(\alpha_n))$, $\abs{\tilde\phi^{n+1}(\alpha_n)-\phi^n(\alpha_n)}=O(\delta s)$. From \eqref{eq:propC1}, we have
  \begin{equation}\label{eq:propC2}
    E_{n+1} - E_n \geqslant  \abs{(\phi^n)'}^2\abs{\nabla U(\phi^n(\alpha_n))}^2\delta s + O(\delta s^2)\geqslant 0
  \end{equation}
  when $\delta s^2$ is sufficiently small. So $\{E_k\}$ is nondecreasing.

  If $\alpha_n=0$, $E_{n+1} = \max U(\phi^{n+1}(\alpha))\geqslant U(\phi^{n+1}(0))=U(\vec{x}_s )=U(\phi^n(0))=E_n$. So $E_{n+1}\geqslant E_n$ still holds. The case $\alpha_n=1$ is similar. In all, we have shown that $\{E_k\}$ is nondecreasing.

  If the scheme reaches a steady state $\phi^\star$, then $E_n$ converges to some $E_c$ by the definition
$E_n = \max_\alpha U(\phi^n(\alpha))$ and the assumption that $\phi^n$ converges. The limit $\phi^\star$ solves
\begin{align}\label{eq:IterLimit}
2K_{E_c}(\phi^\star) (\phi^\star)''
+\abs{(\phi^\star)'}^2 (I-\hat\tau \otimes\hat\tau)\cdot\nabla U(\phi^\star)
+\lambda \hat\tau=0,
\end{align}
subject to the constraint $(|(\phi^\star)'|)'=0$, where
\begin{equation*}
K_{E_c}(\phi^\star):=E_c-U(\phi^\star), \quad E_c := \max\limits_{\alpha}U(\phi^\star(\alpha)).
\end{equation*}
Take inner product of both sides of \eqref{eq:IterLimit} with $\hat\tau$, we get the Lagrange multiplier $\lambda(\alpha)\equiv 0$ and thus $\phi^\star$ solves Eq. \eqref{eq:EL0} with $E=E_c$. Let $\alpha^\star=\arg\max\limits_{\alpha} U(\phi^\star(\alpha))$. From \eqref{eq:IterLimit}, we have $\nabla U(\phi^\star(\alpha^\star))=(\hat\tau\cdot\nabla U(\phi^\star(\alpha^\star)))\hat\tau =0$.
\end{proof}

\begin{remark}
A careful inspection of the proof shows that the properties on $E_n$ and $\phi^\star$ hold also for non-gradient systems as long as the path has second order differentiability.
\end{remark}

Let us re-examine the iteration at $\alpha = \alpha_n$.
Since $E_n = U(\phi^n(\alpha_n))$, we have $\nabla U\cdot(\phi^{n})'=0$, and the iteration in \eqref{eq:SchemeWithoutRep} reduces to $\tilde\phi^{n+1}(\alpha_n) = \phi^n(\alpha_n) + \delta s$ $\abs{(\phi^{n})'}^2\nabla U(\phi^n(\alpha_n))$. This can be viewed as a steepest ascent method to find the maximum of $U(\vec{x})$ with step size $\delta s\abs{(\phi^{n})'}^2$. The convergence result shows that $\phi^n(\alpha_n)$ converges to a point with $\nabla U=0$. So the iteration \eqref{eq:semischeme} can be considered as a combination of
a relaxation method to solve the Euler-Lagrange equation
and the steepest ascent method to find the maximum of $U(x)$ .

In practical computations, the path $\phi^n$ is discretized into a collection
of discrete states, $\phi^n_j=\phi^n(jh)$, where $h=1/N$ and $j=0, 1, \cdots, N$,
and the spatial derivatives of $\phi^n$ are discretized using the central difference
formula. This  gives the following algorithm:

\vspace{0.3cm}
\begin{algorithm}\label{alg:1}
\caption{Energy climbing geometric minimization algorithm (EGMA)}
\begin{enumerate}\label{alg:alg1}
\item Given $\phi^0_j=\phi^0(jh)$ for  $j=0,1,\dots N$ and $h=1/N$.  Set $n=0$.

\item For path $\phi^n_j$, let $E_n = \max_j U(\phi^n_j)$ and compute
\begin{align*}
& D\phi_j^n = (\phi_{j+1}^n-\phi_{j-1}^n)/(2h),\\
& D^2\phi_j^n=(\phi_{j+1}^{n} - 2\phi_j^n +\phi_{j-1}^n)/h^2,
\end{align*}
 for $j=1,2,\dots N-1$.

\item Let $\tilde{\phi}_0^{n+1} = \vec{x}_s  $,
$\tilde{\phi}_N^{n+1} =  \vec{x}_f$. Then compute
\begin{align*}
    \tilde{\phi}_j^{n+1} = & \phi_j^n + \delta s\Big[(2E_n-2U)D^2\phi_j^n
+ \abs{D\phi_j^n}^2\nabla U - (\nabla U\cdot D\phi_j^n)D\phi_j^n    \Big]
  \end{align*}
  for $j=1,2,\dots, N-1$. Here $U$ and $\nabla U$ are both evaluated at $\phi_j^n$.

\item Compute $\phi_j^{n+1}$ by interpolating $\{\tilde{\phi}_j^{n+1}\}$
based on the equi-arclength constraint.

\item  Terminate the iteration if $|\phi^{n+1}-\phi^n|/\delta s < TOL$, where $TOL$ is a prescribed tolerance.

\item Set $n: = n+1$ and goto step 2.
\end{enumerate}
\end{algorithm}

\vspace{0.3cm}
  It is worth noting that the explicit numerical scheme in the above algorithm can be
replaced by a semi-implicit scheme, in which the spatial derivative
$\phi''$ is evaluated at $(n+1)\delta s$:
$$\phi_j''\approx D^2\tilde{\phi}_j^{n+1}=
(\tilde{\phi}_{j+1}^{n+1}-2\tilde{\phi}_{j}^{n+1}+\tilde{\phi}_{j-1}^{n+1})/h^2.
$$
This will be a more stable scheme in practice.
At each step, this semi-implicit scheme requires solving a tri-diagonal linear system,
but allows a relatively large time stepsize.
Such linear systems can be solved by fast algorithms, e.g.
the Thomas algorithm, whose computational cost is comparable to the cost of the
explicit scheme.

The purpose of step (4) in EGMA is to redistribute the discrete
images so that they are equally spaced along the path.
This can be done using interpolation techniques
as introduced in the string method \cite{weinan2002string,weinan2007string}.
One simple strategy that uses the linear interpolation
is illustrated as follows.

\vspace{0.3cm}
\begin{algorithm}\label{alg:2}
\caption{Equal-arclength reparametrization}
\begin{enumerate}
  \item Let $L(0)=0$ and
  \[L(j) = \ssum{m=1}{j}\abs{\tilde{\phi}_j^{n+1} - \tilde{\phi}_{j-1}^{n+1}},\quad j=1,2,\dots, N.\]
  \item Compute the equally spaced arc-length parameter
  $l(j) = jL(N)/N$ for $j=0,1,\dots, N$.
  \item For each $j=1,2,\dots, N-1$, find $i \in\{0,1,\dots, N-1\}$
 such that $L(i)<l(j)\leqslant L(i+1)$, then compute $\phi_j^{n+1}$ as follows:
  \[\phi_j^{n+1} = \tilde{\phi}_i^{n+1} + (l(j)-L(i))\frac{\tilde{\phi}_{i+1}^{n+1}-\tilde{\phi}_i^{n+1}}{\abs{\tilde{\phi}_{i+1}^{n+1}-\tilde{\phi}_i^{n+1}}}.\]
\end{enumerate}
\end{algorithm}

\vspace{0.3cm}
The stopping criterion in  EGMA   is
based on the  rate of change of  $\phi^n$. In the computations below,
we set the tolerance $TOL =10^{-6}$.

\begin{remark}
 EGMA can also be applied to non-gradient type of systems
by solving the following  semi-discrete scheme
\begin{align}\label{eq:NGsemischeme}
  \phi^{n+1} = & \phi^n + \delta s \Big(2K_n(\phi^n)(\phi^n)'' + \sqrt{2K_n(\phi^n)}(\nabla\vec{b}^\T-\nabla\vec{b})\cdot(\phi^n)'\nonumber\\
   &+ \abs{(\phi^n)'}^2 (I-\hat\tau^n\otimes\hat\tau^n)\cdot \nabla U(\phi^n)\Big) +\lambda^{n+1} \hat\tau^{n+1},
\end{align}
where $K_n$ is defined in \eqref{eq:EnKn}. We remark that
in this semi-discrete scheme, the smoothness of $\phi^n$ does not guarantee
the smoothness of  $\phi^{n+1}$. However, this does not introduce any difficulty in applying the
full discretization scheme of \eqref{eq:NGsemischeme} in numerical computations.
\end{remark}

\section{Numerical examples}\label{sec:Results}

We apply the numerical method to three examples: an example in 2D,
the re-arrangement of seven-atom cluster, and the Maier-Stein model.

In these examples, we use
\[\lambda(\alpha) = \frac{\sqrt{2E-2U(\phi(\alpha))}}{\abs{\phi'}}\]
as the indicator of $\lambda$-critical states. Here, $\phi(\alpha)$ is the path we get after iteration, $E=\max_{\alpha} U(\phi)$. The local minima of $\lambda(\alpha)$ are $\lambda$-critical states. When $\epsilon=0$, metastable states and transition states are all $\lambda$-critical states($\lambda(\alpha)=0$ at these states). As proved in proposition \ref{prop:deltaT}, the transition spent longer time at these states than others.

\subsection{A 2D example}
We consider the potential $V$ in two-dimensional space:
\[
  \begin{aligned}
    V(&x,y) = 4(x^2+y^2-1)^2y^2-\exp(-4((x-1)^2+y^2)) -  \exp(-4((x+1)^2+y^2))\\
           &+ \exp(8(x-1.5)) +\exp(-8(x-1.5)) + \exp(\gamma(y+0.25)) + 0.2\exp(-8x^2).
  \end{aligned}
\]
The potential $V(x, y)$ has two minima at $a$ and $b$, and two saddle points at
$s_1$ and $s_2$, as shown in \cref{fig:3}.
For a properly chosen $\gamma$($\gamma\approx 12.16$), the two saddle points have the same potential energy.
The FW functional (i.e. when $\epsilon=0$)
has two minima for the transition between
$a$ and $b$, one being a straight line through $s_1$ and the other a circular
path through $s_2$. The two paths have the same action, and they are also the
minimum energy paths for the potential $V$.

\begin{figure}
  \centering
  \scalebox{0.4}{\includegraphics{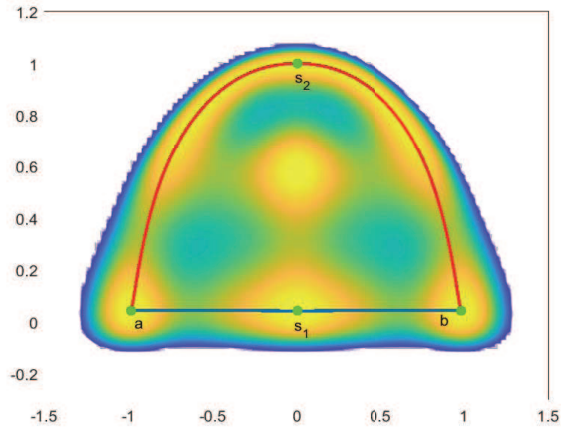}}
  \caption{(Color Online). The heat plot of $V(x,y)$ and two transition paths when $\epsilon=0$ (the FW theory). The blue and red lines correspond to the direct and circular transition paths, respectively. The green dots correspond to the minima ($a$ and $b$) and saddle points ($s_1$ and $s_2$) of $V(x,y)$. The two saddle points $s_1$ and $s_2$ have the same energy when $\gamma\approx 12.16$. The background color indicates the value of $V$, where the yellow color corresponds to small values of $V$, and blue color corresponds to large values of $V$.}\label{fig:3}
\end{figure}

Notice that when $\epsilon=0$, all critical points of $V$, i.e. the
states $\vec{x}$ with $\nabla V(\vec{x})=0$,
are critical points (maxima) of $U$. When $\epsilon>0$ but small,
the local/global maxima of $U$ are still in the neighborhood of the critical points
of $V$. For example, as shown in  \cref{fig:4},
$U$ attains the maxima near the
minima, the saddle points and the maxima of $V$ when $\epsilon=0.05$.
Among these maxima of $U$, the one near $s_2$ is the global maximum, thus the
critical point of $U$.

For $\epsilon=0.05$, we solved \eqref{eq:relax}
for the transition path between $a$ and $b$ using  EGMA.
Using different initial path in the iteration, we obtained
two paths connecting $a$ and $b$: one passing through the three local maxima
of $U$ near $a$, $s_1$ and $b$, respectively, and the other
passing through the global maxima of $U$ near $s_2$. Both satisfies the Euler-Lagrange equation \eqref{eq:EL0}, thus are extremals of $\hat{S}_E[\phi]$. However, the two paths have different energy:
$E_c\approx 1.0849$ for the first one and $E_c\approx 1.5516$ for the second one.
 These energy values are also the maximum of $U$ along the corresponding
path. From the analysis in previous sections, the path passing through the
global maximum of $U$
is the graph limit of the minimizers of $S_T^\OM[\psi]$ as $T\rightarrow\infty$.

\begin{figure}
  \centering
  \scalebox{0.5}{\includegraphics[angle=90]{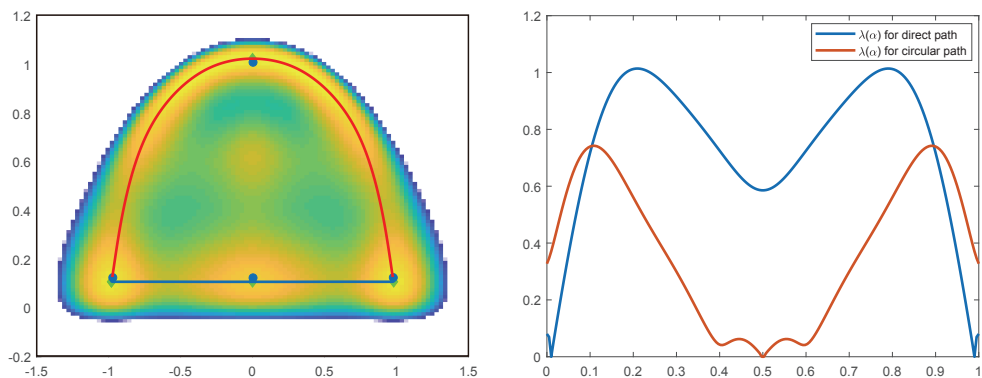}}
  \caption{(Color Online). Left panel: Two solutions of the Euler-Lagrange
equation \eqref{eq:EL0} with $\epsilon=0.05$. The blue straight line passes through local maxima of $U(x,y)$ close to the minima of $V$. The red circular line passes through the global maxima of $U$ that are close to the upper saddle of $V$. The background is the heat plot of $U$, in which the yellow color corresponds to
large values of $U$, and blue color corresponds to small values of $U$, which is opposite to \cref{fig:3}. The blue circles indicate the minima and saddle points of $V$, and the green diamonds are local maxima of $U$. Right panel: The indicator $\lambda(\alpha)$ of two paths. The blue and red curves correspond to the straight and circular paths, respectively. The indicator $\lambda=0$ corresponds to the states with $U(\phi)=E_c$.}\label{fig:4}
\end{figure}

When $\epsilon$ is relatively large, on the other hand,
$U$ may have a different landscape.
For example, when $\epsilon=0.5$, the critical point of $U$ near $s_2$ disappears and
two new critical points appear slightly far away from this saddle point.
\cref{fig:5} shows that the converged path passes through
these two critical points with energy $E_c\approx 23.9939$.
This result shows that with a finite $\epsilon$, the OM functional may
give quite different $\lambda$-critical states compared with the transition states for the zero noise limit.

\begin{figure}
\centering
\scalebox{0.48}{\includegraphics[angle=90]{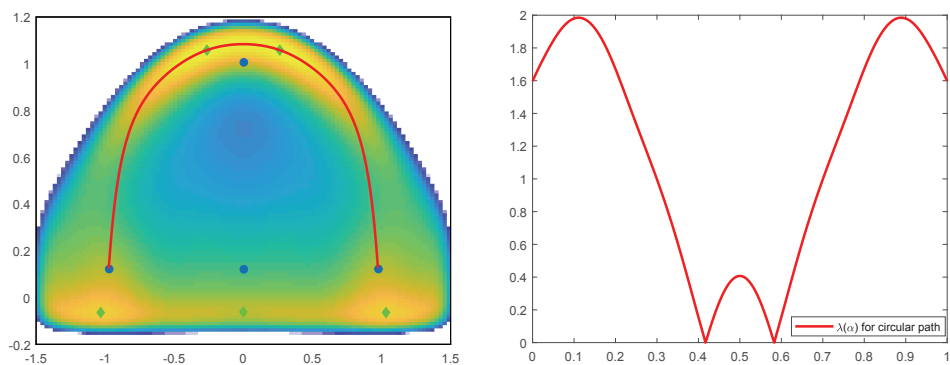}}
\caption{(Color Online). Left panel: The circular transition path
when $\epsilon=0.5$. The path passes through two critical points
that are away from the upper saddle point of $V$. The background is the heat plot of $U$.
Right panel: The indicator $\lambda(\alpha)$. The zeros of $\lambda$ corresponds to two critical points. }\label{fig:5}
\end{figure}

\subsection{Rearrangement of a seven-atom cluster}

In this example, we consider the rearrangement of a cluster consisting of
seven atoms. This problem was studied
in \cite{dellago1998efficient, weinan2002string, pinski2010transition}
using different approaches. The positions of the atoms are denoted
by $\vec{x}^{(i)}\in \mathbb{R}^2, i=1,2,\dots 7$. The state of the system
in the configuration space is represented by a
vector $\vec{x}=(\vec{x}^{(1)},\dots\vec{x}^{(7)})\in\mathbb{R}^{14}$.
The interactions between the atoms are modeled by the Lennard-Jones potential:
\begin{equation}\label{eq:LJP}
  V(\vec{x}) = 2\delta\sum_{i\neq j}\left(\frac{\sigma}{r_{ij}}\right)^{12} - \left(\frac{\sigma}{r_{ij}}\right)^{6},
\end{equation}
where $\delta,\sigma>0$, $r_{ij} = |\vec{x}^{(i)}-\vec{x}^{(j)}|$
is the distance between $\vec{x}^{(i)}$ and $\vec{x}^{(j)}$.
The global minimum of this potential corresponds to the configuration in which
an atom is surrounded by the other six in a hexagonal shape.
We compute the pathway along which
the central atom (colored in white in Figs \cref{fig:5.5}-\cref{fig:8})
migrates to the surface using  EGMA  with different values
of $\epsilon$. We use $N=192$ points to discretize the path, and
take $\delta=\sigma=1$ in the potential.

We first show the transition path and transition states inferred by
the FW theory, i.e. the case when $\epsilon=0$, in  \cref{fig:5.5}.
The curve in  \cref{fig:5.5} shows the indicator $\lambda$.
As shown in the FW theory \cite{heymann2008geometric},
the metastable states and transition states along the transition path are given by the points with $\lambda(\alpha)=0$. They are all $\lambda$-critical states.
We plot the configuration of each $\lambda$-critical state in  \cref{fig:5.5}.
These $\lambda$-critical states are also the minima or saddle points of $V$ and they are
the same as those obtained
in the earlier work \cite{dellago1998efficient,weinan2002string}.

\begin{figure}
  \centering
  \scalebox{0.4}{\includegraphics[angle=90]{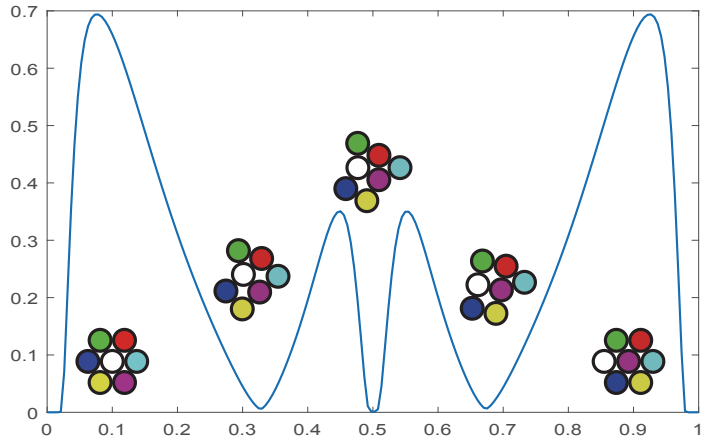}}
  \caption{(Color Online). The indicator $\lambda(\alpha)$
for the MPP and the corresponding $\lambda$-critical states when $\epsilon=0$
(i.e. the FW functional).
The $\lambda$-critical states are the states where $\lambda$ attains local minima.}
\label{fig:5.5}
\end{figure}

\cref{fig:6} shows the indicator $\lambda(\alpha)$ along the path obtained
with $\epsilon=0.01$.
As we have mentioned, each local minimum of $\lambda$ corresponds to a $\lambda$-critical state. These states have different interpretations. The initial and final states, which correspond to the case where one atom is surrounded by  the other six,  are the minimum of $\lambda$. This means that they are more stable than the other three $\lambda$-critical states. Similarly, the symmetric state (observe with a small tilt angle) in the middle of the path  is more stable than another two asymmetric ones aside. Because $\epsilon$ is quite small here, the $\lambda$-critical state we get are  similar as those obtained by FW theory.

\begin{figure}
  \centering
  \scalebox{0.4}{\includegraphics[angle=90]{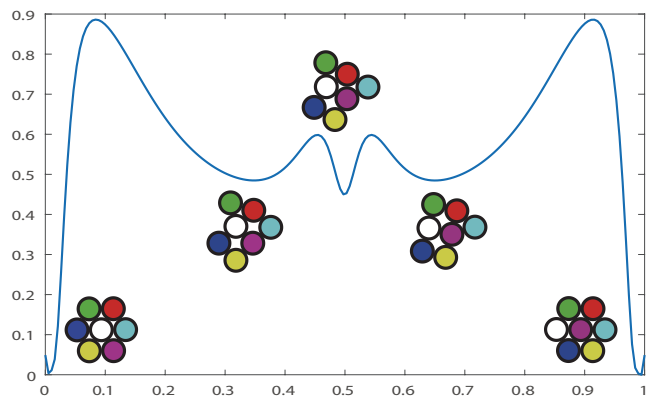}}
  \caption{(Color Online). The indicator $\lambda(\alpha)$ for the  MPP and
the corresponding $\lambda$-critical states when $\epsilon=0.01$.
The transition behavior resembles the case when $\epsilon=0$.
}\label{fig:6}
\end{figure}

We also applied  EGMA  to the cases $\epsilon=0.1$ and $\epsilon=1$, and
the numerical results are shown in Figs. \cref{fig:7} and \cref{fig:8}, respectively.
In  \cref{fig:7}, we observe similar transition pattern as
in the case $\epsilon=0.01$. Although the critical points of $U$ that
the path goes through are slightly perturbed from the initial and final states,
their configurations are qualitatively indistinguishable from those states.
The symmetric state in the middle remains to be a $\lambda$-critical state
as a local minimum of the indicator $\lambda$.
In the case $\epsilon=1$ (shown in Fig, \cref{fig:8})), the two asymmetric states inside replace the initial and final states as the critical states. Furthermore, the middle symmetric state is no longer a $\lambda$-critical state. During the transition, the atoms have slight overlapping due to strong noise perturbations. Similar phenomena was also observed in \cite{pinski2010transition}.

\begin{figure}
  \centering
  \scalebox{0.4}{\includegraphics[angle=90]{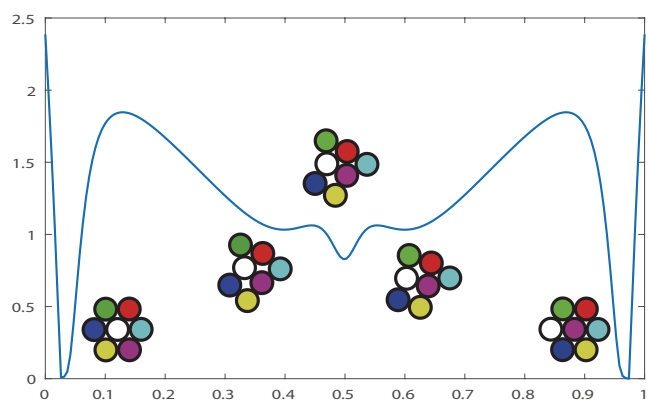}}
  \caption{(Color Online). The indicator $\lambda(\alpha)$ for the  MPP and
the corresponding $\lambda$-critical states when $\epsilon=0.1$.
The qualitative behavior remains unchanged compared with the case $\epsilon=0.01$.}\label{fig:7}
\end{figure}

\begin{figure}
  \centering
  \scalebox{0.4}{\includegraphics[angle=90]{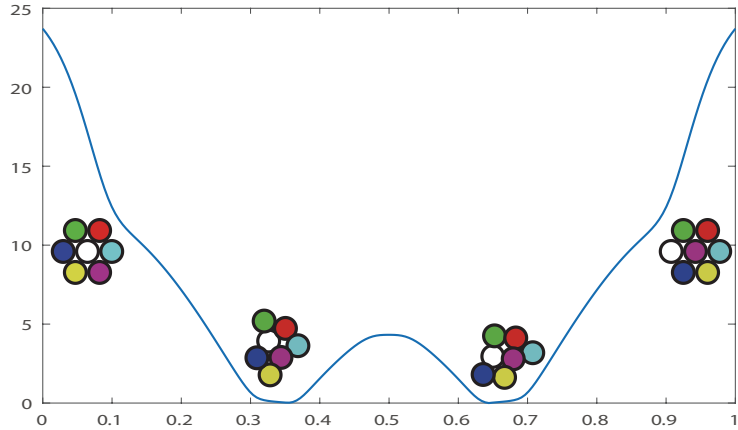}}
  \caption{(Color Online). The indicator $\lambda(\alpha)$ for the  MPP and
the corresponding $\lambda$-critical states when $\epsilon=1$.
The two asymmetric states shown in the figure
replace the initial and final states
as the critical states. The atoms slightly overlap due to the strong noise.}\label{fig:8}
\end{figure}

\subsection{The Maier-Stein model}

In this example, we apply the numerical method to a non-gradient system.
The Maier-Stein model is a standard non-gradient diffusion process
that has been carefully studied \cite{maier1996scaling}.
The drift term $\vec{b}$ in Eq. \eqref{eq:diff} is
\begin{equation}\label{eq:MS}
  \vec{b}(x,y) = \begin{bmatrix}
                   x - x^3 -\beta xy^2 \\
                   -(1+x^2)y
                 \end{bmatrix},
\end{equation}
where $\beta>0$ is a parameter. The system is of the gradient type only when
$\beta=1$. For any $\beta>0$, the deterministic dynamics
$\dot{\vec{x}}=\vec{b}(\vec{x})$ has two stable fixed points at $(\pm 1,0)$
 and one saddle point at $(0,0)$.
The path potential $U(x,y)$ is given by
\begin{equation}\label{eq:MSU}
  U(x,y) = 4\epsilon x^2 + \epsilon\beta y^2 - \frac{1}{2}\big[(x-x^3-\beta xy^2)^2+(1+x^2)^2y^2\big].
\end{equation}

In the limit $\epsilon\to 0$, Maier and Stein studied the transition path from $(-1,0)$ to $(1,0)$
for various $\beta$, and found two transition patterns \cite{maier1996scaling}.
When $\beta<4$, the path is the line segment connecting $(-1,0)$
and $(1,0)$, while when $\beta>4$, the path is composed of two parts:
one from $(-1,0)$ to $(0,0)$ following the curved heteroclinic orbit and the other
from $(0,0)$ to $(0,1)$ following the line segment.

Using  EGMA, we can study the same transition
for $\epsilon>0$.  \cref{fig:1} shows the numerical results for
$\epsilon=0.1$. We also obtain two types of transition paths. However,
the critical value of $\beta$ that separates the two pattens is lowered to
a value between 3.4 and 3.5. More interestingly, as predicted by the
theoretical result, the paths now pass through a critical point $(x_c,0)$ which
is located to the left of $(-1, 0)$.

\begin{figure}
  \centering
  \scalebox{0.45}{\includegraphics{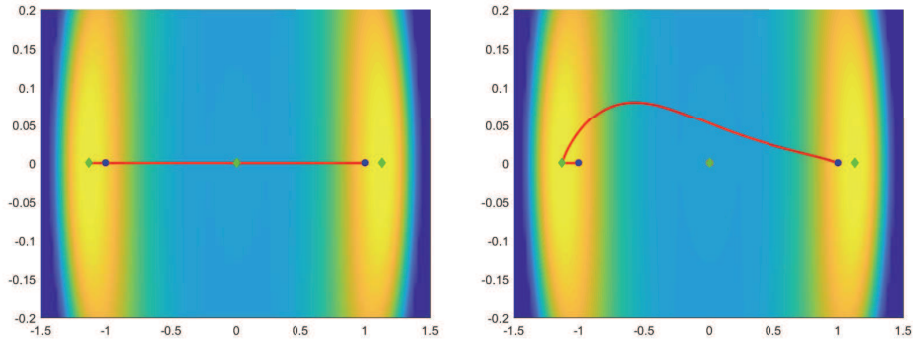}}
  \caption{(Color Online). The MPP of the OM functional for the Maier-Stein model, where $x_s = (-1,0)$, $x_f=(1,0)$ and $\epsilon=0.1$. Left panel: $\beta=3.4$. Right panel: $\beta=3.5$.  Blue circles are fixed points of $\dot{\vec{x}}=\vec{b}(\vec{x})$. Green diamonds are local maximum points of $U(x)$. The background is the heat map of $U(\vec{x})$.
}\label{fig:1}
\end{figure}

The critical point $(x_c,0)$ can be calculated explicitly as $x_c^2=(2+\sqrt{1+24\epsilon})/3$. The corresponding energy is given by $E_c=U(x_c,0)$.
This helps us study the convergence property of our   EGMA. We set $\epsilon=0.1$, $\delta s=0.01$, $\beta=10$ and run 500 iterations
with different spatial resolution $N$.
We compute the difference between $E_c$ and $E_n$ in each iteration step.
The convergence history of the energy in the first 100 iterations
is shown in the upper part of Table \ref{tab:itererr} when $N=1000$.
We observe that $E_n$ increases monotonically towards $E_c$.
A quantitative fitting shows that
$|E_{n}-E_c|\approx O(n^{-2}) + e_N$, where
$e_N$ is the difference of the limit of $E_n$, denoted by $E^\star_N$, and
 $E_c$. The error $e_N$ is mainly determined by the spatial resolution $N$,
which is also shown in Table \ref{tab:itererr}. A fitting
shows that $|E^\star_N-E_c|\approx O(h^{2.314})$. This suggests
approximately second order convergence of the energy parameter
with second order spatial discretization.
We will leave the rigorous convergence analysis to the future study.

\begin{table}[h]
\renewcommand{\arraystretch}{1.2}
\caption{Upper part: Convergence of energy parameter with respect to the iteration number $n$. Lower part: Convergence of the limit energy with respect to the spatial resolution $N$.}\label{tab:itererr}
\begin{center}
\begin{tabular}{c|ccccc}
  \hline\hline
  IterNum $n$ & 10 & 20 & 30 & 50 & 100\\
  \hline
  $E_c-E_n$ &   $2.1\times 10^{-3}$ & $3.9\times 10^{-4}$ & $1.5\times 10^{-4}$ & $8.1\times 10^{-5}$ & $6.2\times 10^{-5}$ \\
    \hline\hline
  SpatialRes $N$ & 100  & 1000 & 2000 & 4000 & 5000\\
  \hline
  $E_c-E^\star_N$ & $6.0\times 10^{-5}$ & $2.0\times 10^{-5}$ & $3.0\times 10^{-6}$  & $1.2\times 10^{-7}$ & $5.3\times 10^{-9}$\\
  \hline
\end{tabular}
\end{center}
\end{table}

From Eq.~\eqref{eq:MSU} we have
\[U(0,y)=(\epsilon\beta-\frac{1}{2})y^2,\]
which shows that $U(x,y)$ is unbound along the $y$-axis
when $\epsilon\beta>\frac{1}{2}$. This violates the Assumption \ref{asm:decay}.
In principle this is beyond our theoretical framework. But it is still interesting to apply  EGMA  in this case. Based on the result
in Proposition \ref{lem:liminfE}, the energy parameter $E$
will tend to infinity while the path may diverge along $y$-axis.
The numerical results confirm this conjecture,
although the current theory does not cover this case.
The divergent curves are shown in \cref{fig:diverge}.

\begin{figure}
  \centering
  \scalebox{0.5}{\includegraphics{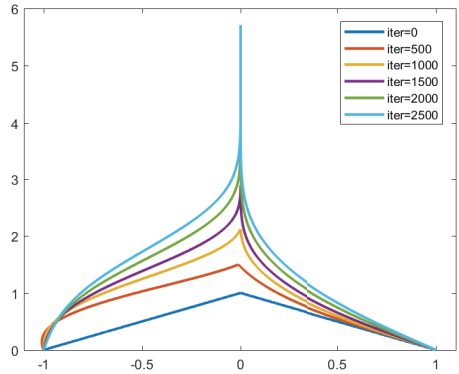}}
  \caption{(Color Online). Divergence of the $y$-component of
$\phi^n(\alpha)$ in the iterations when $\epsilon\beta>\frac{1}{2}$. Here we choose $\epsilon=1$, $\beta=10$.}\label{fig:diverge}
\end{figure}

\section{Conclusion and discussion} \label{sec:con}

In this paper, we studied the minimization problem of the OM
functional $S^{\OM}_T[\psi]$ when $T$ tends to infinity. Under mild conditions, we rigorously showed that
the infimum of $S^{\OM}_T[\psi]$ over $\psi$ and $T$ is always $-\infty$
and it occurs only when $T\to+\infty$. Moreover, we proved that
when $T\to +\infty$, the minimizer of $S^{\OM}_T[\psi]$ has convergent
subsequence in the configuration space. With the help of Maupertuis principle,
the problem of finding the graph limit of the minimizers of the OM functional
is translated into that of finding an extremal of $\hS{E}[\phi]$
with the energy $E^\star = \max_{\vec{x}} U(\vec{x})$,
where $U$ is the path potential.

Based on these theoretical results, we designed an efficient
algorithm (EGMA) to find the energy $E^\star$ and solve the Euler-Lagrange
equation simultaneously.
This algorithm can be viewed as a nontrivial extension of
the geometric minimum action method (gMAM), which was proposed
 for the double minimization problem of  the FW functional at zero temperature
\cite{heymann2008geometric}.
In gMAM, the energy $E$ is fixed at zero,
and the corresponding optimal transition
time $T=\int_{0}^{1}(2E-2U(\phi))^{-\frac{1}{2}}\abs{\phi'}\d\alpha$
can be either finite or infinite.
We note that the method can be extended to the case of non-zero $E$,
which corresponds to a prescribed value for the transition time.
 In EGMA, the energy changes step by step to
ensure that the time $T$ always goes to infinity.
This algorithm was successfully applied to several typical rare event examples.
Although the rigorous proof of the convergence of the algorithm is still absent,
our numerical examples in Sections 6.1 and 6.2 demonstrated that it
converges to an extremal of $\hS{E^\star}[\phi]$ in practice.

Some possible extensions and unsolved problems arise naturally based
on the theoretical analysis of the current paper. Below we list some of them.

\begin{enumerate}
\item[(i)] As mentioned in the remarks, the conclusions that
$\inf_{T}\inf_{\psi}S^{\OM}_T[\psi]=-\infty$ and the minimizer of $S^{\OM}_T[\psi]$
has convergent subsequence in the configuration space can be generalized
to non-gradient systems with drift $\vec{b}(\vec{x})$ under suitable assumptions.
However, the proof of the key result that $\phi^\star$ is an extremal
of $\hS{E^\star}[\phi]$ relies on the gradient form of $\vec{b}(\vec{x})$.
Indeed, the proof of Lemma \ref{lem:finitjump} holds because
the value of $\hS{E}[\phi]$ is non-negative. Besides, we used an
analog of Hartman-Grobman theorem to transfer the uniform BV property of $\phi_k'$ for linearized problem to  the non-linear case in the neighborhood of the  critical point $\vec{x}_c$. This approach requires that $(\vec{x}_c,0)$ is a hyperbolic fixed point of the first order system \eqref{eq:nlfoel}. For non-gradient case, the first order system \eqref{eq:nlfoel} becomes
\begin{equation}\label{eq:NGnlfoel}
  \begin{aligned}
  \dot{\Psi}_1 &= \Psi_2,\\
  \dot{\Psi}_2 &= -\nabla^2U(\Psi_1) - (\nabla\vec{b}^\T(\Psi_1)-\nabla\vec{b}(\Psi_1))\Psi_2.
  \end{aligned}
\end{equation}
 The point $(\vec{x}_c,0)$ is still a fixed point. However, the eigenvalues of the Jacobian of \eqref{eq:NGnlfoel} at $(\vec{x}_c,0)$ have non-zero imaginary past. To study the behavior of $\phi_k'$ near the critical point, we might need more delicate  study of the dynamics \eqref{eq:NGnlfoel} on center manifold and more advanced result on linearization. So how to extend Theorem \ref{thm:extremal} to
 the non-gradient case remains an interesting issue.

\item[(ii)] Presumably, the Freidlin-Wentzell functional can be viewed as a limit of Onsager-Machlup functional when $\epsilon\to 0$. However, we have $\lim_{\epsilon\to 0}\inf_{T}$ $\inf_{\psi}S^{\OM}_T[\psi]=-\infty$ by the Proposition \ref{prop:D}. This suggests that a naive process is not valid to establish such connection. A possible alternative to investigate the limit of OM functional is to relate $T$ and $\epsilon$ by a function $T=T(\epsilon)$ and study the convergence of $\inf_{\psi}S^{\OM}_{T(\epsilon)}[\psi]$ as $\epsilon\to 0$. This idea has been partially studied in \cite{pinski2012gamma}. They showed that for the scaling $T=\epsilon^{-\alpha}$, $0<\alpha\leqslant 1$, the OM functional $\Gamma$-converges to a functional completely characterized by the FW theory. However, for a more general and physical scaling $T=T(\epsilon)$, the convergence of OM functional and its minimizer is not clear. The renormalized OM functional \eqref{eq:ROM} might be a good candidate to perform such analysis.

\item[(iii)] The discretization scheme described in  EGMA
can be further improved. For example, one may discretize $\hS{E}[\phi]$ first then
use some optimization methods like quasi-Newton or conjugate gradient type methods
to search for the extremal. However, the algorithm we proposed here
combines the iteration for solving Euler-Lagrange equation and
finding the maximum of $U(\vec{x})$ simultaneously. It has the advantage
that we can compute the optimal energy parameter and the MPP at the same time.
 This strategy may not apply  for the optimization methods.
Designing more efficient numerical methods to perform these
two tasks together is an issue of practical interests.
\end{enumerate}

In summary, the current work provides new ideas on the FW-OM dilemma. It will be instructive to further  study the FW-OM connections with this new perspective.


\Acknowledgements{The authors are grateful to Profs.  Shaobo Gan, Eric Vanden-Eijnden, Jiazhong Yang and Shulin Zhou for stimulating discussions. Special thanks are due to Prof. Wenmeng Zhang for his patient explanation about their recent progress on $C^{1,\beta}$-linearization problem. The work of T. Li is supported by the NSFC under grants No. 11421101, 91530322 and 11825102. The work of W. Ren is partially supported by Singapore MOE ACRF grants R-146-000-267-114 (Tier-1) and R-146-000-232-112 (Tier-2), and NSFC grant No. 11871365. The work of X. Li is supported by  the Construct Program of the Key Discipline in Hunan Province.}



\bibliographystyle{plain}
\bibliography{OM}

\begin{appendix}

\section{Proof of Proposition \ref{prop:condev}}
We now prove Proposition \ref{prop:condev} in this Appendix. This relies on a series of lemmas. Some of them are quite technical. Recall that  $\psi_k$ is the minimizer of $S^{\OM}_{T_k}[\psi]$, where $T_k\to +\infty$ as $k\to \infty$. The corresponding energy $E_k=E(T_k)$ and the extremal of $\hS{E_k}[\phi]$ is $\phi_k$ for $k=1,2,\dots$. By theorem \ref{prop:conv}, we may assume $E_k\to E^\star$ and $\phi_k\to \phi^\star$ without loss of generality.

Our main idea is to show that each component of $\{\phi_k'\}$ has uniformly bounded variations. Denote $\phi_k'^{(i)}$ the $i$th component of $\phi'_k$. We have the total variation
  \begin{equation}\label{eq:componentTV}
  \begin{aligned}
    \bigvee\limits_0^1\phi_k'^{(i)}(\alpha) &= \sup\limits_{\Delta}\sum_{j=0}^{n} \big|\phi_k'^{(i)}(\alpha_{j+1})-\phi_k'^{(i)}(\alpha_j)\big| \\
    &\leqslant\sup\limits_{\Delta}\sum_{i=1}^{d}\sum_{j=0}^{n} \big|\phi_k'^{(i)}(\alpha_{j+1})-\phi_k'^{(i)}(\alpha_j)\big|\\
    &\leqslant\sup\limits_{\Delta}\sum_{j=0}^{n}\sqrt{d}\left(\sum_{i=1}^{d}\big|\phi_k'^{(i)}(\alpha_{j+1})-\phi_k'^{(i)}(\alpha_j)\big|^2\right)^{\frac{1}{2}}\\
    &=\sup\limits_{\Delta}\sum_{j=0}^{n}\sqrt{d}\abs{\phi'_k(\alpha_{j+1})-\phi'_k(\alpha_j)},
  \end{aligned}
  \end{equation}
  where $\Delta$ is a partition of $[0,1]$ with subdivision points $0=\alpha_0<\alpha_1<\alpha_2\dots<\alpha_{n+1}=1$. We only need to show that
  \begin{equation}\label{eq:summation}
    \sup\limits_{\Delta}\sum_{j=0}^{n}\abs{\phi'_k(\alpha_{j+1})-\phi'_k(\alpha_j)}
  \end{equation}
  is uniformly bounded.

Next, we divide $[0,1]$ into intervals in which $\phi_k\in C^2$. We need the following lemma.

\begin{lemma}\label{lem:finitjump}
  There is a constant $K>0$ which is independent of $T$, such that for any $T>0$, $\#\{t\in[0,T]|U(\psi_T(t))=E(T)\}\leqslant K$.
\end{lemma}
\begin{proof}
  By Assumption \ref{asm:levelset}, we have the decomposition
  \[\Big\{\vec{x}|U(\vec{x})=E(T)\Big\} = \bigcup\limits_{i=1}^N B_i,\]
  where $B_i\cap B_j=\emptyset$ for $i\neq j$, $B_i$ is closed and connected. We are going to show that $\psi_T$ passes through each $B_i$ at most once.

 Assume that there exist $0<t_1<t_2<T$ such that $\vec{x}_1=\psi_T(t_1)\in B_i$, $\vec{x}_2=\psi_T(t_2)\in B_i$. Define $\tilde{\psi}_{\tilde T}(t)=\psi_T(t+t_1)$ for $t\in[0,t_2-t_1]$, where $\tilde T=t_2-t_1$. $\tilde{\psi}_{\tilde T}$ must be the minimizer of $S_{\tilde T}[\psi]$ with boundary conditions $\psi(0)=\vec{x}_1$, $\psi(\tilde T)=\vec{x}_2$.
Since $\tilde{\psi}_{\tilde{T}}$ is the minimizer of $S_{\tilde{T}}[\psi]$,  by Maupertuis principle we know that the minimizer $\tilde{\psi}_{\tilde{T}}$ induces a geodesic on $B_i$ from $\vec{x}_1$ to $\vec{x}_2$ with Riemannian metric \cite{Rota1990Mathematical}(Theorem in page 246)
  \[\d\rho_E = \sqrt{2E(T)-U(\phi)}\abs{d\phi}.\]
$\tilde{\psi}_{\tilde{T}}$ must lie on $B_i$ since it minimizes the distance induced by $\rho_E$. However, we have
  \[T\geqslant\int_{\vec{x}_1}^{\vec{x}_2}\frac{\abs{\d \tilde{\psi}_{\tilde{T}}}}{\sqrt{2E(T)-2U(\tilde{\psi}_{\tilde{T}})}}=+\infty.\]
  This contradiction implies that $\psi_T$ passes through $B_i$ at most once.

 Since there are only finite components $B_i$ of level set $\{\vec{x}|U(\vec{x})=E(T)\}$, there is a positive lower bound of the distance between each two components
  \[\min\limits_{i\neq j}\mbox{dist}(B_i,B_j)=l>0.\]
By Assumption \ref{asm:uniformlybounded}, the length of $\psi_T$ is uniformly bounded. So we have for any $T>0$,
  \[M\geqslant \int_{0}^{T}|\dot{\psi}_T|\d t\geqslant (K-1) l.\]
  This leads to the conclusion that $K$ is finite and independent of $T$.
\end{proof}

Lemma \ref{lem:finitjump} shows that for any $k$, the points $\alpha$ satisfying $U(\phi_k(\alpha))=E_k$ are finite. Note that when $U(\phi_k(\alpha))<E_k$, $\phi_k\in C^2$. By Euler-Lagrange equation \eqref{eq:EL0}, in an interval $[\beta_1,\beta_2]$ such that $\phi_k\in C^2$, the total variation of $\phi_k'$ is
\begin{equation}\label{eq:ddphi}
  \bigvee_{\beta_1}^{\beta_2}\phi_k^{(i)} \leqslant \int_{\beta_1}^{\beta_2}\abs{\phi_k''}\d\alpha=\int_{\beta_1}^{\beta_2}\frac{\abs{\phi_k'}\sqrt{\abs{\nabla U}^2\abs{\phi_k'}^2-\inp{\nabla U}{\phi_k'}^2}}{2E_k-2U(\phi_k)}\d\alpha.
\end{equation}
We can rewrite Eq. \eqref{eq:ddphi} with time parameterization as
\begin{equation}\label{eq:ddphipsi}
  \begin{aligned}
  \int_{\beta_1}^{\beta_2}\abs{\phi_k''}\d\alpha =& \int_{\beta_1}^{\beta_2}\frac{\abs{\phi_k'}\sqrt{\abs{\nabla U}^2\abs{\phi_k'}^2-\inp{\nabla U}{\phi_k'}^2}}{2E_k-2U(\phi_k)}\d\alpha\\
  & = \int_{\beta_1}^{\beta_2}\frac{\abs{\phi_k'}^2\sqrt{\abs{\nabla U}^2\abs{\hat{\tau}_k}^2-\inp{\nabla U}{\hat{\tau}_k}^2}}{2E_k-2U(\phi_k)}\d\alpha\\
  &=\int_{t_1}^{t_2}\frac{\abs{\phi_k'}^2\sqrt{\abs{\nabla U}^2|\dot{\psi}_k|^2-\langle\nabla U, \dot{\psi}_k\rangle^2}}{|\dot{\psi}_k|^3}\dd{\alpha}{t}\d t\\
  & = \abs{\phi_k'}\int_{t_1}^{t_2}\sqrt{\frac{\abs{\nabla U}^2|\dot{\psi}_k|^2-\langle \nabla U, \dot{\psi}_k\rangle ^2}{|\dot{\psi}_k|^4}}\d t,
  \end{aligned}
\end{equation}
where $\hat{\tau}_k=\phi_k'/|\phi_k'|$, $t_1 = \ell^{-1}(\beta_1)$ and $t_2 = \ell^{-1}(\beta_2)$.

For any $T>0$, we define
\begin{equation}\label{eq:Theta}
  \Theta_T(t) = \frac{\abs{\nabla U(\psi_T)}^2|\dot{\psi}_T|^2-\langle\nabla U(\psi_T), \dot{\psi}_T\rangle^2}{|\dot{\psi}_T|^4}
\end{equation}
 and denote $\Theta_k(t)=\Theta_{T_k}(t)$. It is obvious that $\Theta_T(t)$ is well-defined for $U(\psi_T(t))<E(T)$. Recall that the minimizer $\psi_T$ satisfies the Euler-Lagrange equation
\[\ddot{\psi} + \nabla U(\psi)=0,\]
or equivalently the first order system
\begin{equation}\label{eq:nlfoel}
  \begin{aligned}
  \dot{\Psi}_1 &= \Psi_2, \\
  \dot{\Psi}_2 &= -\nabla U(\Psi_1).
  \end{aligned}
\end{equation}
We have $\psi\equiv \Psi_1$. We will use $\tilde{\Psi}$ to denote the extended variable $(\Psi_1,\Psi_2)\in\mathbb{R}^{2d}$. The state $\tilde{\Psi}_c=(\vec{x}_c, 0)$ is a fixed point of the system, where $\vec{x}_c$ is a critical point. The following lemma shows that the function $\Theta_T(t)$ can be continuously extended to $t\in [0,T]$.

\begin{lemma}[Preliminary properties of $\Theta_T$ and $\Psi$]\label{lem:bounded}
  \begin{enumerate}
    \item Given $\vec{x}_s\neq \vec{x}_c$, $\vec{x}_f\neq \vec{x}_c$ and $T>0$,  $\lim_{s\to t}\Theta_T(s)$ exists for any $t\in [0,T]$.
    \item There exists a neighborhood $\mathcal{U}$ of critical point $\vec{x}_c$ and constant $m>0$, such that for any given $t_0>0$, $\vec{x}_s,\vec{x}_f\in\partial\mathcal{U}$, we have $\abs{\Psi_2}\geqslant m>0$ for $t\in [0,t_0]$  when $T\to +\infty$.
  \end{enumerate}
\end{lemma}
\begin{proof}
  (1) Since $\Psi_1,\Psi_2\in C^1[0,T]$, $\Theta_T$ is continuous except the points with $\Psi_2(t_c)=0$. Below we will show these points are removable singularities.

Suppose $\Psi_2(t_c)=0$.  We have the following Taylor expansion near $t_c$
\begin{equation}\label{eq:taylorpsi}
\begin{aligned}
  -\nabla U(\Psi_1) &= \Psi^{(2)} + \frac{1}{2}\Psi^{(4)}(t-t_c)^2 + \frac{1}{24}\Psi^{(6)}(t-t_c)^4 + o(t-t_c)^5,\\
  \Psi_2 & = \Psi^{(2)}(t-t_c) + \frac{1}{6}\Psi^{(4)}(t-t_c)^3 + \frac{1}{120}\Psi^{(6)}(t-t_c)^5 + o(t-t_c)^6,
\end{aligned}
\end{equation}
where $\Psi^{(2)}$, $\Psi^{(4)}$, $\Psi^{(6)}$ are corresponding higher derivatives of $\Psi_1(t)$ assuming value at $t_c$. The odd order derivatives of $\Psi_1$ disappear because $\Psi^{(1)}=0$,
\[\Psi^{(3)} = \dd{\,}{t}\Psi_1^{(2)}=-\dd{\,}{t}\nabla U(\Psi_1)=-\nabla^2 U\cdot\dot{\Psi}_1=0\]
and
\[\Psi^{(5)} = \dd{\,}{t}\Psi_1^{(4)}=\dd{\,}{t}\nabla^2 U\nabla U=\nabla(\nabla^2 U\nabla U)\cdot\dot{\Psi}_1=0.\]
 Substituting \eqref{eq:taylorpsi} into $\Theta_T(t)$, we obtain that in a neighborhood of $t_c$,
\begin{equation}\label{eq:Thetat}
  \Theta_T(t) = \frac{\big(\abs{\Psi^{(2)}}^2\abs{\Psi^{(4)}}^2 - \inp{\Psi^{(2)}}{\Psi^{(4)}}^2\big)(t-t_c)^2 + o(t-t_c)^2}{9\abs{\Psi^{(4)}}^2+o(1)}.
\end{equation}

With the fact that $\Psi^{(2)} = -\nabla U$, $\Psi^{(4)}=\nabla^2 U\nabla U$ taking value at $\Psi_1(t_c)$, we have
\[\Theta_T(t) = \frac{[\widehat{\nabla U}(\nabla^2 U)^2\widehat{\nabla U} - (\widehat{\nabla U}\nabla^2 U\widehat{\nabla U})^2](t-t_c)^2 + o(t-t_c)^2}{9+o(1)},\]
where $\widehat{\nabla U}=\nabla U/\abs{\nabla U}$. Since $\Psi_1$ is uniformly bounded, $\lim_{t\to t_c}\Theta_T(t)=0$. This ends the proof of the first statement.

(2) By assumption \ref{asm:nondegenerate}, all eigenvalues of $\nabla^2U(\vec{x}_c)$ are negative, so there is a neighborhood $\mathcal{U}$ of $\vec{x}_c$, for any $\vec{x}\in\partial\mathcal{U}$, $\vec{x}\neq\vec{x}_c$, $U(\vec{x})<U(\vec{x}_c)$. For $\vec{x}, \vec{y}\in\mathbb{R}^d$, let $(\Phi_1(t,\vec{x},\vec{y}),\Phi_2(t,\vec{x},\vec{y}))$ be the solution of \eqref{eq:nlfoel} at time $t$ with initial condition $\Psi_1(0)=\vec{x}$, $\Psi_2(0)=\vec{y}$. So $\Psi_1(t)=\Phi_1(t,\vec{x}_s,\vec{y}_s(T))$, $\Psi_2(t)=\Phi_2(t,\vec{x}_s,\vec{y}_s(T))$ for some proper $\vec{y}_s$. Conservation of energy implies
\begin{equation}\label{eq:boundy}
  \abs{\vec{y}_s}^2 = 2E(T) - 2U(\vec{x}_s).
\end{equation}
From $E(T)\to E^\star$, $\vec{x}_s$ assumes value on a compact set $\partial\mathcal{U}$, we know $\abs{\vec{y}_s}^2\leqslant 2(E^\star+1) + 2\max_{\vec{x}_s\in\partial\mathcal{U}}U(\vec{x}_s)$ for large enough $T$.

Because $\Phi_1(t,\vec{x},\vec{y})$, $\Phi_2(t,\vec{x},\vec{y})$ are continuous functions of $(t,\vec{x},\vec{y})$, the triple $(t,\vec{x},\vec{y})$ assumes value in a compact set $[0,t_0]\times\partial\mathcal{U}\times\{\abs{\vec{y}}\leqslant 2(E^\star+1) + 2\max_{\vec{x}_s\in\partial\mathcal{U}}U(\vec{x}_s)\}$, the continuous function $\abs{\Phi_1-\vec{x}_c}^2 + \abs{\Phi_2}^2$ has a minimum $m$. Note that $(\vec{x}_c,0)$ is a fixed point of \eqref{eq:nlfoel}, the minimum $m$ must be positive, otherwise the uniqueness of initial value problem will be violated. So we have
\begin{equation}\label{eq:lowerbound}
  \abs{\Phi_1(t,\vec{x}_s,\vec{y}_s)-\vec{x}_c}^2 + \abs{\Phi_2(t,\vec{x}_s,\vec{y}_s)}^2\geqslant m>0.
\end{equation}
Let $t_m(T)=\arg\min_{t\in[0,t_0]}\abs{\Phi_2(t,\vec{x}_s,\vec{y}_s(T))}$, we have
\begin{equation}\label{eq:lowerPhi2}
  \liminf_{T\to+\infty}\abs{\Phi_2(t_m(T),\vec{x}_s,\vec{y}_s(T))}>0.
\end{equation}
Otherwise, there is a subsequence $T_k$ such that $\abs{\Phi_2(t_m(T_k),\vec{x}_s,\vec{y}_s(T_k))}\to 0$. By the energy conservation \eqref{eq:conservation}, we get
\begin{equation*}
\begin{aligned}
  \liminf_{k\to\infty}2E(T_k)& =2E^\star=2U(\vec{x_c})\\
  & =\liminf_{k\to\infty}\left\{\abs{\Phi_2(t_m,\vec{x}_s,\vec{y}_s(T_k))}^2+2U(\Phi_1(t_m(T_k),\vec{x}_s,\vec{y}_s(T_k)))\right\}\\
  & = 2\liminf_{k\to\infty}U(\Phi_1(t_m(T_k),\vec{x}_s,\vec{y}_s(T_k))).
\end{aligned}
\end{equation*}
So we can select a subsequence $U(\Phi_1(t_m(T_{k_l}),\vec{x}_s,\vec{y}_s(T_{k_l})))$ such that
\[\lim\limits_{l\to\infty}U(\Phi_1(t_m(T_{k_l}),\vec{x}_s,\vec{y}_s(T_{k_l}))) = U(\vec{x}_c).\]
Since the set $\{\Phi_1(t_m(T_{k_l}),\vec{x}_s,\vec{y}_s(T_{k_l}))\}$ is bounded, we can select a subsequence from $\Phi_1(t_m(T_{k_l}),\vec{x}_s,\vec{y}_s(T_{k_l}))$ which converges to a point $\vec{x}_m$. For simplicity, we still denote it by $\Phi_1(t_m(T_{k_l}),\vec{x}_s,\vec{y}_s(T_{k_l}))$. By continuity, $U(\vec{x}_c)=U(\vec{x}_m)$.  Because $\vec{x}_c$ is the unique maximum of $U(\vec{x})$ in $\mathcal{U}$, we must have $\vec{x}_m=\vec{x}_c$.

However, we know from \eqref{eq:lowerbound} that
\[0=\liminf_{l\to\infty}\abs{\Phi_1(t_m(T_{k_l}),\vec{x}_s,\vec{y}_s(T_{k_l}))-\vec{x}_c}^2 + \abs{\Phi_2(t_m(T_{k_l}),\vec{x}_s,\vec{y}_s(T_{k_l}))}^2\geqslant m>0.\]
This contradiction implies that \eqref{eq:lowerPhi2} is true.
So we have
\[\abs{\Phi_2(t,\vec{x}_s,\vec{y}_s)}\geqslant\abs{\Phi_2(t_m,\vec{x}_s,\vec{y}_s)}\geqslant m>0.\]
The proof is completed.
\end{proof}

 With the help of Lemmas \ref{lem:finitjump} and \ref{lem:bounded}, we can show that for any fixed $k$, $\phi_k'$ has bounded variation.
\begin{lemma}\label{lem:BV}
  For fixed $T_k>0$, $\phi'_k$ has bounded variation.
\end{lemma}
\begin{proof}
  It is sufficient to show
  \begin{equation}\label{eq:BVTk}
    \sup\limits_{\Delta}\sum_{j=0}^{n}\abs{\phi'_k(\alpha_{j+1})-\phi'_k(\alpha_j)}
  \end{equation}
  is bounded, where $\Delta$ is a partition of $[0,1]$, i.e.,
  \[\Delta: 0=\alpha_0<\alpha_1<\alpha_2<\dots<\alpha_{n+1}=1.\]

  By Lemma \ref{lem:finitjump}, the set $\{\alpha\in [0,1]|U(\phi_k(\alpha))=E_k\}$ is finite. Denote the elements in this set by $\alpha^k_1<\alpha^k_2<\dots<\alpha^k_N$. The summation in \eqref{eq:BVTk} can be divided into two cases.

  \textbf{Case I:} $\alpha_j\in(\alpha_i^k,\alpha_{i+1}^k)$ for some $0\leqslant i\leqslant N-1$ and $\alpha_{j+1}\notin (\alpha_i^k,\alpha_{i+1}^k)$. Denote the index set for $j$ in this case by I.

By Lemma \ref{lem:finitjump}, \#$\{j\in \mathrm{I}\}\leqslant N$. We have
  \begin{equation}\label{eq:caseI}
    \sum_{j\in\mathrm{I}}\abs{\phi'_k(\alpha_{j+1})-\phi'_k(\alpha_j)}\leqslant 2NM.
  \end{equation}

  \textbf{Case II:} $\alpha_j, \alpha_{j+1}\in (\alpha_i^k,\alpha_{i+1}^k)$.   Denote the index set for $j$ in this case by II.

  For $j\in \mathrm{II}$, $\phi_k\in C^2[\alpha_j,\alpha_{j+1}]$. We have
\[
  \begin{aligned}
    \sum_{j\in\mathrm{II}}\abs{\phi'_k(\alpha_{j+1})-\phi'_k(\alpha_j)}&=\sum_{j\in\mathrm{II}}\abs{\int_{\alpha_j}^{\alpha_{j+1}}\phi''_k\d\alpha}\\
    &\leqslant\sum_{l=1}^{N}\int_{\alpha^k_l}^{\alpha^k_{l+1}}\abs{\phi''_k}\d\alpha\\
    &\leqslant\sum_{l=1}^{N}\int_{t^k_l}^{t^k_{l+1}}\sqrt{\Theta_k(t)}\d t\\
    & = \int_{0}^{T_k}\sqrt{\Theta_k(t)}\d t,
  \end{aligned}
\]
  where $t^k_l = \ell_k^{-1}(\alpha^k_l)$. $\ell_k$ is defined in \eqref{eq:repa}. By Lemma \ref{lem:bounded}, $\Theta(t)$ is continuous so it is bounded by $M_{\Theta}$ in $[0,T_k]$. This leads to
  \[\sum_{j\in\mathrm{II}}\abs{\phi'_k(\alpha_{j+1})-\phi'_k(\alpha_j)}\leqslant \sqrt{M_{\Theta}}T_k.\]
Combining with \eqref{eq:caseI}, $\phi'_k$ has bounded variation.
\end{proof}

By Proposition \ref{prop:conv}, we know that $\phi_k$, which has the same graph as $\psi_k$, tends to $\phi^\star$ passing through a critical point $\vec{x}_c$ when $T\rightarrow\infty$. We will show that $\phi_k'$ has uniformly bounded variation in a neighborhood of $\vec{x}_c$. This will be done through linearization analysis and Hartman-Grobman theorem.

In the neighborhood of $(\vec{x}_c,0)$, The nonlinear system \eqref{eq:nlfoel} can be well-approximated by its linearizaion
\begin{equation}\label{eq:lfosys}
  \begin{aligned}
  \dot{\vec{x}}_1 &= \vec{x}_2, \\
  \dot{\vec{x}}_2 &= A^2\vec{x}_1,
  \end{aligned}
\end{equation}
where $A^2=-\nabla^2 U(\vec{x}_c)$. Because all eigenvalues of $\nabla^2 U(\vec{x}_c)$ are negative, we may denote the eigenvalues of $A$ by $0<\mu_1\leqslant \mu_2\leqslant\dots\leqslant \mu_d$, the corresponding unit orthogonal eigenvectors by $\xi_1, \xi_2\dots,\xi_d$. We also use $\tilde{\vec{x}}$ to denote the extended variable $\tilde{\vec{x}}=(\vec{x}_1,\vec{x}_2)$. For the linearized system \eqref{eq:lfosys}, we have the following lemma.
\begin{lemma}\label{lem:unibounded}
  The solution of  \eqref{eq:lfosys} with boundary condition $\vec{x}_1(0)=\vec{x}_s$, $\vec{x}_1(T)=\vec{x}_f$ is
\begin{equation}\label{eq:solution}
  \begin{aligned}
    \vec{x}_1(t) & =  (\me^{AT}-\me^{-AT})^{-1}[\me^{At}(\vec{x}_f-\me^{-AT}\vec{x}_s) + \me^{A(T-t)}(\vec{x}_s-\me^{-AT}\vec{x}_f)], \\
    \vec{x}_2(t) & =  (\me^{AT}-\me^{-AT})^{-1}A[\me^{At}(\vec{x}_f-\me^{-AT}\vec{x}_s) - \me^{A(T-t)}(\vec{x}_s-\me^{-AT}\vec{x}_f)].
  \end{aligned}
\end{equation}
For any $\vec{x}_s\neq 0, \vec{x}_f\neq 0, \vec{x}_s\neq\vec{x}_f$, the integral
  \[\int_{0}^{T}\sqrt{\Theta_T^L(t)}\d t\]
  is uniformly bounded with respect to $T$, where
  \begin{equation}\label{eq:ThetaL}
    \Theta_T^L(t) = \frac{(\vec{x}_1^\T A^4\vec{x}_1)|\vec{x}_2|^2 - (\vec{x}_1^\T A^2\vec{x}_2)^2}{\abs{\vec{x}_2}^4}
  \end{equation}
\end{lemma}
\begin{proof}
It is straightforward to check that \eqref{eq:solution} is the solution of the boundary value problem. The function $\Theta_T^L(t)$ is a special case of $\Theta_T(t)$ defined in \eqref{eq:Theta} by taking $U(\vec{x}) = -\vec{x}^\T A^2\vec{x}$. By Lemma \ref{lem:bounded}, we only need to consider the case that $T$ is sufficiently large. For simplicity, we denote $\Theta_T^L(t)=F(t)/|\vec{x}_2|^4$, where
\begin{equation}\label{eq:Ft}
  F(t) = (\vec{x}_1^\T A^4\vec{x}_1)|\vec{x}_2|^2-(\vec{x}_1^TA^2\vec{x}_2)^2.
\end{equation}
Denote $x_s^i=\vec{x}_s^\T\xi_i, x_f^i=\vec{x}_f^\T\xi_i, i=1,2,\dots d$. We have explicit form of $F(t)$ and $|\vec{x}_2|^2$

\[
\begin{aligned}
F(t)  = & \sum_{i<j}^{\,} \Big\{ \mu_i^2\mu_j(\me^{\mu_iT}-\me^{-\mu_iT})^{-1}(\me^{\mu_jT}-\me^{-\mu_jT})^{-1}[p_i\me^{\mu_i t}+q_i\me^{\mu_i (T-t)}]\cdot\\
&\quad [p_j\me^{\mu_j t}-q_j\me^{\mu_j (T-t)}] - \mu_i\mu_j^2(\me^{\mu_iT}-\me^{-\mu_iT})^{-1}(\me^{\mu_jT}-\me^{-\mu_jT})^{-1}\cdot\\
& \quad [p_i\me^{\mu_i t}-q_i\me^{\mu_i (T-t)}][p_j\me^{\mu_j t}+q_j\me^{\mu_j (T-t)}]\Big\}^2\\
=:  &\sum_{i< j} f_{ij}^2(t),\\
|\vec{x}_2|^2 = &\sum_{i} \mu_i^2(\me^{\mu_iT}-\me^{-\mu_iT})^{-2}[p_i\me^{\mu_i t}-q_i\me^{\mu_i (T-t)}]^2,
\end{aligned}
\]
where $p_i=x_f^i-\me^{-\mu_iT}x_s^i$, $q_i = x_s^i-\me^{-\mu_iT}x_f^i$. Clearly, if $x_s^i, x_f^i\neq 0$, when $T$ gets large enough, $p_i, q_i\neq 0$ and they are uniformly bounded.

We first show that in $[0,T/2]$, $\int_{0}^{T/2}\sqrt{\Theta_T^L(t)}\d t$ is uniformly bounded. We have
\[ \sqrt{\Theta_T^L(t)} = \frac{\sqrt{\sum_{i<j}f_{ij}^2}}{|\vec{x}_2|^2}\leqslant C\frac{\max_{i<j}\abs{f_{ij}}}{|\vec{x}_2|^2}\leqslant C\sum_{i<j}\frac{\abs{f_{ij}}}{|\vec{x}_2|^2}.\]
It is sufficient to show for each pair $i<j$,
\[\int_{0}^{T/2}\frac{\abs{f_{ij}}}{|\vec{x}_2|^2}\d t\]
is uniformly bounded.

The key idea of the proof is to verify  that $f_{ij}$ and $|\vec{x}_2|^2$ can be dominated by an exponential function. For given $\vec{x}_s$, let  us denote the sets
\[S=\{i|x_s^i\neq 0\}, \quad\quad \bar{S}=\{1,2,\dots, d\}\backslash S.\]
 Since $\vec{x}_s\neq 0$, $S$ is not empty. For different choices of $\vec{x}_s$ and $\vec{x}_f$, we divide the proof into 4 cases.

{\bf Case 1: $i, j\in \bar{S}, i\neq j$.}

If $\bar{S}$ is not empty, for every $i\in \bar{S}$, we can always assume $x_f^i\neq 0$. Otherwise $f_{ij}=0$ for any $j$. If $\mu_i=\mu_j$, $f_{ij}=0$. For $\mu_i<\mu_j$
\[
  \abs{f_{ij}}\leqslant C\me^{(\mu_i+\mu_j)(t-T)},
\]
where $C$ is a generic positive constant independent of $T$. At the same time,
\[
  |\vec{x}_2|^2\geqslant C \me^{2\mu_i(t-T)}.
\]
Thus
\begin{equation}\label{eq:intininSjninS}
  \begin{aligned}
    \int_{0}^{T/2}\frac{\abs{f_{ij}}}{|\vec{x}_2|^2}\d t& \leqslant C\int_{0}^{T/2}\me^{(\mu_j-\mu_i)(t-T)}\d t \\
         & =\frac{C}{\mu_j-\mu_i}(\me^{-\frac{1}{2}(\mu_j-\mu_i)T}-\me^{-(\mu_j-\mu_i)T})\leqslant \frac{2C}{\mu_j-\mu_i},
  \end{aligned}
\end{equation}
which is uniformly bounded.

{\bf Case 2.1: $i\neq j\in S$, $\mu_i<\mu_j$.}

In this case, $f_{ij}$ can be bounded by an exponential function
\begin{equation}\label{eq:fijiinSjinS}
  \abs{f_{ij}}\leqslant C\me^{-(\mu_i+\mu_j)t}.
\end{equation}
The denominator $|\vec{x}_2|^2$ can be estimated by
\begin{equation}\label{eq:x2iinSjinS}
  |\vec{x}_2|^2\geqslant \mu_i^2[p_i\me^{\mu_i(2t-T)}-q_i]^2\me^{-2\mu_i t} + \mu_j^2[p_j\me^{\mu_j(2t-T)}-q_j]^2\me^{-2\mu_j t}.
\end{equation}

If for every $t\in [0,T/2]$, $(p_i\me^{\mu_i(2t-T)}-q_i)^2>0$, there is a positive lower bounded such that
\[|\vec{x}_2|^2\geqslant C\me^{-2\mu_it}.\]
We have
\begin{equation}\label{eq:intiinSjinS1}
  \begin{aligned}
    \int_{0}^{T/2}\frac{\abs{f_{ij}}}{|\vec{x}_2|^2}\d t& \leqslant C\int_{0}^{T/2}\me^{-(\mu_j-\mu_i)t}\d t \\
         & =\frac{C}{\mu_j-\mu_i}(1-\me^{-\frac{1}{2}(\mu_j-\mu_i)T})\leqslant \frac{C}{\mu_j-\mu_i}.
  \end{aligned}
\end{equation}

If there is a $t_i\in [0,T/2]$ such that $p_i\me^{\mu_it_i}-q_i\me^{\mu_i(T-t_i)}=0$, then we define $\rho_i = (q_i/p_i)^{\frac{1}{\mu_i}}$, $\rho_j = (q_j/p_j)^{\frac{1}{\mu_j}}$ if $p_jq_j>0$. The constants $\rho_i, \rho_j$ are uniformly bounded with respect to $T$. We have $\me^{2t_i}=\rho_i\me^T$. We will show that in a neighborhood of $t_i$, the integral is uniformly bounded.

If $\rho_i=\rho_j=\rho$,  we can obtain the estimation
\begin{equation}\label{eq:fijiinSjinS2}
  \abs{f_{ij}}\leqslant 4\mu_i\mu_j\abs{p_i}\abs{p_j}\rho^{\frac{1}{2}(\mu_i+\mu_j)}(\mu_j-\mu_i)(t-t_i)^2\me^{-\frac{1}{2}(\mu_j+\mu_i)T} + o(t-t_i)^2\me^{-\frac{1}{2}(\mu_j+\mu_i)T}.
\end{equation}
The denominator
\begin{equation}\label{eq:x2iinSjinS2}
  |\vec{x}_2|^2\geqslant C(t-t_i)^2\me^{-2\mu_iT} + o(t-t_i)^2\me^{-2\mu_iT}.
\end{equation}
So there is a $\delta>0$ independent of $T$, in $[t_i-\delta,t_i+\delta]$,
\[\frac{\abs{f_{ij}}}{|\vec{x}_2|^2}\leqslant \frac{(C+o(1))\me^{-\frac{1}{2}(\mu_j-\mu_i)T}}{C+o(1)}\leqslant 1.\]
The integration $\int_{t_i-\delta}^{t_i+\delta}\abs{f_{ij}}/|\vec{x}_2|^2\d t$ is uniformly bounded.

If $\rho_i\neq \rho_j$, there is a neighborhood of $t_i$ such that $(p_j\me^{\mu_j(2t-T)}-q_j)^2>0$. With Taylor expansion near $t_i$, we obtain
\begin{equation}\label{eq:fijiinSjinS3}
  \abs{f_{ij}}\leqslant C\me^{-\frac{1}{2}(\mu_i+\mu_j)T}.
\end{equation}
\begin{equation}\label{eq:x2iinSjinS3}
  |\vec{x}_2|^2\geqslant (C^2+o(1))(t-t_i)^2\me^{-\mu_iT} + \me^{-\mu_jT}.
\end{equation}
So there is a $\delta>0$ such that in $[t_i-\delta,t_i+\delta]$,
\[|\vec{x}_2|^2\geqslant C(t-t_i)^2\me^{-\mu_iT} + C\me^{-\mu_jT}.\]
In all, the integration
\begin{equation}\label{eq:intiinSjinSlocal}
\begin{aligned}
  \int_{t_i-\delta}^{t_i+\delta}\frac{\abs{f_{ij}}}{|\vec{x}_2|^2}\d t\leqslant & C\int_{t_i-\delta}^{t_i+\delta}\frac{\me^{-\frac{1}{2}(\mu_j-\mu_i)T}}{C^2(t-t_i)^2+\me^{-(\mu_j-\mu_i)T}}\d t\\
 = &2\int_{0}^{\delta}\frac{C\d \me^{\frac{1}{2}(\mu_j-\mu_i)T}s}{1+(C\me^{\frac{1}{2}(\mu_j-\mu_i)T}s)^2}\d t\\
 = &C\int_{0}^{C\delta \me^{\frac{1}{2}(\mu_j-\mu_i)T}}\frac{\d s}{1+s^2}\leqslant C\frac{\pi}{2}
\end{aligned}
\end{equation}
is uniformly bounded.

To summarize, we have shown that there exists a $\delta>0$, the integration in the neighborhood $[t_i-\delta,t_i+\delta]$ is uniformly bounded. Outside this interval, we have
\[|\vec{x}_2|^2\geqslant C\me^{-2\mu_it}.\]
Thus
\begin{equation}\label{eq:intiinSjinS3}
\begin{aligned}
  \int_{0}^{T/2}\frac{\abs{f_{ij}}}{|\vec{x}_2|^2}\d t &=\int_{0}^{t_i-\delta} + \int_{t_i-\delta}^{t_i+\delta} +\int_{t_i+\delta}^{T/2}\frac{\abs{f_{ij}}}{|\vec{x}_2|^2}\d t \\
  &  \leqslant C + \int_{0}^{T/2}\me^{-(\mu_j-\mu_i)t}\d t\leqslant C + \frac{1}{\mu_j-\mu_i}.
  \end{aligned}
\end{equation}

{\bf Case 2.2: $i\neq j\in S$, $\mu_i=\mu_j=\mu$.}

The proof in this case is similar as  Case 2.1. The main difference is that the estimation of $\abs{f_{ij}}$ is replaced by
\begin{equation}\label{eq:fijiinSjinS4}
  \abs{f_{ij}}\leqslant C\me^{-\mu T}.
\end{equation}
If there exists $t_i$ such that $p_i\me^{\mu t_i}-q_i\me^{\mu(T-t_i)}=0$, the estimation \eqref{eq:fijiinSjinS2}, \eqref{eq:x2iinSjinS2}, \eqref{eq:fijiinSjinS3} and \eqref{eq:x2iinSjinS3} imply that $\abs{f_{ij}}/|\vec{x}_2|^2$ is uniformly bounded in $[t_i-\delta,t_i+\delta]$.  Outside this interval,
\[|\vec{x}_2|^2\geqslant C\me^{-2\mu t}.\]
Thus
\begin{equation}\label{eq:intiinSjinS4}
\begin{aligned}
  \int_{0}^{T/2}\frac{\abs{f_{ij}}}{|\vec{x}_2|^2}\d t & =\int_{0}^{t_i-\delta} + \int_{t_i-\delta}^{t_i+\delta} +\int_{t_i+\delta}^{T/2} \frac{\abs{f_{ij}}}{|\vec{x}_2|^2}\d t\\
  & \leqslant C + \me^{-\mu T}\int_{0}^{T/2}\me^{2\mu t}\d t\leqslant C + \frac{1}{\mu}.
\end{aligned}
\end{equation}

{\bf Case 3: $i\in S$, $j\in\bar{S}$.}

In this case, we can estimate $f_{ij}$ and $|\vec{x}_2|^2$ by
\begin{equation*}
  \begin{aligned}
   \abs{f_{ij}}  \leqslant \Big|\mu_ix_f^i & [p_i\me^{\mu_i (2t-T)}+q_i][\me^{\mu_j t}-\me^{-\mu_j t}] - \mu_j x_f^j[p_i\me^{\mu_i (2t-T)}-q_i][\me^{\mu_j t}+\me^{-\mu_j t}]\Big|\cdot\\
   & \me^{-\mu_it}\me^{-\mu_jT} \leqslant C\me^{-\mu_i t}\me^{\mu_j(t-T)}.
  \end{aligned}
\end{equation*}

\begin{equation}\label{eq:x2iinSjninS}
  |\vec{x}_2|^2\geqslant \mu_i^2[p_i\me^{\mu_i(2t-T)}-q_i]^2\me^{-2\mu_i t} + \mu_j^2\abs{x_f^j}^2(\me^{\mu_j t}+\me^{-\mu_j t})^2\me^{-2\mu_jT}.
\end{equation}

If there exists $t_i\in [0,T/2]$ such that $p_i\me^{\mu_i(2t-T)}-q_i=0$, we can take similar argument as in Case 2.1. So in a $\delta$-neighborhood of $t_i$, the integration is uniformly bounded. Outside this neighborhood, we have
\[|\vec{x}_2|^2\geqslant C(\me^{-2\mu_i t} + \me^{2\mu_j (t-T)}).\]
Since $\mu_j\geqslant\mu_i$, $t\leqslant T/2$, we have $\me^{-2\mu_i t}\geqslant \me^{2\mu_j (t-T)}$. Thus
\begin{equation}\label{eq:intiinSjninS}
\begin{aligned}
  \int_{0}^{T/2}\frac{\abs{f_{ij}}}{|\vec{x}_2|^2}\d t& = \int_{0}^{t_i-\delta} + \int_{t_i-\delta}^{t_i+\delta} + \int_{t_i+\delta}^{T/2}\frac{\abs{f_{ij}}}{|\vec{x}_2|^2}\d t \\
  &\leqslant C+C\int_{0}^{T/2}\me^{(\mu_i+\mu_j)t}\me^{-\mu_j T}\d t\\
& \leqslant C + \frac{C}{\mu_i+\mu_j}(\me^{-(\mu_j-\mu_i)T}-\me^{-\mu_jT})\\
& \leqslant C+\frac{2C}{\mu_i+\mu_j}.
\end{aligned}
\end{equation}

{\bf Case 4: $i\in\bar{S}$, $j\in S$.}

As in Case 3, we have
\begin{align}
  \abs{f_{ij}} &\leqslant C\me^{-\mu_j t}\me^{\mu_i(t-T)}, \label{eq:fijininSjinS}\\
  |\vec{x}_2|^2 & \geqslant \mu_j^2[p_j\me^{\mu_j(2t-T)}-q_j]^2\me^{-2\mu_j t} + \mu_i^2\abs{x_f^i}^2(\me^{\mu_i t}+\me^{-\mu_i t})^2\me^{-2\mu_iT}. \label{eq:x2ininSjinS}
\end{align}

If there exists a $t_j\in [0,T/2]$ such that $p_j\me^{\mu_j(2t_j-T)}-q_j=0$, similar argument as the Case 2.1 holds.  We have boundedness of the integrand in a $\delta$-neighborhood of $t_j$, and
\[|\vec{x}_2|^2\geqslant C(\me^{-2\mu_j t} + \me^{2\mu_i (t-T)})\]
outside the $\delta$-neighborhood.

Denote $t_c=\frac{\mu_iT}{\mu_i+\mu_j}$. When $t\leqslant t_c$, $\me^{-2\mu_j t}\geqslant \me^{2\mu_i (t-T)}$. When $t\geqslant t_c$, $\me^{-2\mu_j t}\leqslant \me^{2\mu_i (t-T)}$. So we have the estimate
\[
\begin{aligned}
  \int_{0}^{T/2}\frac{\abs{f_{ij}}}{|\vec{x}_2|^2}\d t & \leqslant \int_{t_j-\delta}^{t_j+\delta}\frac{\abs{f_{ij}}}{|\vec{x}_2|^2}\d t  + C\int_{0}^{t_c}\me^{(\mu_i+\mu_j)t}\me^{-\mu_iT}\d t + C\int_{t_c}^{T/2}\me^{-\mu_jt}\me^{-\mu_i(t-T)}\d t\\
& \leqslant C+ \frac{C}{\mu_i+\mu_j}\me^{-\mu_iT}(\me^{\mu_iT}-1) + \frac{C}{\mu_i+\mu_j}\me^{\mu_iT}(\me^{-\mu_iT} - \me^{-\frac{1}{2}(\mu_i+\mu_j)T})\\
& \leqslant C+ \frac{C}{\mu_i+\mu_j}.
\end{aligned}
\]

So far, we have shown that $\int_{0}^{T/2}\sqrt{\Theta_T^L(t)}\d t$ is uniformly bounded. As for the interval $[T/2,T]$, we can define $s=T-t$, repeat previous discussions to show that $\int_{T/2}^{T}\sqrt{\Theta_T^L(t)}\d t$ is also uniformly bounded. The proof is done.
\end{proof}

In order to show the uniformly bounded variation of $\phi_k'$ in the neighborhood of a critical point $\vec{x}_c$,  we employ a strengthened version of Hartman-Grobman theorem to control the nonlinearity\cite{Zhang2017Differentiability}. Let $T^t(\tilde{\vec{x}}_0)$, $L^t(\tilde{\vec{x}}_0)$ be the solution of the nonlinear  system \eqref{eq:nlfoel}
or the linear system \eqref{eq:lfosys}, respectively, at time $t$ with the initial condition $\tilde{\Psi}(0)=\tilde{\vec{x}}_0$ or  $\tilde{\vec{x}}(0)=\tilde{\vec{x}}_0$.
\begin{lemma}\label{lem:ZZL}
  There exists a neighborhood $\mathcal{U}$ of $\tilde{\Psi}_c$ and a homeomorphism $H:\mathcal{U}\to \mathbb{R}^{2d}$ such that $HT^t=L^{t}H$ for any $t>0$ and
  \begin{equation}\label{eq:1+beta}
    H(\tilde{\Psi}) = \tilde{\Psi} - \tilde{\Psi}_c + O(|\tilde{\Psi}-\tilde{\Psi}_c|^{1+\beta}),\quad H^{-1}(\tilde{\vec{x}}) = \tilde{\vec{x}} + \tilde{\Psi}_c +O(\abs{\tilde{\vec{x}}}^{1+\beta})
  \end{equation}
  for some constant $\beta\in (0,1)$.
\end{lemma}
\begin{proof}
  This is a corollary of Theorem 7.1 in  \cite{Zhang2017Differentiability}. Without loss of generality, we assume $\tilde{\Psi}_c=0$. Otherwise we can define a shift $S:\mathcal{U}\to\mathbb{R}^{2d}$, $S(\tilde{\Psi})=\tilde{\Psi}-\tilde{\Psi}_c$. Since $\tilde{\Psi}_c$ is invariant under $T^t$, we have $ST^tS^{-1}(0)=0$. We can consider the homeomorphism $\tilde{H}=HS$.

We first consider the 1-time solution $T^1$. By Theorem 7.1 in \cite{Zhang2017Differentiability}, there is a homeomorphism $H_0$ which satisfies $H_0(0)=0$, $H_0T^1=L^1H_0$ and
\[H_0(\tilde{\vec{x}}) = \tilde{\vec{x}} + O(\abs{\tilde{\vec{x}}}^{1+\beta}),\quad H_0^{-1}(\tilde{\vec{x}}) = \tilde{\vec{x}} +O(\abs{\tilde{\vec{x}}}^{1+\beta}).\]
  Then we define
  \begin{equation}\label{eq:H}
  H = \int_{0}^{1}L^{-s}H_0T^s\d s.
  \end{equation}
  We first verify $HT^t=L^tH$ for any $t>0$. Indeed,
\begin{equation*}
  L^tH = \int_{0}^{1}L^{t-s}H_0T^{s-t}\d s T^{t} = \left(\int_{-t}^{0}+\int_{0}^{1-t}\right)L^{-s}H_0T^s\d s T^{t}.
\end{equation*}
Since $H_0=L^{-1}H_0T$, the first term is
\begin{equation*}
  \int_{-t}^{0}L^{-s}H_0T^s\d t = \int_{-t}^{0}L^{-s-1}H_0T^{s+1}\d s=\int_{1-t}^{1}L^{-s}H_0T^s\d s.
\end{equation*}
So we obtain
\[
  L^tH=\int_{0}^{1}L^{-s}H_0T^s\d s T^{t}=HT^{t}.
\]

 To show that $H$ also satisfies condition \eqref{eq:1+beta}, we utilize Taylor expansion in a neighborhood of $\tilde{\Psi}_c$
  \[T^s\tilde{\vec{{x}}} = \nabla T^s(\tilde{\Psi}_c) \tilde{\vec{x}} + O(\abs{\tilde{\vec{x}}}^2).\]
  Note that $T^t(\tilde{\vec{x}})$ is the solution of \eqref{eq:nlfoel}, we have
  \[\dd{\,}{t}\nabla T^t(\tilde{\vec{x}}) = \nabla\dd{T^t}{t} = \nabla \begin{bmatrix}
                                                             \Psi_2 \\
                                                             -\nabla U(\Psi_1)
                                                           \end{bmatrix}
                                                           =\begin{bmatrix}
                                                              0 & I \\
                                                              -\nabla U(\Psi_1) & 0
                                                            \end{bmatrix}\nabla T^t(\tilde{\vec{x}}).
  \]
  Taking $\vec{\tilde{x}}=\tilde{\Psi}_c$, we obtain
  \[\dd{\,}{t}\nabla T^t(\tilde{\Psi}_c) = J\nabla T^t(\tilde{\Psi}_c),\]
  where
  \[J=\begin{bmatrix}
        0 & I \\
        A^2 & 0
      \end{bmatrix}.\]
Based on the fact $T^t(\tilde{\Psi}_c)=\tilde{\Psi}_c$, we get $\nabla T^t(\tilde{\Psi}_c) = \me^{Jt}$ and
 \[T^s\tilde{\vec{{x}}} = \me^{Js}\tilde{\vec{x}} + O(\abs{\tilde{\vec{x}}}^2).\]
We obtain
\[
 H_0T^s\tilde{x}=T^s\tilde{x} + O(\abs{T^s\tilde{\vec{x}}}^{1+\beta}) = \me^{Js} \tilde{\vec{x}} + O(\abs{\tilde{\vec{x}}}^{1+\beta}).
\]
 Substituting the above into \eqref{eq:H} and note that $L^s=\me^{Js}$, we get
\[
 H = \int_{0}^{1}L^{-s}H_0T^s\d s = \tilde{\vec{x}} + O(\abs{\tilde{\vec{x}}}^{1+\beta}).
\]

 We now turn to $H^{-1}$. We can check that
 \begin{equation}\label{eq:H-1}
 H^{-1}=\int_{0}^{1}T^{-s}H_0^{-1}L^s\d s.
 \end{equation}
From \eqref{eq:H} we obtain
 \[H\int_{0}^{1}T^{-s}H_0^{-1}L^s\d s = \int_{0}^{1}\int_{0}^{1}L^{-t}H_0T^{t-s}H_0^{-1}L^s\d t\d s.\]
\eqref{eq:H-1} follows by noting that $L^{t-s}=H_0T^{t-s}H_0^{-1}$. Similar procedure can be applied to get the estimate $H^{-1}(\tilde{\vec{x}}) = \tilde{\vec{x}}+O(\abs{\tilde{\vec{x}}}^{1+\beta})$.
\end{proof}

\begin{lemma}\label{lem:nlunibounded}
  Let $\mathcal{U}$ be the neighborhood of  $\tilde{\Psi}_c$ ensured by Lemma \ref{lem:ZZL}. Then
\begin{equation}\label{eq:intTheta}
  \int_{0}^{T}\sqrt{\Theta_T(t)}\d t
\end{equation}
is uniformly bounded with respect to $T$ for any initial $\vec{x}_s\in\partial\mathcal{U}$ and terminal $\vec{x}_f\in\partial\mathcal{U}$.
\end{lemma}
\begin{proof}
By Lemma \ref{lem:bounded}, we only need to consider the case $T\to +\infty$. Since $\psi_T$ is uniformly bounded, for any  finite $t_0$, $|\dot{\psi}_T|$ has a positive lower bound and thus
  \[\lim_{T\to\infty}\int_{0}^{t_0}\sqrt{\Theta_T(t)}\d t\leqslant C.\]
So we will only consider the case when $T$ and $t$ are both sufficiently large.

  In the neighborhood $\mathcal{U}$, we can estimate $\Theta_T(t)$ using the linearized version $\Theta_T^L(t)$. By Lemma \ref{lem:ZZL}, there exists $\beta\in (0,1)$ such that
\[
  \begin{aligned}
    \Psi_1 &= \vec{x}_1 + \vec{x}_c + O(r^{1+2\beta}),\\
    \Psi_2 &= \vec{x}_2 + O(r^{1+2\beta}),
  \end{aligned}
\]
where $r=\sqrt{|\vec{x}_1|^2+|\vec{x}_2|^2}$. A direct calculation shows
\[\sqrt{\Theta_T(t)} = \frac{\sqrt{F(t)+G_1(t)}}{|\vec{x}_2|^2 + G_2(t)}.\]
The functions $G_1=O(r^{4+2\beta})$, $G_2=O(r^{2+2\beta})$. As in the proof of Lemma \ref{lem:unibounded}, we first consider the interval $[0,T/2]$. $r^2$ can be explicitly written as
\[
\begin{aligned}
  r^2  = & \sum_{i\in S}(1-\me^{-2\mu_iT})^{-2}[(p_i\me^{\mu_i(2t-T)}+q_i)^2+\mu_i^2(p_i\me^{\mu_i(2t-T)}-q_i)^2]\me^{-2\mu_it}\\
 & + \sum_{i\in\bar{S}}(1-\me^{-2\mu_iT})^{-2}[(1-\me^{-2\mu_it})^2+\mu_i^2(1+\me^{-2\mu_it})^2]\me^{2\mu_i(t-T)}.
\end{aligned}
\]
Note that the coefficients of $\me^{-2\mu_i t}$ and $\me^{2\mu_i(t-T)}$ can not be zero, so there are constants $c>0, C>0$ such that
\[c\left(\sum_{i\in S}\me^{-2\mu_i t} + \sum_{i\in\bar{S}}\me^{2\mu_i(t-T)}\right)\leqslant r^2\leqslant C\left(\sum_{i\in S}\me^{-2\mu_i t} + \sum_{i\in\bar{S}}\me^{2\mu_i(t-T)}\right).\]
If we denote $m=\min S$, $n=\min \bar{S}$, we have
\begin{equation}\label{eq:r2lowup}
  c(\me^{-2\mu_m t} + \me^{2\mu_n(t-T)})\leqslant r^2\leqslant C(\me^{-2\mu_m t} + \me^{2\mu_n(t-T)})
\end{equation}
when $t$ is sufficiently large.

Recall that
\[
\begin{aligned}
   |\vec{x}_2|^2 = & \sum_{i\in S}\mu_i^2(1-\me^{-2\mu_iT})^{-2}[p_i\me^{\mu_i(2t-T)}-q_i]^2\me^{-2\mu_it}\\
 & + \sum_{i\in\bar{S}}\mu_i^2(1-\me^{-2\mu_iT})^{-2}[1+\me^{-2\mu_it}]^2\me^{2\mu_i(t-T)}.
\end{aligned}
\]

{\bf Case 1: $\mu_m>\mu_n$}. Let
\[t_c=\frac{\mu_nT}{\mu_n+\mu_m}<\frac{T}{2}.\]
 It is easy to check  that $[p_i\me^{\mu_i(2t-T)}-q_i]^2$ has a positive lower bound when $t<t_c$. In $[0,t_c]$, $\me^{-2\mu_mt}\geqslant \me^{2\mu_n(t-T)}$. So we have
\[r^{2(1+\beta)}\leqslant C\me^{-2(1+\beta)\mu_m t}\leqslant C \me^{-2\mu_mt}\leqslant \frac{1}{2}|\vec{x}_2|^2\]
when $t$ is sufficiently large. In $[t_c,T/2]$, $\me^{-2\mu_mt}\leqslant \me^{2\mu_n(t-T)}$. Since $T-t\geqslant T/2$, we have
\[r^{2(1+\beta)}\leqslant C\me^{2(1+\beta)\mu_n (t-T)}\leqslant C \me^{2\mu_n(t-T)}\leqslant \frac{1}{2}|\vec{x}_2|^2\]
when $T$ is large enough. Thus we obtain $\abs{G_2}\leqslant\frac{1}{2}|\vec{x}_2|^2$, which yields
\begin{equation}\label{eq:intmgn}
\begin{aligned}
  \int_{0}^{T/2}\sqrt{\Theta_T(t)}\d t \leqslant & C\int_{0}^{T/2}\frac{\sqrt{F}}{|\vec{x}_2|^2} + \frac{\sqrt{\abs{G_1}}}{|\vec{x}_2|^2}\d t\\
  \leqslant & C\int_{0}^{T/2}\sqrt{\Theta_T^L(t)}\d t + C\int_{0}^{t_c}\me^{-2\mu_m(1+\beta)t + 2\mu_mt}\d t\\
  & + C\int_{t_c}^{T/2} \me^{2\mu_n(1+\beta)(t-T)-2\mu_n(t-T)}\d t\\
  \leqslant & C\int_{0}^{T/2}\sqrt{\Theta_T^L(t)}\d t + \frac{C}{2\beta\mu_m} + \frac{C}{2\beta\mu_n}.
\end{aligned}
\end{equation}
By Lemma \ref{lem:unibounded},  we obtain that $\int_{0}^{T/2}\sqrt{\Theta_T(t)}\d t$ is uniformly bounded.

{\bf Case 2: $\mu_m\leqslant\mu_n$}. We have $\me^{-2\mu_mt}\geqslant \me^{2\mu_n(t-T)}$. Eq. \eqref{eq:r2lowup} yields $r^2=O(\me^{-2\mu_mt})$. If there exists $t_m$ such that $p_m\me^{\mu_m(2t-T)}-q_m=0$, then in a $\delta$-neighborhood of $t_m$,
\[\sqrt{\Theta(t)}\leqslant\frac{F(t)+O(\me^{-\mu_m(1+\beta)T})}{C^2(t-t_m)^2\me^{-\mu_mT} + \me^{-\mu_nT} + O(\me^{-\mu_m(1+\beta)T})}.\]
Applying similar argument as in the proof of Case 2.1, Lemma \ref{lem:unibounded}, we can show that $\int_{t_m-\delta}^{t_m+\delta}\sqrt{\Theta(t)}\d t$ is uniformly bounded. Outside this $\delta$-neighborhood, we have
\[r^{2(1+\beta)}\leqslant C\me^{-2(1+\beta)\mu_mt}\leqslant C \me^{-2\mu_mt}\leqslant \frac{1}{2}|\vec{x}_2|^2\]
for  sufficiently large $t$. So the integration
\begin{equation*}
\begin{aligned}
  \int_{0}^{T/2}\sqrt{\Theta_T(t)}\d t& \leqslant C\int_{0}^{T/2}\frac{\sqrt{F}}{|\vec{x}_2|^2} + \frac{\sqrt{G}_1}{|\vec{x}_2|^2}\d t\\
  &\leqslant C\int_{0}^{T/2}\sqrt{\Theta_T^L(t)}\d t + C\int_{t_m-\delta}^{t_m+\delta} + C\int_{0}^{T/2}\me^{-2\mu_m(1+\beta)t + 2\mu_mt}\d t\\
  & \leqslant C\int_{0}^{T/2}\sqrt{\Theta_T^L(t)}\d t + C +\frac{C}{2\beta\mu_m}
\end{aligned}
\end{equation*}
is also uniformly bounded.

In all, we have shown that $\int_{0}^{T/2}\sqrt{\Theta_T(t)}\d t$ is uniformly bounded. Similar argument applies to $\int_{T/2}^{T}\sqrt{\Theta_T(t)}\d t$.
\end{proof}

Lemma \ref{lem:nlunibounded} shows that  in a neighborhood of critical point $\vec{x}_c$, $\phi_k'$ has uniformly bounded variation.
\begin{lemma}\label{lem:UBV}
  Suppose that the graph limit $\phi^\star$ passes through a critical point $\vec{x}_c$ at $\alpha_c$. Then there is an interval $[\alpha_{-},\alpha_{+}]$ which contains $\alpha_c$ such that  $\{\phi'_k(\alpha)\}$ has uniformly bounded variation in this interval.
\end{lemma}
\begin{proof}
  Denote $\mathcal{U}=\{\vec{x}|\abs{\vec{x}-\vec{x}_c}<2\delta\}$ in which Lemma \ref{lem:nlunibounded} holds, and $\mathcal{V}=\{\vec{x}|\abs{\vec{x}-\vec{x}_c}<\delta\}\subset \mathcal{U}$. Define $(\alpha_{-},\alpha_{+})$ to be an  interval satisfying the following two conditions: (1) $\alpha_c\in (\alpha_{-},\alpha_{+})$; (2) $\forall\alpha\in(\alpha_{-},\alpha_{+})$, $\phi^\star(\alpha)\in \mathcal{V}$; (3) $\alpha_{+}-\alpha_{-}>0$.  Since $\phi_k$ uniformly converges to $\phi^\star$, when $k$ is sufficiently large,
  \[\abs{\phi_k(\alpha)-\vec{x}_c}\leqslant\abs{\phi_k(\alpha)-\phi^\star(\alpha)}+\abs{\phi^\star(\alpha)-\vec{x}_c}<2\delta,\quad \alpha\in (\alpha_{-},\alpha_{+}).\]
  Denote $I=[\alpha_{-},\alpha_{+}]$, we have $\phi_k(\alpha)\in \mathcal{U}$ for any $\alpha\in I$ when $k$ is sufficiently large.

  The total variation of $\phi'_k$ in $I$ is
  \begin{equation}\label{eq:summation2}
    \bigvee\limits_{\alpha_{-}}^{\alpha_{+}}\phi_k'=\sup_{\Delta}\sum_{j=1}^{n}\abs{\phi'_k(\alpha_{j+1})-\phi'_k(\alpha_j)},
  \end{equation}
  where $\Delta$ is a partition of $I$ such that
  \[\alpha_{-}=\alpha_0<\alpha_1<\alpha_2\dots<\alpha_{n+1}=\alpha_{+}.\]

  If for any $k$, $E_k>U(\phi_k(\alpha))$ for $\alpha\in I$, then $\phi_k\in C^2(I)$. The total variation in this interval can be estimated by
  \[\bigvee\limits_{\alpha_{-}}^{\alpha_{+}}\phi'_k\leqslant\int_{\alpha_{-}}^{\alpha_{+}}\abs{\phi''_k}\d\alpha.\]
  By Lemma \ref{lem:nlunibounded}, it is uniformly bounded.

  If $U(\phi_k(\alpha))=E_k$ holds for some $\alpha$, by Lemma \ref{lem:finitjump}, these points are finite. As in Lemma \ref{lem:BV}, we denote the elements of $\{\alpha|U(\phi_k(\alpha))=E_k\}$ by $\alpha^k_1<\alpha^k_2<\dots<\alpha^k_N$, and divide the summation in \eqref{eq:summation2} into two cases. For Case I, we have
  \[\sum_{j\in\mathrm{I}}\abs{\phi'_k(\alpha_{j+1})-\phi'_k(\alpha_j)}\leqslant 2MN.\]
For Case II, we have
  \[\sum_{j\in\mathrm{II}}\abs{\phi'_k(\alpha_{j+1})-\phi'_k(\alpha_j)}\leqslant\sum_{l=1}^{N}\int_{\alpha^k_l}^{\alpha^k_{l+1}}\abs{\phi''_k}\d \alpha\leqslant\int_{0}^{T_k}\sqrt{\Theta_k(t)}\d t\]
by noting the fact $\phi_k\in C^2(\alpha^k_l,\alpha^k_{l+1})$. It is also uniformly bounded by Lemma \ref{lem:nlunibounded}.
\end{proof}

\textit{Proof of Proposition \ref{prop:condev}.}
  By Assumption \ref{asm:discret}, the number of critical points are finite. Without loss of generality, we assume that $\phi^\star$ passes through only one critical point $\vec{x}_c$ at $\alpha=\alpha_c$. We will show that every component of $\phi'_k$  is uniformly bounded and has uniformly bounded variation. Denote $\phi{'}_k^{(i)}$ the $i$th component of $\phi'_k$. $|\phi{'}_k^{(i)}|\leqslant\abs{\phi'_k}\leqslant M$ trivially holds.

  By Lemma \ref{lem:UBV}, there is an interval $I=(\alpha_{-},\alpha_{+})$ which contains $\alpha_c$ such that $\{\phi'_k\}$ has uniformly bounded variation in $I$. Outside the interval $I$, the function $E^\star - U(\phi^\star(\alpha))>2m>0$ for some positive constant $m$ since $E^\star=\max U(\vec{x})$. From the fact that $\phi_k\to \phi^\star$ uniformly, we obtain $E_k-U(\phi_k)>m$ for any $\alpha\in [0,1]\backslash(\alpha_{-},\alpha_{+})$ when $k$ is large enough. This implies $\phi_k\in C^2([0,1]\backslash I)$. The total variation of $\phi_k'$ on $[0,1]\backslash I$ can be bounded by
  \[\bigvee\limits_{[0,1]\backslash I}\phi'_k = \int_{[0,1]\backslash I}\abs{\phi_k''}\d\alpha=\int_{[0,1]\backslash I}\frac{\abs{\phi_k'}\sqrt{\abs{\nabla U}^2\abs{\phi_k'}^2-\inp{\nabla U}{\phi_k'}^2}}{2E_k-2U(\phi_k)}\d\alpha\leqslant\frac{M_UM^2}{m},\]
  where $M_U$ is the upper bound of $\abs{\nabla U(\phi_k)}$. This shows that the sequence $\{\phi_k\}$ has uniformly bounded variation.

With Helly's theorem, we can choose subsequence $\{\phi'_{k_l}\}$ such that $\phi'_{k_l}$ converges almost everywhere  to some  function $g$ with $\abs{g(\alpha)}\leqslant M$ and bounded variation. By applying the dominated convergence theorem to
  \[\phi_{k_l}(\alpha)=\vec{x}_s + \int_{0}^{\alpha}\phi_{k_l}'(\beta)\d\beta,\]
we get
  \[\phi^*(\alpha)=\lim_{l\to\infty}\phi_{k_l}(\alpha)=\vec{x}_s + \lim_{l\to\infty}\int_{0}^{\alpha}\phi_{k_l}'(\alpha)\d\alpha = \vec{x}_s +\int_{0}^{\alpha}g(\beta)\d\beta.\]
  So $g(\alpha)=\phi^{\star}{'}(\alpha)$ almost everywhere. The proof is done.
\endproof

\end{appendix}

\end{document}